\documentclass{article}
\usepackage{amsmath,amsthm,amssymb,amsfonts,bbm,color}
\usepackage{enumerate}
\usepackage{young}
\usepackage{graphicx}
\usepackage{epstopdf}
\usepackage{float}
\usepackage[mathscr]{eucal}
\usepackage{microtype}
\usepackage{hyperref}
\usepackage{pslatex}
\newtheorem{theorem}{Theorem}[section]
\newtheorem{corollary}[theorem]{Corollary}
\newtheorem{lemma}[theorem]{Lemma}
\newtheorem{remark}[theorem]{Remark}

\newtheorem{proposition}[theorem]{Proposition}
\newtheorem{fact}[theorem]{Fact}

\newcommand\bin{\operatorname{Bin}}
\newcommand\inv{\operatorname{Inv}}

\newcommand{\p}{\mathbb{P}}
\renewcommand{\P}{\mathbb{P}}
\newcommand{\E}{\mathbb{E}}

\newcommand\LIS{\operatorname{LIS}}
\newcommand\LDS{\operatorname{LDS}}
\newcommand\var{\operatorname{Var}}

\newcommand\txtred[1]{{\color{black}#1}}

\begin{document}

\title{Lengths of Monotone Subsequences in a Mallows Permutation}
\author{Nayantara Bhatnagar\thanks{Department of Mathematical Sciences, University of Delaware, USA, Email: {\tt naya@math.udel.edu}. Supported by NSF grant DMS-1261010.} \and Ron Peled\thanks{School of Mathematical Sciences, Tel Aviv University, Tel Aviv, Israel. E-mail: {\tt peledron@post.tau.ac.il}. Supported by an ISF grant and an IRG grant.}}

\maketitle

\begin{abstract}

We study the length of the longest increasing and longest decreasing
subsequences of random permutations drawn from the Mallows measure.
Under this measure, the probability of a permutation $\pi \in S_n$
is proportional to $q^{\inv(\pi)}$ where $q$ is a real parameter and
$\inv(\pi)$ is the number of inversions in $\pi$. The case $q=1$
corresponds to uniformly random permutations. The Mallows measure
was introduced by Mallows in connection with ranking problems in
statistics.

We determine the typical order of magnitude of the lengths of the
longest increasing and decreasing subsequences, as well as large
deviation bounds for them. We also provide a simple bound on the
variance of these lengths, and prove a law of large numbers for the
length of the longest increasing subsequence. Assuming without loss
of generality that $q<1$, our results apply when $q$ is a function
of $n$ satisfying $n(1-q) \to \infty$. The case that $n(1-q)=O(1)$
was considered previously by Mueller and Starr. In our parameter
range, the typical length of the longest increasing subsequence is
of order $n\sqrt{1-q}$, whereas the typical length of the longest
decreasing subsequence has four possible behaviors according to the
precise dependence of $n$ and $q$.

We show also that in the graphical representation of a
Mallows-distributed permutation, most points are found in a
symmetric strip around the diagonal whose width is of order
$1/(1-q)$. This suggests a connection between the longest increasing
subsequence in the Mallows model and the model of last passage
percolation in a strip.

\end{abstract}

\section{Introduction}
The length of the longest increasing subsequence of a uniformly
random permutation has attracted the attention of researchers from
several areas with significant  contributions from Hammersley
\cite{Ham72}, Logan and Shepp \cite{LogShe77} Vershik and Kerov
\cite{VerKer77}, Aldous and Diaconis \cite{AldDia99} and culminating
with the breakthrough work of Baik, Deift and Johansson
\cite{BaiDeiJoh99} who related this length to the theory of random
matrices and proved that it has a Tracy-Widom limiting distribution.
In this work we study the lengths of monotone subsequences
(increasing or decreasing) of a random permutation having a
different probability law, introduced by Mallows in \cite{Mal57} in order to study the statistical properties of non-uniformly random permutations (see also \cite{DiaRam00} and references therein for more background). The
Mallows distribution is parameterized by a number $q>0$, with the
probability of a permutation $\pi$ proportional to $q^{\inv(\pi)}$,
where $\inv(\pi)$ is the number of inversions in $\pi$, or pairs of
elements of $\pi$ which are out of order.

For $q>0$ and integer $n\ge 1$, the $(n,q)$-Mallows measure over
permutations in $S_n$ is given by
\begin{equation}\label{eq:mu_n_q_def}
\mu_{n,q}(\pi) := \frac{q^{\inv(\pi)}}{Z_{n,q}},
\end{equation}
where
\begin{equation*}
  \inv(\pi) := |\{(i,j)\,:\,\text{$i<j$ and $\pi(i)>\pi(j)$}\}|
\end{equation*}
denotes the number of inversions in $\pi$, and $Z_{n,q}$ is a
normalizing constant, given explicitly by the following well-known
formula \cite[pg.~21]{Sta97} (see also the remark after
Lemma~\ref{lem:Mallows_process_distribution} below)
\begin{equation}\label{eq:Z_formula}
Z_{n,q} = \displaystyle\prod_{i=1}^{n}\frac{1-q^{i}}{1-q}.
\end{equation}
Let $I=(i_1,\ldots,i_m)$ be an increasing sequence of indices. We
say $I$ is an {\em increasing subsequence} of a permutation $\pi$ if
$\pi(i_{k+1})>\pi(i_k)$ for $1\le k\le m-1$. Define a {\em
decreasing
  subsequence} analogously. Denote by $\LIS(\pi)$ the maximal length of an increasing
subsequence in $\pi$. That is,
\begin{equation*}
  \LIS(\pi) = \max\{m\,:\, \exists\, i_1<\cdots<i_m \text{
  satisfying }\pi(i_1)<\cdots<\pi(i_m)\}.
\end{equation*}
Analogously define $\LDS(\pi)$ to be the maximal length of a
decreasing subsequence in $\pi$. That is,
\begin{equation*}
  \LDS(\pi) = \max\{m\,:\, \exists\, i_1<\cdots<i_m \text{
  satisfying }\pi(i_1)>\cdots>\pi(i_m)\}.
\end{equation*}

Our goal is to investigate the distribution of $\LIS(\pi)$ and
$\LDS(\pi)$ when $\pi$ is randomly sampled from the Mallows measure.
We mention that the asymptotics of these lengths for other non-uniform distributions have been
considered in the literature previously. 
For instance, Baik and Rains \cite{BaiRai01} study the longest increasing and decreasing subsequences of random permutations satisfying certain symmetry conditions such as uniformly chosen involutions.
F\'{e}ray and M\'{e}liot
\cite{FerMel12} studied a
distribution similar to \eqref{eq:mu_n_q_def}, but with $\inv$
replaced by another permutation statistic, the {\em major index}. Fulman \cite{Ful02} relates the longest increasing subsequence in this major index distribution to the study of eigenvalues of random matrices over finite fields, analogously to the relation of the longest increasing subsequence of a uniform permutation with random Hermitian matrices. In
addition, $\LIS(\pi)$ and $\LDS(\pi)$ have been studied for the
Mallows distribution itself, by Mueller and Starr \cite{MueSta11}, as
detailed below.

We focus our investigations on the Mallows measure with $q<1$. This
restriction can be made without loss of generality since there is a
duality between the measures $\mu_{n,q}$ and $\mu_{n,1/q}$. Indeed,
if $\pi\sim\mu_{n,q}$ then its reversal $\pi^R$, defined by
$\pi^R(i) := \pi(n+1-i)$, is distributed as $\mu_{n,1/q}$ (see
Lemma~\ref{lem:reversal_inverse_Mallows} below). In particular,
$\LIS(\pi)$ is distributed as $\LDS(\pi^R)$. It is natural to allow
$q$ to be a function of $n$. Mueller and Starr \cite{MueSta11} studied
the regime where $n(1-q)$ tends to a finite limit $\beta$. They
showed that $\LIS(\pi) / \sqrt{n}$ converges in probability to
$\ell(\beta)$, where $\ell(\beta)$ is an explicitly given function
of $\beta$ satisfying $\ell(0)=2$ (see Theorem~\ref{thm:MS-wlln}
for the precise statement), thus extending the results of \cite{AldDia99,LogShe77,VerKer77}.
This implies an analogous result for $\LDS(\pi)$ by the
above-mentioned duality. Thus, in this limiting sense, in the regime
where $n(1-q)$ tends to a finite constant as $n$ tends to infinity,
$\LIS(\pi)$ and $\LDS(\pi)$ have the same order of magnitude as for
a uniformly random permutation, with a different leading constant.
In this paper we complete this picture by considering the case that
$n(1-q)$ tends to infinity with $n$. We find the typical order of
magnitude of $\LIS(\pi)$ and $\LDS(\pi)$ (which now differ from the
uniformly random case) and establish large deviation results for these
lengths and a law of large numbers for $\LIS(\pi)$. We also prove a
simple bound on the variance of $\LIS(\pi)$ and $\LDS(\pi)$.

Our first result concerns the displacement $|\pi(i)-i|$ of an
element in a random Mallows permutation. The result gives bounds on
the tails of this displacement. This theorem is not used later in
our analysis of monotone subsequences of random Mallows permutations
but it is useful in developing intuition for their behavior. \txtred{The
upper bound follows by methods of Braverman and Mossel \cite[Lemma
17]{BraMos09} as well as Gnedin and Olshanski \cite[Remark 5.2]{GneOls12}.
In \cite{GneOls12}, the authors studied a model of random
permutations of the infinite group of integers $\mathbb{Z}$ which is
obtained as a limit of the Mallows model, and obtained precise
formulas for the distribution of displacements in this limiting
model.}
%
%
\begin{figure}[t]
\centering
\includegraphics[width=5.5cm]{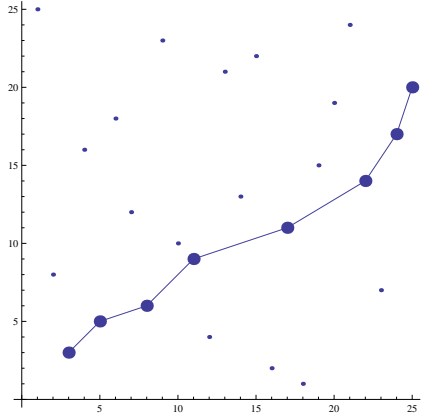}
\caption{An increasing piecewise linear curve corresponding to a
longest increasing subsequence.} \label{fig:mallows-points-25}
\end{figure}

\begin{figure}[t]
\center
\includegraphics[width=5cm]{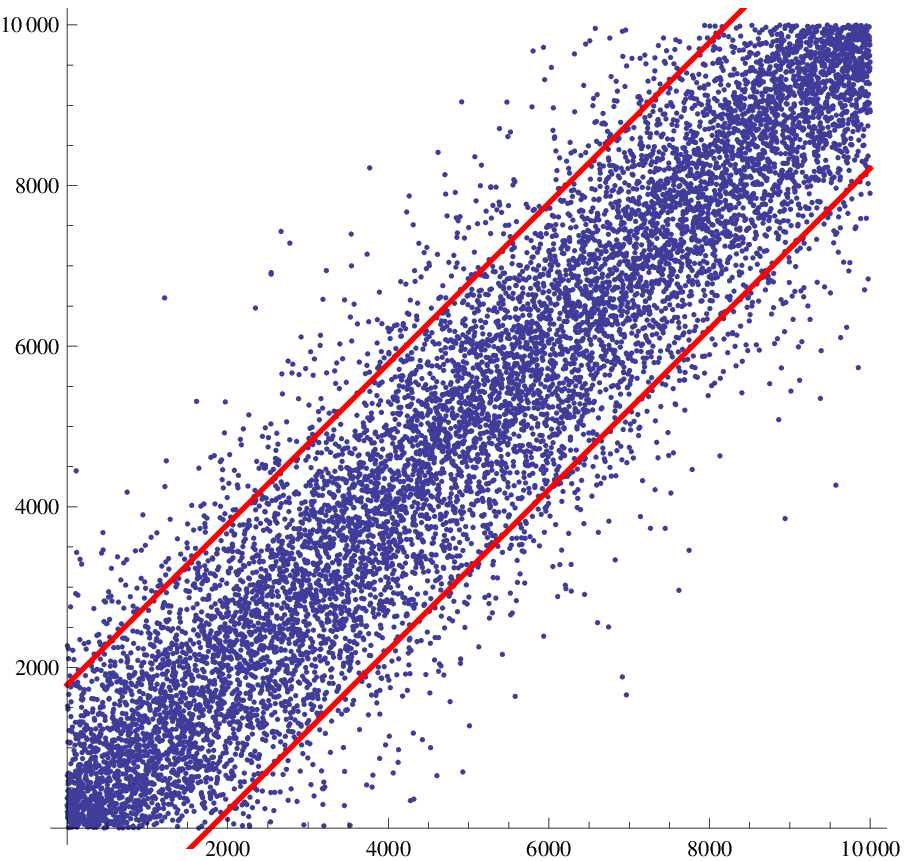}\hspace{0.5cm}
\includegraphics[width=5cm]{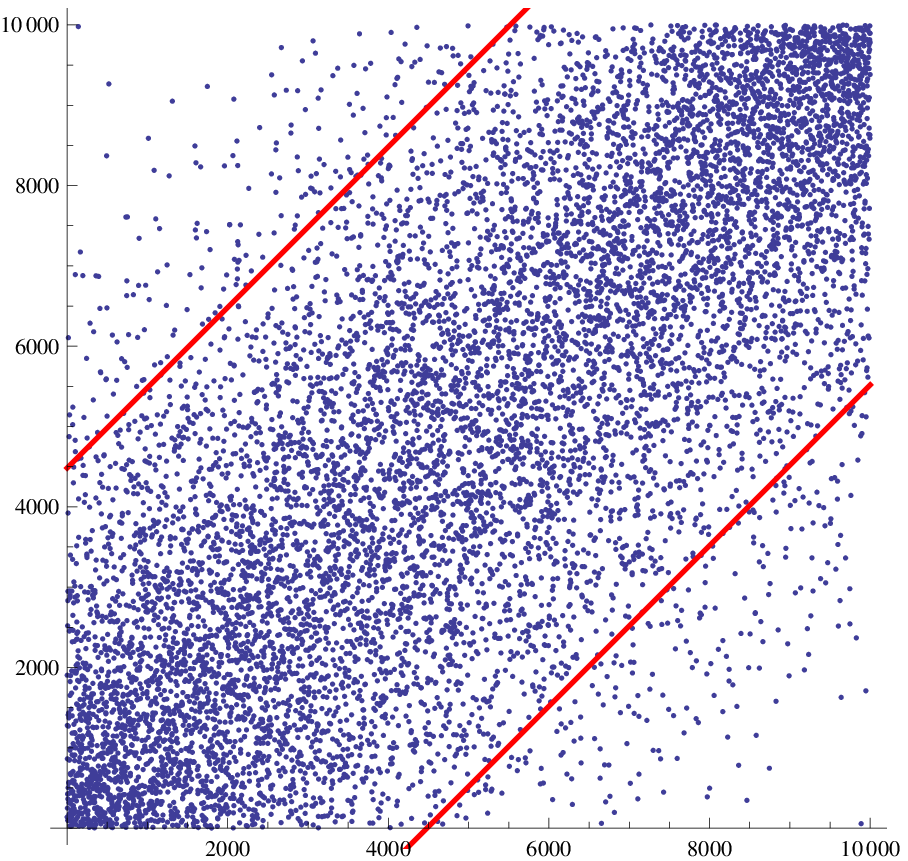}\hspace{0.5cm}
\includegraphics[width=5cm]{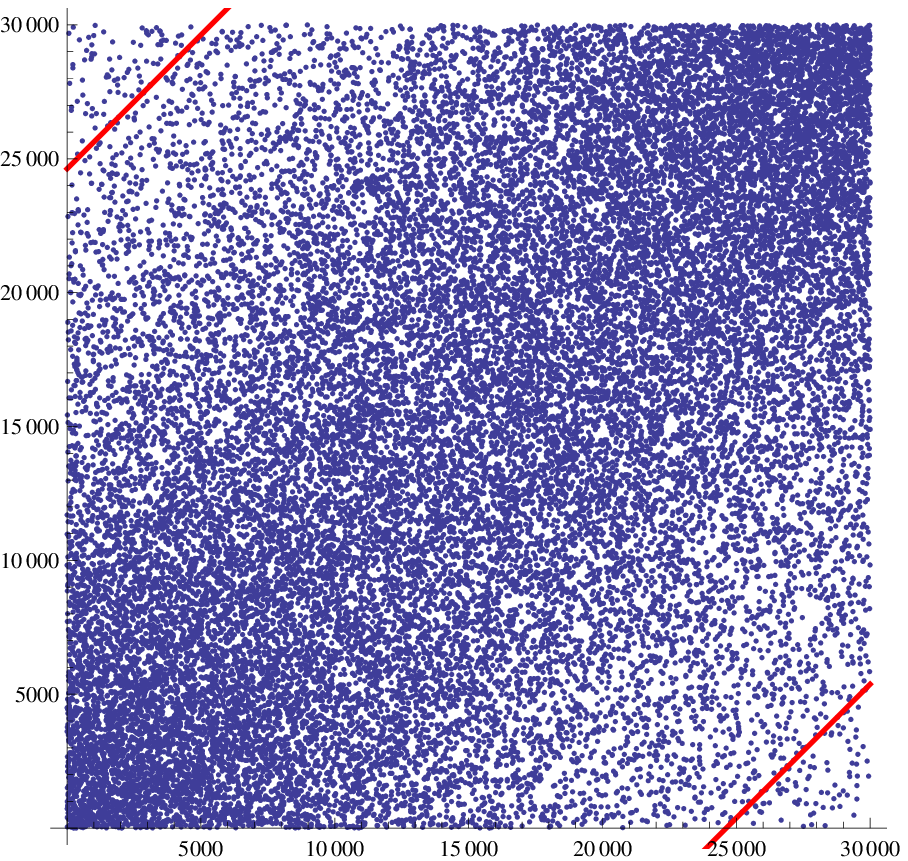}
\caption{Graphical representation for random Mallows-distributed
permutations with $1-q=n^{-0.7}, \ n^{-0.8}$
  and $n^{-0.88}$. The diagonal lines delineate a symmetric strip with width proportional to $\frac{1}{1-q}$. Theorem~\ref{thm:displacement} shows that most points of the permutation must lie in such a strip.}
\label{fig:mallows-points}
\end{figure}
\begin{theorem}\label{thm:displacement} For all $0<q<1$, and integer $n\ge 1$, $1\le i\le n$ and $t\ge 1$, if $\pi \sim
\mu_{n,q}$ then
\begin{equation}\label{eq:displacement_probability_bound}
  \P(|\pi(i)-i|\ge t)\le 2q^t,
\end{equation}
and
\begin{equation}\label{eq:required_displacement_bound}
  c \min \left( \frac{q}{1-q}, n-1\right) \le \E|\pi(i) - i| \le \min \left( \frac{2q}{1-q},n-1 \right)
\end{equation}
for some absolute constant $c>0$. In addition, if $n\ge 3$ and $1\le
t\le \frac{n+5}{8}$ then
  \begin{equation*}
    \P(|\pi(i) - i|\ge t)\ge\frac{1}{2}q^{2t-1}.
  \end{equation*}
\end{theorem}
A permutation $\pi$ in $S_n$ can be naturally associated to a
collection of $n$ points in the square $[1,n]^2$ by placing a point
at $(i,\pi(i))$ for each $i$. In this \emph{graphical
representation}, increasing subsequences correspond to increasing
curves passing through the points (see
Figure~\ref{fig:mallows-points-25}), and decreasing subsequences
correspond to decreasing curves. The graphical representation is
depicted in Figure~\ref{fig:mallows-points} for permutations
simulated from the Mallows distribution $\mu_{n,q}$ for various
choices of $n$ and $q$. The figure illustrates the fact that most
points of the permutation are displaced by less than a constant
times $q / (1-q)$, as Theorem~\ref{thm:displacement} proves.

The previous remark suggests a connection between the study of the
longest increasing subsequence of a random Mallows permutation, and the last passage
percolation model in a strip. In one version of the latter model,
one puts independent and identically distributed random points in a
strip, and studies the last passage time, which is the same as the
longest increasing subsequence when these points are taken to be the
graphical representation of a permutation. In
Section~\ref{sec:open-ques} we mention some works related to the
limiting distribution of the last passage time and raise the
question of whether the same limiting distributions arise also for
the Mallows model.

Our next results concern the typical order of magnitude of
$\LIS(\pi)$ when $\pi$ is sampled from the Mallows distribution. A
heuristic guess for this order of magnitude may be obtained from
Figure~\ref{fig:mallows-points-boxes}. 
Suppose that $\beta/(1-q)$ and $n(1-q)/\beta$ are integers for some large constant $\beta>0$. Consider $n(1-q)/\beta$ disjoint squares of side length $\beta/(1-q)$ along the strip delineated in the figure, such that the bottom left corner of each square equals the top right corner of the preceding square. The figure hints that the distribution of points in each square is close to a sample from the $\mu_{\beta/(1-q),q}$ distribution (here “close” should be interpreted as saying that the box contains a significant subsample of a Mallows distributed permutation of size $\beta/(1-q)$. Theorem \ref{thm:displacement} and the
results in Section~\ref{sec:Mallows_process_basic_properties} give rigorous meaning to such statements).
%
%
%
%
%
Thus the parameters fall in the regime
of \cite{MueSta11} and according to their results, the typical length of the longest increasing
subsequence in each square is of order $1 / \sqrt{1-q}$. We may thus
create an increasing subsequence with length of order $n\sqrt{1-q}$
by concatenating the longest increasing subsequences in each of the $n(1-q)/\beta$
squares. This reasoning gives rise to the prediction that
$\LIS(\pi)$ is about $C n\sqrt{1-q}$ for some constant $C>0$. The
next theorem establishes the correctness of this prediction, with a
precise constant $C=1$, in the limit
\eqref{eq:limit_for_weak_convergence}.

\begin{theorem}\label{thm:lp-convergence-lis}
Let $(q_n)$ be a sequence satisfying
\begin{equation}\label{eq:limit_for_weak_convergence}
  q_n\to 1\quad\text{and}\quad n(1-q_n)\to\infty
\end{equation}
as $n$ tends to infinity. Suppose $\pi_n\sim\mu_{n,q_n}$. Then
\begin{align*}
\frac{\LIS(\pi_n)}{n\sqrt{1-q_n}} \to 1
\end{align*}
as $n$ tends to infinity, where the convergence takes place in $L_p$
for any $0< p <\infty$.
\end{theorem}


\begin{figure}[t]
\centering
\includegraphics[width=5.5cm]{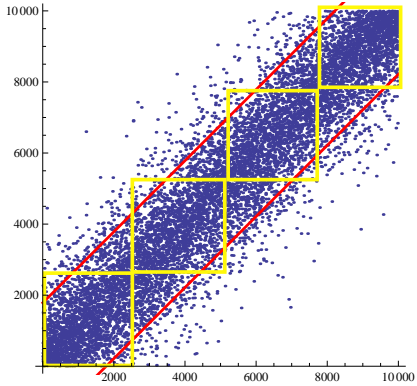}
\caption{Disjoint boxes with side length $\frac{\beta}{1-q}$
\txtred{along} a symmetric strip around the diagonal.}
\label{fig:mallows-points-boxes}
\end{figure}

In addition to this limiting behavior,
Theorem~\ref{thm:lis-largedev} below gives large deviation bounds on
the length of the longest increasing subsequence for fixed values of
$n$ and $q$. The proof of Theorem \ref{thm:lp-convergence-lis}
proceeds along the lines of the heuristic outlined above, combining
our large deviation results with the weak law of large numbers shown
in \cite{MueSta11}.\\


%

{\bf Notation:} We will write $a_{n,q} \approx b_{n,q}$ if there exist absolute constants $0 < c \le C < \infty$ such that $c b_{n,q} \le a_{n,q} \le C b_{n,q}$ for all $n$ and $q$ in a specified regime.


\begin{theorem}\label{thm:lis-largedev}
Suppose that $n\ge 1$, $\frac{1}{2} \le q \le 1-\frac{4}{n}$ and
$\pi\sim\mu_{n,q}$. Then,
\begin{align}\label{eq:E(LIS)}
\E(\LIS(\pi)) \approx n\sqrt{1-q}.
\end{align}
Furthermore, there exist absolute constants $0<C,c<\infty$ such that
\begin{enumerate}[(i)]
\item For integer $L\ge Cn\sqrt{1-q}$,
  \begin{align}\label{eq:LIS_greater_than_L}
    \left(\frac{c(1-q)n^2}{L^2} \right)^L \le \P(\LIS(\pi) \ge L)
    \le \left(\frac{C(1-q)n^2}{L^2} \right)^L.
\end{align}
\item \label{part:lis-less-than-ub} For integer $n(1-q) \le L \le cn\sqrt{1-q}$,
\begin{align}\label{eq:LIS_less_than_L_UB}
\P(\LIS(\pi) < L) \le \exp\left(-\frac{c(1-q)n^2}{L}\right).
\end{align}
\end{enumerate}
\end{theorem}
The bound \eqref{eq:LIS_less_than_L_UB} can be improved for certain
regimes of $n,q$ and $L$; for details see section~\ref{sec:erdos-szekeres}.
Complementing the regime of $q$ in \eqref{eq:E(LIS)}, we have the following simple bound
on $\E(\LIS(\pi))$, which is rather precise for small $q$.
\begin{proposition}\label{prop:small_q_LIS_exp_bound}
Suppose that $n\ge 1$, $0 < q \le 1$ and $\pi\sim\mu_{n,q}$. Then
\begin{equation*}
  n(1-q) \le \E(\LIS(\pi))\le n - \frac{q}{1+q}(n-1).
\end{equation*}
\end{proposition}
When $\pi$ is sampled uniformly from $S_n$, symmetry implies that
$\LIS(\pi)$ and $\LDS(\pi)$ have the same distribution.
For the Mallows measure, the analogous fact is not true. Indeed,
looking at Figure \ref{fig:mallows-points} one expects $\LDS(\pi)$
to be of a smaller order of magnitude than $\LIS(\pi)$ when
$\pi\sim\mu_{n,q}$ with $q<1$ since the overall trend of the points
is positive. Our next theorem establishes the order of magnitude for
$\LDS(\pi)$, confirming this expectation. Interestingly, we find as
many as four different behaviors for this order of magnitude
according to the relation between $n$ and $q$.

\begin{theorem}\label{thm:E(LDS)}
There exist constants $C_0, c_1>0$ such that the following is true.
Suppose that $n\ge 2$, $0 < q < 1$ and $\pi\sim\mu_{n,q}$.
  \begin{enumerate}[(i)]
    \item \begin{equation}\label{eq:E_LDS_large_q}
      \E(\LDS(\pi))\approx\begin{cases}
      \frac{1}{\sqrt{1-q}}&1-\frac{C_0}{(\log n)^2}\le q\le
      1-\frac{4}{n}\\[1em]
      \frac{\log n}{\log((1-q)(\log n)^2)}&1-\frac{c_1(\log\log n)^2}{\log n} \le q\le 1-\frac{C_0}{(\log
    n)^2}\\[1em]
    \sqrt{\frac{\log n}{\log\left(\frac{1}{q}\right)}}&\frac{1}{n}\le q\le 1-\frac{c_1(\log\log n)^2}{\log n}
      \end{cases}.
    \end{equation}
    \item If $0<q\le\frac{1}{n}$ then
      \begin{equation*}
         \E(\LDS(\pi))-1\approx nq.
      \end{equation*}
  \end{enumerate}
\end{theorem}
We pause briefly to give an informal reasoning for the results of
Theorem~\ref{thm:E(LDS)}. As explained before
Theorem~\ref{thm:lp-convergence-lis} above, one may again employ the
idea of placing $n(1-q)/\beta$ disjoint squares of side length $\beta/(1-q)$
along the diagonal as in Figure~\ref{fig:mallows-points-boxes}.
Since we expect the distribution of the points in each such square
to be close to that of the Mallows $\mu_{\beta/(1-q), q}$ measure, the
results of \cite{MueSta11} suggest that the typical order of magnitude
of the length of the longest decreasing subsequence in each square
is of order $1/\sqrt{1-q}$. When considering decreasing subsequences we
cannot concatenate the subsequences of disjoint squares, since the
overall trend of the points is positive. This heuristic suggests
that $\LDS(\pi)$ should have order of magnitude at least as large as
$1/\sqrt{1-q}$ and possibly not much larger. This is indeed the
order of magnitude obtained in the first regime of
Theorem~\ref{thm:E(LDS)}. However, as $q$ decreases a different
behavior takes over. Since we have $n(1-q)/\beta$ disjoint squares in
which to consider the longest decreasing subsequence, we may expect that one of these
squares exhibits \emph{atypical} behavior, with a decreasing subsequence of order which is
significantly longer than $1/\sqrt{1-q}$. The length of such an
atypical decreasing subsequence may be predicted rather accurately
using the large deviation results in Theorem~\ref{thm:lds-largedev}
below and it turns out to be indeed significantly longer than
$1/\sqrt{1-q}$ when $(\log n)^2(1-q)\to \infty$. This is what causes
the transition between the first two regimes in
Theorem~\ref{thm:E(LDS)}. A different strategy for obtaining a
decreasing subsequence should also be considered. Consider the
length of a longest decreasing subsequence composed solely of
\emph{consecutive} elements, i.e., the largest $m$ for which
$\pi(j)>\pi(j+1)>\ldots>\pi(j+m-1)$ for some $j$. The proof of
Theorem~\ref{thm:E(LDS)} shows that the length of such a decreasing
subsequence will have the same order of magnitude as the longest decreasing subsequence when
$q$ is so small that the typical longest decreasing subsequence is longer than $1/(1-q)$.
This is what governs the behavior in the third regime of the
parameters in the theorem as well as in part of the second regime.
Lastly, when $q\le\frac{1}{n}$, i.e., in the fourth regime of the
theorem, the probability that the random permutation differs from
the identity is of order $nq$ (see
Proposition~\ref{prop:identity_for_small_q} below). This is what
governs the behavior in the fourth regime of the theorem.

%

\begin{remark}\label{rem:LDS_limit}It seems likely that $\LDS(\pi)$ satisfies a law of
large numbers similar to the one in
Theorem~\ref{thm:lp-convergence-lis}. Indeed, if one formally takes
the limit $\beta\to-\infty$ in the results of $\cite{MueSta11}$ one
obtains that $\LDS(\pi)\sqrt{1-q}$ should tend to the constant
$\pi$. We expect this result to hold when $n(1-q)\to\infty$ and
$(\log n)^2(1-q)\to 0$, corresponding to the first regime in
\eqref{eq:E_LDS_large_q}, see also Section~\ref{sec:open-ques}.
\end{remark}
Analogously to Theorem \ref{thm:lis-largedev}, we obtain large
deviation estimates for $\LDS(\pi)$ holding for fixed $n$ and $q$.
\begin{theorem}\label{thm:lds-largedev}
There exist constants $C,c>0$ such that the following is true. Let
$n\ge 2$, $0 < q < 1$ and $\pi\sim\mu_{n,q}$.
\begin{enumerate}[(i)]
\item  \label{part:lds-greater-than-ub} If $0<q<1-\frac{2}{n}$ then for integer $L\ge
2$,
\begin{align}\label{eq:LDS_greater_than_L}
\P(\LDS(\pi) \ge L) \le
  n^8\begin{cases}\left(\frac{C}{(1-q)L^2}\right)^L&
  L\le\frac{3}{1-q}\\(C(1-q))^L
  q^{\frac{L(L-1)}{2}}&L>\frac{3}{1-q}\end{cases}.
\end{align}
Moreover, if  $\; 0<q<\frac{1}{2}$ then for integer $L\ge 2$,
\begin{equation}\label{eq:LDS_greater_than_L_refined}
  \P(\LDS(\pi)\ge L)\le nC^L q^{\frac{L(L-1)}{2}}.
\end{equation} \label{part:LDS_large_dev_upperbound}
\item \label{part:LDS_greater_than_L_lb} For integer $L$,
\begin{align}\label{eq:LDS_greater_than_L_lb}
\P(\LDS(\pi) \ge L) \ge
\begin{cases}
1- \left(1-\left(\frac{c}{(1-q)L^2}\right)^L\right)^{\lfloor \frac{n(1-q)}{4} \rfloor} &
\text{if }\frac{C}{\sqrt{1-q}}\le L \le \frac{1}{1-q}\;\;\text{ and }\; \; \frac{1}{2} \le q \le 1-\frac{4}{n}\\
 1-\left(1-q^{\frac{L(L-1)}{2}} (1-q)^L\right)^{\lfloor\frac{n}{L}\rfloor} & \text{for any $L\ge 2$}
\end{cases}.
\end{align}
\item \label{part:lds-less-than-ub} Let $\frac{1}{2} \le q \le 1-\frac{4}{n}$. For integer $2\le L<\frac{c}{\sqrt{1-q}}$,
\begin{align}\label{eq:LDS_less_than_L}
\p(\LDS(\pi) < L) \le (C(1-q)L^2)^{\frac{n}{L}}.
 \end{align}
\end{enumerate}
\end{theorem}

The discussion above focused on the typical order of magnitude and
large deviations of $\LIS(\pi)$ and $\LDS(\pi)$ when $\pi$ is
distributed according to the Mallows distribution. Also interesting,
and seemingly more difficult, is the study of the typical deviations
of $\LIS(\pi)$ and $\LDS(\pi)$ from their expected value. In this
paper we make only a modest contribution towards understanding these
quantities, as given in the following proposition. We denote by
$\var(X)$ the variance of $X$.
\begin{proposition}\label{prop:var_prop}
  Let $n\ge 1, 0<q<\infty$ and $\pi\sim\mu_{n,q}$. Then
  \begin{equation*}
    \var(\LIS(\pi))\le n-1.
  \end{equation*}
  Furthermore, for all $t>0$,
  \begin{equation*}
    \P(|\LIS(\pi)-\E(\LIS(\pi))|> t\sqrt{n-1})< 2e^{-t^2/2}.
  \end{equation*}
\end{proposition}
We note that the proposition applies equally well to the
distribution of $\LDS(\pi)$ since it applies to arbitrary $q$ and,
as noted above, the reversal of $\pi$ is distributed as
$\mu_{n,1/q}$, and satisfies that $\LIS(\pi) = \LDS(\pi^R)$. We
expect that when $n$ tends to infinity with $0<q<1$ fixed then
$\var(\LIS(\pi))$ will indeed be of order $n$. However, if $q$
increases to $1$ as $n$ tends to infinity then we expect the
variance to be of smaller order, 
see the discussion in
Section~\ref{sec:open-ques}.

We finish the description of our main results with a simple
proposition which is useful for very small $q$. It shows that when
$nq$ is much smaller than $1$, the Mallows distribution is
concentrated on the identity permutation.
\begin{proposition}\label{prop:identity_for_small_q}
  Suppose $n\ge 2$, $0<q\le\frac{1}{n}$ and $\pi\sim\mu_{n,q}$. Then
  \begin{equation*}
    \P(\pi\text{ is not the identity})\approx nq.
  \end{equation*}
\end{proposition}

{\bf Policy on constants:} In what follows, $C$ and $c$ denote
positive numerical constants (independent of all other parameters)
whose value can change each time they occur (even inside the same
calculation), with the value of $C$ increasing and the value of $c$
decreasing. In contrast, the value of numbered constants, such as
$C_0$ or $c_0$, is fixed and will not change between occurrences.

%

\subsection{Techniques}\label{sec:techniques}
Previous work on the asymptotics of the longest increasing
subsequence followed two main approaches: either through analysis of
combinatorial asymptotics or by the probabilistic analysis of
systems of interacting particle processes. The combinatorial
approach to the longest increasing subsequence makes use of a bijection between permutations
and Young tableaux known as the Robinson-Schensted-Knuth (RSK)
correspondence \cite{Rob38,Sch61,Knu70}.  This bijection is intimately
related to the representation theory of the symmetric group
\cite{Jam78,Dia88}, the theory of symmetric functions \cite{Sta99}, and
the theory of partitions \cite{And76}. The uniform measure on
permutations induces the Plancherel measure on Young diagrams under
the RSK correspondence. Vershik and Kerov and Logan and Shepp
independently showed a limiting shape for diagrams under the
Plancherel measure and proved that
\begin{equation}\label{eq:unif_perm_E_LIS}
  \E(\LIS(\pi)) = 2\sqrt{n}+o(\sqrt{n})\quad\text{when $\pi$ is uniformly
  distributed}.
\end{equation}
This approach was extended much later in the groundbreaking work of
Baik, Deift and Johansson \cite{BaiDeiJoh99} who determined completely
the limiting distribution and fluctuations of the longest increasing
subsequence of a uniformly distributed permutation.

The second approach has been through the framework of interacting
particle processes. Hammersley \cite{Ham72} investigated ``Ulam's
problem'' of finding the constant in the expected length of the
longest monotone subsequence in a uniformly random permutation.
Implicit in this work was a certain one-dimensional interacting
particle process which Aldous and Diaconis \cite{AldDia99} call
Hammersley's process. Aldous and Diaconis gave hydrodynamical
limiting arguments for Hammersley's process to obtain an independent
proof of the result \eqref{eq:unif_perm_E_LIS}. This approach led to
other generalizations, such as the work of Deuschel and Zeitouni
\cite{DeuZei95} who found the leading behavior of $\E(\LIS(\pi))$ when
$\pi$ is a random permutation whose graphical representation is
obtained by putting independent and identically distributed points
in the plane.


Mueller and Starr \cite{MueSta11} were the first to consider the longest
increasing subsequence of a random Mallows permutation. Their work
focuses on the regime of parameters where $n(1-q) \to
\beta\in(-\infty,\infty)$ as $n\to\infty$. In this regime Starr
\cite{Sta09} developed a Botzmann-Gibbs formulation of the Mallows
measure and found a limiting density for the graphical
representation of the random permutation. Mueller and Starr relied
on this limiting density and applied similar techniques to those of
Deuschel and Zeitouni \cite{DeuZei95} to find the leading behavior of
$\E(\LIS(\pi))$.


Our analysis uses a third approach. In his paper, Mallows \cite{Mal57}
describes an iterative procedure for generating a
Mallows-distributed permutation. This procedure, which we term the
\emph{Mallows process}, is defined formally in
Section~\ref{sec:Mallows_process}. Informally, it may be described
as follows: A set of $n$ folders is put in a random order into a
drawer using the rule that each new folder is inserted at a random
position, pushing back all the folders behind it. The probability that
the $i$th folder is inserted at position $j$, for $1\le j\le i$, is
proportional to $q^{j-1}$, independently of all other folders. It is not
hard to check that after all $n$ folders have been placed in the
drawer, their positions have the $(n,1/q)$-Mallows distribution. Our
analysis consists of tracking the dynamics of the increasing and
decreasing subsequences throughout the evolution of this process.

%

\subsection{Reader's guide}
The remainder of the paper is organized as follows. In Section
\ref{sec:Mallows_process} we define the Mallows process formally and
derive some useful properties of the Mallows measure from it. In
Section \ref{sec:displacement} we bound the displacement of elements
in a random Mallows permutation, proving Theorem
\ref{thm:displacement}. Section \ref{sec:lis} is devoted to the
study of $\LIS(\pi)$. We establish there the large deviation bounds
for $\LIS(\pi)$ and determine its typical order of magnitude,
proving Theorem~\ref{thm:lis-largedev} and Proposition
\ref{prop:small_q_LIS_exp_bound}. In Section \ref{sec:wlln} we prove
the law of large numbers for $\LIS(\pi)$, establishing Theorem
\ref{thm:lp-convergence-lis}. In Section \ref{sec:lds} we study
$\LDS(\pi)$, establishing large deviation bounds for it and
determining its typical order of magnitude, proving
Theorem~\ref{thm:E(LDS)}, Theorem~\ref{thm:lds-largedev} and
Proposition~\ref{prop:identity_for_small_q}. In Section
\ref{sec:lis-var} we prove Proposition \ref{prop:var_prop}, giving a
simple bound on the variance of $\LIS(\pi)$ and showing a Gaussian
tail inequality. Finally, we end with some directions for further
research in Section \ref{sec:open-ques}.

\subsection{Acknowledgements}
We thank Elchanan Mossel for pointing out to us that the technique of proof
in \cite[Lemma 17]{BraMos09} simplifies our previous proof of the upper bound in
\eqref{eq:required_displacement_bound} and yields also the deviation
bound \eqref{eq:displacement_probability_bound}.

\section{The Mallows process}\label{sec:Mallows_process}

In this section we describe a random evolution process on
permutations, which we term the Mallows process. This process is
central to our later analysis of the length of monotone
subsequences. The process was known to Mallows \cite{Mal57}, and was
also used by Gnedin and Olshanski \cite{GneOls10,GneOls12} to study
variants and extensions of the Mallows measure to infinite groups of
permutations. The underlying idea is also useful in the analysis of
the number of inversions of a uniformly random permutation, e.g., as
in Feller \cite[Chap. X.6]{Fel68}.

Let $q>0$. The \emph{$q$-Mallows process} is a permutation-valued
stochastic process $(p_n)_{n \ge 1}$, where each $p_n \in S_n$. The
process is initialized by setting $p_1$ to be the (only) permutation
on one element. The process iteratively constructs $p_{n}$ from
$p_{n-1}$ and an independent random variable $p_n(n)$ distributed as
a truncated geometric. Precisely, letting $(p_n(n))$ be a sequence
of independent random variables with the distributions
\begin{equation}\label{eq:Mallows_folder_dist}
 \p(p_n(n) = j) :=  \frac{q^{j-1}}{1+q+\cdots+q^{n-1}} =
 \frac{(1-q)q^{j-1}}{1-q^n}
  \quad(1\le j\le n),
\end{equation}
each permutation $p_n$ is defined by
\begin{equation}\label{eq:Mallows_iteration_procedure}
  p_{n}(i) =
\begin{cases}
p_{n-1}(i)& p_{n-1}(i) < p_n(n)\\
p_{n-1}(i)+1 & p_{n-1}(i) \ge p_n(n)
\end{cases} \quad (1 \le i \le n-1).
\end{equation}
Alluding to our intuitive description in
Section~\ref{sec:techniques}, we may think of $p_n(i)$ as denoting
the {\em position} of the $i$th folder at time $n$ in the drawer. It
is clear by construction that $p_n$ is a permutation in $S_n$. Also,
note that for each $i$ and $n \ge i$, $p_n(i)$ is non-decreasing in
$n$. Below is an example to illustrate the process. For example, we
see that in the second step $n=2$, since the position of the second
folder is $1$, the position of the first folder becomes $2$. In
general, in step $n$, the position of a folder increases by $1$ if
its position in step $n-1$ is at or after the position where the
$n$th folder is inserted and otherwise it stays the same. We also
note the process $(p_n^{-1})$ which may be thought of as the
contents of the drawer at time $n$, in the intuitive description of
Section~\ref{sec:techniques}.
\[
\begin{array}{ccll}
n & p_n(n) & p_n & (p_n^{-1})\\
 1 & 1  & 1 & 1\\
 2   & 1 & 21 & 21\\
 3 & 2 & 312 &23 1\\
 4 & 4 & 3124  & 2314\\
 5 & 2 & 41352  & 25314\\
6  & 3 & 514623  & 256314\\
\end{array}
\]

\begin{lemma}\label{lem:Mallows_process_distribution}
Let $q>0$ and let $(p_n)_{n \ge 1}$ be the $q$-Mallows process. Then $p_n$ is distributed according to the Mallows distribution with parameter $1/q$.
\end{lemma}
\begin{proof}
The claim is trivial for $n=1$. Assume by induction that for any
$\sigma_n \in S_n$, $\p(p_n = \sigma_n) \propto q^{-\inv(\sigma_n)}$
and let us prove the same for $n+1$. Fix a permutation $\sigma_{n+1}
\in
S_{n+1}$. 
For $1 \le i  \le n$, define a permutation $\sigma_n\in S_n$ by
\begin{align*}
\sigma_n(i) :=
\begin{cases}
\sigma_{n+1}(i) -1 & \text{ if } \sigma_{n+1}(i) > \sigma_{n+1}(n+1)\\
\sigma_{n+1}(i) & \text{ if } \sigma_{n+1}(i) < \sigma_{n+1}(n+1)\\
\end{cases}
\end{align*}
It follows from the definition of the Mallows process that $p_{n+1}
= \sigma_{n+1}$ if and only if $p_{n+1}(n+1)=\sigma_{n+1}(n+1)$ and
$p_n = \sigma_n$. Noting that $\inv(\sigma_{n+1}) = \inv(\sigma_n) +
n+1-\sigma_{n+1}(n+1)$,
the induction hypothesis implies that
\begin{align*}
 \p(p_{n+1}  = \sigma_{n+1}) & = \p(p_n = \sigma_n)\cdot\p(p_{n+1}(n+1) = \sigma_{n+1}(n+1)) \\
 & = \frac{q^{-\inv(\sigma_n)}}{Z_{n,1/q}} \cdot \frac{q^{\sigma_{n+1}(n+1)-1}}{1+q+\cdots+q^n}  \\
 & = \frac{q^{-\inv(\sigma_n)}}{Z_{n,1/q}} \cdot \frac{(1/q)^{n-\sigma_{n+1}(n+1)+1}}{1+(1/q)+\cdots+(1/q)^n} \propto
 q^{-\inv(\sigma_{n+1})}.\qedhere
\end{align*}
\end{proof}
As a by-product, the above recursion also shows that the formula
\eqref{eq:Z_formula} for the normalizing constant holds. Recall that
$\pi^R$, the reversal of a permutation $\pi$, is defined by
$\pi^R(i) = \pi(n+1-i)$.
\begin{lemma}\label{lem:reversal_inverse_Mallows}
For any $n\ge 1$ and $q>0$, if $\pi\sim\mu_{n,q}$ then
$\pi^R\sim\mu_{n,1/q}$ and $\pi^{-1}\sim\mu_{n,q}$.
\end{lemma}
\begin{proof}
  The lemma is immediate upon noting that both taking reversal and
  taking inverse are bijections on $S_n$, and that $\inv(\pi^R) = {n \choose 2} -
  \inv(\pi)$ and $\inv(\pi^{-1}) = \inv(\pi)$.
\end{proof}
This lemma allows us to define four different permutations related
to the $q$-Mallows process, all having the Mallows distribution
$\mu_{n,q}$.
\begin{corollary}\label{cor:perms-that-are-mallows}
Let $q>0$ and let $(p_n)_{n\ge 1}$ be the $q$-Mallows process. Then
each of the following permutations is distributed as $\mu_{n,q}$.
\begin{enumerate}[(i)]
\item $\pi:=p_n^R$. That is, $\pi(i) = p_n(n+1-i)$.
\item $\pi:=(p_n^R)^{-1}$. That is, $\pi(i)=n+1 - p_n^{-1}(i)$.
\item $\pi:=(p_n^{-1})^R$. That is, $\pi(i)=p_n^{-1}(n+1-i)$.
\item $\pi:=((p_n^{-1})^R)^{-1}$. That is, $\pi(i) =
n+1-p_n(i)$.
\end{enumerate}
\end{corollary}
%
%
%
%
This corollary will be useful in the sequel, allowing us to
prove results about the Mallows distribution by choosing from the
above list a convenient coupling of the Mallows distribution and the
Mallows process.

\subsection{Basic properties of the Mallows
process}\label{sec:Mallows_process_basic_properties}

In this section we let $q$ be an arbitrary positive number and let
$(p_n)$ be the $q$-Mallows process. Let $I=(i_1,\ldots,i_k)$ be an
increasing sequence of indices and let $\pi$ be any permutation. Let
$\pi_I \in S_{k}$ denote the induced relative ordering of $\pi$
restricted to $I$. That is, $\pi_I(j)>\mathscr \pi_I(k)$ if and only
if $\pi(i_j)>\pi(i_k)$. The following fact is clear from the
definition of the Mallows process.
\begin{fact}\label{fact:indep-of-orderings-on-blocks}
Let $I=(i_1,\ldots,i_k)$ be an increasing sequence and let $n\ge
i_k$. Then $(p_n)_I$ is a function only of
$p_{i_1}(i_1),p_{i_1+1}(i_1+1),\ldots,p_{i_k-1}(i_k
-1),p_{i_k}(i_k)$. In other words, $(p_n)_I$ is independent of the
set of $(p_i)_i$, $i<i_1$ or $i>i_k$.
\end{fact}

\begin{lemma} \label{lem:indep-seqs}(Independence of induced
  orderings) Let $I=(i_1,\ldots,i_k)$ and $I'= (i'_1,\ldots,i'_\ell)$
  be two increasing sequences such that $i_k < i_1'$. Let $\pi \sim \mu_{n,q}$ for $n\ge i'_\ell$. Then, $\pi_I$ and $\pi_{I'}$ are independent.
\end{lemma}

\begin{proof}
Using Corollary~\ref{cor:perms-that-are-mallows}, we couple $\pi$
with $(p_n)$ so that $\pi(i) = n+1 - p_n(i)$ for all $i$.
By the definition of the Mallows process, the variables $(p_i(i))$
are independent. By Fact \ref{fact:indep-of-orderings-on-blocks},
$\pi_I$ and $\pi_{I'}$ are functions of independent variables and
are therefore independent.
\end{proof}

For a sequence of indices $I=(i_1, \ldots, i_m)$ and an integer $b$,
define the sequence $I+b:=(i_1+b, \ldots, i_m+b)$.
\begin{lemma}\label{lem:shift}(Translation invariance)
Let $I=(i_1, \ldots, i_k)$ be an increasing sequence and let
$\pi\sim \mu_{n,q}$. Then, for any integer $1 \le b \le n-i_k$,
$\pi_I$ and $\pi_{I+b}$ have the same distribution. That is, for any
$\omega \in S_k$,
\[
\p(\pi_I = \omega) = \p(\pi_{I+b} = \omega).
\]
\end{lemma}
\begin{proof}
Observe that we can make the following simplifying assumptions.
First, we may assume that $b=1$ since then the claim follows by
applying the result $b$ times. Second, under the assumption $b=1$,
$I$ is contained in $(1,2,\ldots, n-1)$ and hence we may deduce the
lemma with the given $I$ from the lemma with $I=(1,2,\ldots, n-1)$.
%
%

Assume then that $b=1$ and $I=(1,2,\ldots, n-1)$.
It is straightforward to see that there exists a unique bijection
$T$ from $S_n$ to itself which preserves the number of inversions
(and hence the Mallows distribution), such that $(T(\pi))_{I+1} =
\pi_I$. This establishes the lemma.
\end{proof}

It is simple to check that the above fact is not necessarily true
for sequences which are not translates. Suppose $\pi\sim\mu_{3,q}$.
By explicit
calculation, 
\begin{equation*}
\p(\pi(2)>\pi(1)) =
\frac{1+q+q^2}{Z_{3,q}}\quad\text{whereas}\quad\p(\pi(3)>\pi(1)) =
\frac{1+2q}{Z_{3,q}},
\end{equation*}

so that the probabilities are different for all $q\neq 1$.

One corollary of translation invariance is that the permutation
induced on any sequence of consecutive elements is distributed like a shorter
Mallows permutation.

\begin{corollary}\label{cor:induced_permutation_Mallows}
Let $I=(i,i+1,\ldots,i+m-1)\subseteq[n]$ be a sequence of
consecutive elements. If $\pi\sim\mu_{n,q}$ then
$\pi_I\sim\mu_{m,q}$.
\end{corollary}

\begin{proof}
Since $q$ is arbitrary, it suffices to prove the corollary with
$\pi$ replaced by $p_n$, so that $q$ is replaced by $1/q$. For
$i=1$, the claim follows simply by the definition of the Mallows
process. That is, since $\pi_{I} = p_m \sim \mu_{m,1/q}$.  For
$i>1$, the claim follows by the translation invariance given by
Lemma \ref{lem:shift}.
\end{proof}

\txtred{
\begin{remark} One can also construct a Mallows permutation indexed by the infinite sets $\mathbb N$ or $\mathbb Z$ \cite{GneOls10,GneOls12}. A version of Corollary 2.7 would still be valid in this case, yielding the finite Mallows distribution as an induced permutation of the infinite one. The infinite permutation has the advantage that it is constructed out of a sequence of i.i.d. geometric random variables rather than just independent truncated geometric variables as in the finite construction. However, the fact that the geometric random variables are unbounded complicates some aspects of our proofs and in this paper we chose to work only in the finite setting.
\end{remark}
}

\section{The Displacement of an element in a Mallows permutation}\label{sec:displacement}
In this section we prove Theorem~\ref{thm:displacement}. Our proof
of the upper bounds follows that of \cite[Lemma 17]{BraMos09}, with
slightly more precise estimates.

Fix $0<q<1$. Recall the $q$-Mallows process $(p_i)$ from
Section~\ref{sec:Mallows_process}, defined for all $i\ge 1$. We
first prove the upper bounds in the theorem. Fix $n\ge 1$ and
consider the permutation $\pi$ defined by $\pi(i):=n+1-p_n(i)$,
which by Corollary~\ref{cor:perms-that-are-mallows} is distributed
according to $\mu_{n,q}$. Note first that for all $1\le i\le n$,
  \begin{equation*}
    \pi(i) - i = n+1 - p_n(i) - i =
    n-i-p_n(i)+p_i(i) - (p_i(i) - 1).
  \end{equation*}
  Thus, since $p_i(i)\ge 1$ and $p_n(i) - p_i(i) \le n-i$, we have
  \begin{equation}\label{eq:pi_lower_displacement_bound}
    |\pi(i) - i|{\mathbbm 1}_{(\pi(i)-i<0)}\le p_i(i)-1\quad\text{ for }1\le
    i\le n.
  \end{equation}
Similarly, let $\pi'$ be defined by $\pi'(i):=p_n(n+1-i)$, so that
$\pi'\sim\mu_{n,q}$ by
  Corollary~\ref{cor:perms-that-are-mallows}. For all $1\le i\le n$,
  \begin{equation*}
    \pi'(n+1-i) - (n+1-i) = p_n(i) - (n+1-i) =
    -(n-i-p_n(i)+p_i(i)) + (p_i(i) - 1).
  \end{equation*}
  Thus, again since $p_i(i)\ge 1$ and $p_n(i) - p_i(i) \le n-i$, we have
  \begin{equation*}
    |\pi'(n+1-i) - (n+1-i)|{\mathbbm 1}_{(\pi'(n+1-i) - (n+1-i)>0)}\le
    p_i(i)-1,
  \end{equation*}
  and exchanging the roles of $i$ and $n+1-i$ we obtain
  \begin{equation}\label{eq:pi'_upper_displacement_bound}
    |\pi'(i) - i|{\mathbbm 1}_{(\pi'(i) - i>0)}\le
    p_{n+1-i}(n+1-i)-1\quad\text{ for } 1\le i\le n.
  \end{equation}
  Putting together \eqref{eq:pi_lower_displacement_bound} and
  \eqref{eq:pi'_upper_displacement_bound}, and recalling that $\pi,
  \pi'\sim\mu_{n,q}$ we conclude that for all $1\le i\le n$
  and integer $t\ge 1$,
  \begin{equation}\label{eq:pi_displacement_probability_first_bound}
  \P(|\pi(i)-i|\ge t)= \P(\pi(i)-i\ge t)+\P(\pi(i)-i\le -t)\le
  \P(p_{n+1-i}(n+1-i)\ge t+1) + \P(p_i(i)\ge t+1).
\end{equation}
Now recall from \eqref{eq:Mallows_folder_dist} that $p_j(j)$ has the
distribution of a geometric random variable with parameter $1-q$,
conditioned to be at most $j$. In particular, $p_j(j)$ is
stochastically dominated by this geometric random variable and thus
\begin{equation}\label{eq:p_i_i_bound}
  \P(p_j(j)\ge t+1)\le q^t\quad\text{ for } 1\le
  j\le n\text{ and integer }t\ge 1.
\end{equation}
Putting together \eqref{eq:pi_displacement_probability_first_bound}
and \eqref{eq:p_i_i_bound} yields
\eqref{eq:displacement_probability_bound}. Thus, the upper bound of
\eqref{eq:required_displacement_bound} follows since $|\pi(i)-i| \le n-1$ and
\begin{equation*}
  \E|\pi(i)-i|=\sum_{t=1}^\infty \P(|\pi(i)-i|\ge t)\le \sum_{t=1}^\infty
  2q^t=\frac{2q}{1-q}.
\end{equation*}

Next we derive a lower bound on the displacement. This is done in
the next three claims. We start by observing a monotonicity property
of the Mallows process. Let
\begin{equation*}
A = \{(a_1,a_2,\ldots)\,:\, a_j\in\{1,\ldots, j\}\}.
\end{equation*}
By definition of the Mallows process, for each $n$, the permutation
$p_n$ is a function of the vector $(p_1(1), \ldots, p_n(n))$, whose
elements satisfy $p_j(j)\in\{1,\ldots, j\}$. For $a\in A$, denote by
$p_n^a$ the permutation $p_n$ resulting from taking $p_j(j) = a_j$.
\begin{lemma}\label{lem:monotonicity_Mallows_process}
  For each $n\ge 1$ and $1\le j\le n$, $p_n^a(j)$ is increasing in $a_j$. That is, if $a, a'\in A$
  satisfy $a_k = a'_k$ for all $k\neq j$ and $a_j>a'_j$ then
  $p_n^a(j)>p_n^{a'}(j)$.
\end{lemma}
\begin{proof}
  Fix $n, j, a,a'$ as in the lemma. Trivially
  $p_j^a(j) > p_j^{a'}(j)$. Hence it suffices to observe by induction that for
  $k\ge j$,
\begin{equation*}
p_{k+1}^a(j) = p_k^a(j) + \mathbbm 1_{(a_{k+1} \le p_k^a(j))} =
p_k^a(j) + \mathbbm 1_{(a'_{k+1} \le p_k^a(j))} > p_k^{a'}(j) +
\mathbbm 1_{(a'_{k+1} \le p_k^{a'}(j))} = p_{k+1}^{a'}(j).\qedhere
\end{equation*}
\end{proof}
\begin{lemma}\label{lem:displacement_lower_bound_max}
  For all integer $n\ge 1, 1\le i\le n$ and $t\ge 1$, if $\pi\sim\mu_{n,q}$
  then
  \begin{equation*}
    \P(|\pi(i) - i| \ge t)\ge \max(\P(p_i(i)\ge 2t),\,
    \P(p_{n+1-i}(n+1-i)\ge 2t)).
  \end{equation*}
\end{lemma}
\begin{proof}
  Fix $n,i$ and $t$ as in the lemma. Couple $\pi$ with the Mallows
  process so that $\pi(j)=n+1-p_n(j)$ as in Corollary~\ref{cor:perms-that-are-mallows}. Condition on $(p_j(j))$ for
  $j\neq i$ and observe that under this conditioning, the value of
  $p_n(i)$, and hence the value of $\pi(i)$, is a function of $p_i(i)$. By
  Lemma~\ref{lem:monotonicity_Mallows_process}, under the
  conditioning, there are at most $2t-1$ (contiguous) values of $p_i(i)$ for
  which $|\pi(i) - i|<t$. Since the $(p_j(j))$ are independent and $\P(p_i(i)=s)$ is a decreasing
  function of $s$, it follows that
  \begin{equation*}
    \P(|\pi(i) - i|\ge t) = \E\left[ \P(|\pi(i) - i|\ge t\, |\,
    (p_j(j))_{j\neq i})\right] \ge \E\left[ \P(p_i(i)\ge 2t\,|\, (p_j(j))_{j\neq
    i})\right] = \P(p_i(i)\ge 2t).
  \end{equation*}
  The proof of the bound $\P(|\pi(i) - i|\ge t)\ge
  \P(p_{n+1-i}(n+1-i)\ge 2t)$ is analogous by using the coupling
  $\pi(j)=p_n(n+1-j)$ of Corollary~\ref{cor:perms-that-are-mallows}
  and applying Lemma~\ref{lem:monotonicity_Mallows_process} with
  $j=n+1-i$.
\end{proof}
\begin{corollary}\label{cor:displacement_lower_bound}
  For all integer $n\ge 3, 1\le i\le n$ and $1\le t\le \frac{n+5}{8}$, if $\pi\sim\mu_{n,q}$
  then
  \begin{equation*}
    \P(|\pi(i) - i|\ge t)\ge\frac{1}{2}q^{2t-1}.
  \end{equation*}
\end{corollary}
\begin{proof}
  Let $j = \max(i, n+1-i)$. Observe that $j\ge \frac{n+1}{2}$. Note
  also that our assumptions imply that $2t\le
  \frac{n+1}{2}\le j$. By
  Lemma~\ref{lem:displacement_lower_bound_max} and
  \eqref{eq:Mallows_folder_dist},
  \begin{equation*}
    \P(|\pi(i) - i|\ge t) \ge \P(p_j(j)\ge 2t) = \frac{1-q^{j-2t+1}}{1-q^j}
    q^{2t-1}.
  \end{equation*}
  Our assumptions imply that $t\le \frac{n+5}{8}\le \frac{j+2}{4}$
  and thus $j-2t+1\ge \frac{j}{2}$. Hence we conclude that
  \begin{equation*}
    \P(|\pi(i) - i|\ge t) \ge \frac{1-q^{j/2}}{1-q^j}q^{2t-1} =
    \frac{q^{2t-1}}{1+q^{j/2}} \ge \frac{1}{2} q^{2t-1}.\qedhere
  \end{equation*}
\end{proof}
Finally, we fix $n\ge 2, 1\le i\le n$ and prove a lower bound for
$\E|\pi(i) - i|$. We consider separately three cases. If $n\ge 3$
and $q<1-\frac{1}{n}$ then by
Corollary~\ref{cor:displacement_lower_bound},
\begin{equation*}
  \E|\pi(i) - i| \ge \sum_{t=1}^{\lfloor \frac{n+5}{8} \rfloor}
  \P(|\pi(i) - i|\ge t)\ge \frac{1}{2} \sum_{t=1}^{\lfloor \frac{n+5}{8}
  \rfloor} q^{2t-1} = \frac{q (1-q^{2\lfloor (n+5)/8\rfloor})}{2
  (1-q^2)} \ge c\frac{q}{1-q}
\end{equation*}
for some absolute constant $c>0$. If $n\ge 3$ and $q\ge
1-\frac{1}{n}$ then, similarly, by
Corollary~\ref{cor:displacement_lower_bound},
\begin{equation*}
  \E|\pi(i) - i| \ge \sum_{t=1}^{\lfloor \frac{n+5}{8} \rfloor}
  \P(|\pi(i) - i|\ge t)\ge \frac{1}{2} \sum_{t=1}^{\lfloor \frac{n+5}{8}
  \rfloor} q^{2t-1} \ge \frac{1}{2} \sum_{t=1}^{\lfloor \frac{n+5}{8}
  \rfloor} \left(1-\frac{1}{n}\right)^{2t-1}\ge c n
\end{equation*}
for some absolute constant $c>0$. Finally, if $n=2$ then by
Lemma~\ref{lem:displacement_lower_bound_max},
\begin{equation*}
  \E|\pi(i) - i| = \P(|\pi(i) - i|\ge 1) \ge \P(p_2(2)\ge 2) =
  \frac{q}{1+q}\ge \frac{q}{2}.
\end{equation*}
Thus in all cases we have shown that $\E|\pi(i) - i|\ge
c\min\left(\frac{q}{1-q}, n-1\right)$, as required.

\section{Increasing subsequences}\label{sec:lis}

Our goal in this section is to establish Theorem
\ref{thm:lis-largedev} and Proposition
\ref{prop:small_q_LIS_exp_bound}. We begin in Section
\ref{sec:lis_lb} with the lower bound in
\eqref{eq:LIS_greater_than_L} and the bound
\eqref{eq:LIS_less_than_L_UB}. In Section \ref{sec:upper_bound_LIS}
we use a union bound argument to show that the probability of a very
long increasing subsequence cannot be too large and establish the
upper bound in \eqref{eq:LIS_greater_than_L}. In the same section we
complete the proof of Theorem~\ref{thm:lis-largedev} and
Proposition~\ref{prop:small_q_LIS_exp_bound} by applying the
previous results to estimate the expectation of $\LIS(\pi)$. Lastly,
a result extending our tail bounds for $\LIS(\pi)$ is proved at the
end of Section~\ref{sec:upper_bound_LIS}. This result is used in the
arguments of Section~\ref{sec:wlln}.


\subsection{Lower bounds on the probability of a long increasing subsequence}\label{sec:lis_lb}
In this section we will show a lower bound on the probability that
there is a long increasing subsequence, proving the lower bound of
\eqref{eq:LIS_greater_than_L} and the bound
\eqref{eq:LIS_less_than_L_UB} in Theorem \ref{thm:lis-largedev}.
The proof proceeds by defining a sequence of stopping times for the
Mallows process at which elements are added to an increasing
subsequence. We show that the waiting time to build a long
increasing subsequence in this way is not too large with high
probability.
\subsubsection{Large deviation bounds for binomial random variables}
The next proposition collects some standard results on binomial
random variables which will be used in the sequel.
\begin{proposition} \label{prop:binom-largedev}
Suppose $n\ge 1$, $0<p<1$ and let $S\sim\bin(n,p)$.
\begin{enumerate}
  \item For all $t>0$,
\begin{equation*}
  \P(S - np < -t) < \exp\left(-\frac{t^2}{2np}\right).
\end{equation*}
In particular,
\begin{equation}\label{eq:bin_ld_half_exp}
  \P\left(S<\frac{1}{2}np\right) \le \exp\left(-\frac{1}{8}np\right).
\end{equation}
\item If $p<\frac{1}{2}$ then for all integer $np\le t\le n$,
\begin{align}\label{eq:bin_ld_lower_bound_for_upper_tail}
\P(S \ge t)  \ge  \left(\frac{np}{et}\right)^t.
\end{align}
\end{enumerate}
\end{proposition}
\begin{proof}
  The first part is proved, for instance, in \cite[Theorem
A.1.13]{AloSpe08}. For the second part, observe first that
\begin{equation*}
 \P(S\ge t) \ge   \binom{n}{t} p^t (1-p)^{n-t} \ge \left(\frac{np}{t}\right)^{t}
 (1-p)^{n-t}.
\end{equation*}
Now note that $\log(1-p)\ge -p-p^2$ for $0\le p\le 1/2$. Thus, using
that $t\ge np$ in the third inequality,
\begin{equation*}
\P(S\ge t) \ge \left(\frac{np}{t}\right)^t e^{-(n-t)(p+p^2)}\ge
\left(\frac{np}{t}\right)^t e^{-np+p(t-np)}\ge
\left(\frac{np}{t}\right)^t e^{-t}.\qedhere
\end{equation*}
\end{proof}

\subsubsection{Lower bounds for $\P(\LIS(\pi)\ge L)$}

Fix $n\ge 1$ and $\frac 12 \le q \le 1-\frac 4n$. Let $(p_m)$ be the
$q$-Mallows process, and define, for $m\ge 1$, $\pi_m := (p_m)^R$ so
that $\pi_m\sim\mu_{m,q}$ by Corollary
\ref{cor:perms-that-are-mallows}.
Fix an integer $1\le L\le n$ and consider the following strategy for
finding an increasing subsequence in $\pi_n$. Let
\begin{equation*}
  W := \left[\frac{1}{1-q},\, \frac{1}{1-q} + \frac{n}{1000L} + 1\right]
  \cap \mathbb{Z}
\end{equation*}
and set $T_0:=\max(W)$. Consider the minimal time $S_1>T_0$ for
which $p_{S_1}(S_1)\in W$, and consider the first subsequent time
$T_1>S_1$ for which $p_{T_1}(S_1)\notin W$. Then repeat the process
and find the next subsequent time $S_2>T_1$ for which
$p_{S_2}(S_2)\in W$, and so on. Formally, with $T_0=\max(W)$, we
inductively define the stopping times for $i\ge 1$ as follows:
\begin{align*}
  S_i &:= \min \{t > T_{i-1} \ : \ p_t(t) \in W \},\\
  T_i &:= \min \{t > S_i \ : \ p_t(S_i)\notin W\}.
\end{align*}
We claim that for $k\ge 1$ and $m\ge S_k$, the sequence
$(\pi_m(S_1),\ldots, \pi_m(S_k))$ is increasing. This is equivalent
to the sequence $(p_m(S_1),\ldots, p_m(S_k))$ being decreasing. To
see this note that, by definition of the Mallows process, the
relative order of $p_m(S_i)$ and $p_m(S_{i+1})$ is the same as for
$p_{S_{i+1}}(S_i)$ and $p_{S_{i+1}}(S_{i+1})$. Now observe that the
definition of the stopping times above implies that
$p_{S_{i+1}}(S_i)>\max W\ge p_{S_{i+1}}(S_{i+1})$. We conclude that
if $m\ge S_k$ then $\LIS(\pi_m)\ge k$.
Thus we arrive at
\begin{equation}\label{eq:LIS_S_L_relation}
  \P(\LIS(\pi_n)\ge L) \ge \P(S_L\le n).
\end{equation}
In the rest of the section we focus on estimating the right-hand
side of the above inequality in two regimes of $n,L$ and $q$. We
start by describing a common part to both regimes. We always take
\begin{equation}\label{eq:cond_q_L}
  \frac{1}{2} \le q\le 1 - \frac{4}{n}\quad\text{and}\quad L\ge
  n(1-q)
\end{equation}
and observe that this implies that
\begin{equation}\label{eq:max_W_bound}
  \max(W) \le \frac{2}{1-q} \le \frac{n}{2}.
\end{equation}
Thus, by \eqref{eq:Mallows_folder_dist}, for any $i>\max(W)$ and any
$1\le j\le \max(W)+1$,
\begin{equation*}
  \P(p_i(i)=j) = \frac{(1-q)q^{j-1}}{1-q^i} \ge (1-q)q^{j-1} \ge
  \frac{(1-q)}{16}
  =: c_1(1-q).
\end{equation*}
The second inequality follows from the bound $q \ge 1/2$ once we note that for $x \le 1/2$, $(1-x)^{1/x} \ge 1/4$. In particular, if $i>\max(W)$ then
\begin{align}
  &\P(p_i(i)\in W) \ge c_1(1-q)|W|\ge
  \frac{c_1(1-q)n}{1000L} =: \frac{c_2(1-q)n}{L},\label{eq:good_assignment_bound}\\
  &\P(p_i(i)\le \min(W)) \ge c_1(1-q)\min(W)\ge c_1.\label{eq:drift_assignment_bound}
\end{align}
Next, we note the simple decomposition
\begin{equation*}\label{eq:S_L_decomp}
S_{L} = T_0 + \sum_{i=1}^L S_i - T_{i-1} + \sum_{i=1}^{L-1} T_{i} -
S_{i}.
\end{equation*}
Since $T_0\le \frac{n}{2}$ by the definition of $T_0$ and
\eqref{eq:max_W_bound}, we may plug this decomposition into
\eqref{eq:LIS_S_L_relation} to obtain
\begin{align}\label{eq:first_second_regime_bound}
  \P(\LIS(\pi_n)\ge L)\ge \P\left(\sum_{i=1}^L S_i - T_{i-1}
  \le\frac{n}{4},\ \ \sum_{i=1}^{L-1} T_{i} -S_{i}
  \le\frac{n}{4}\right).
\end{align}
We aim to bound the right-hand side by a product of two terms.

First, we note explicitly the following simple facts which follow
from the definition of the Mallows process and our definition of the
stopping times $(T_i)$ and $(S_i)$:
\begin{enumerate}
  \item For each $k\ge 0$, $|\{T_k < i \le S_{k+1}\,:\, p_i(i)\in
  W\}|= 1$.
  \item For each $k\ge 1$, $|\{S_k < i \le T_{k}\,:\, p_i(i)\le \min W\}|\le |W|$.
\end{enumerate}

Second, we let $(U_j)$ and $(V_j)$, $j\ge 1$, be two independent
sequences of independent Bernoulli random variables satisfying
\begin{equation*}
  \P(U_j = 1) = \frac{c_2(1-q)n}{L}\quad\text{and}\quad\P(V_j = 1) =
  c_1.
\end{equation*}
Third, we couple $((U_j),(V_j))$ with the Mallows process $(p_m)$ as
follows.
If $T_k< i\le S_{k+1}$ for some $k\ge 0$ then we consider the next
``unused'' $U_j$, i.e.,
\begin{equation*}
  j = |\{i'\,:\, i'<i,\ T_k< i'\le S_{k+1}\text{ for some $k\ge 0$}\}| +
  1,
\end{equation*}
and couple $U_j$ to $p_i(i)$ in a way that if $U_j=1$ then
$p_i(i)\in W$. Such a coupling is possible due to the bound
\eqref{eq:good_assignment_bound} and the fact that the event $T_k<
i\le S_{k+1}$ is determined solely by $(p_j(j))$ for $j<i$.
Similarly, if $S_k< i \le T_k$ for some $k\ge 1$ then we consider
the next ``unused'' $V_j$, i.e.,
\begin{equation*}
  j = |\{i'\,:\, i'<i,\ S_k< i'\le T_{k}\text{ for some $k\ge 1$}\}| +
  1,
\end{equation*}
and couple $V_j$ and $p_i(i)$ in a way that if $V_j=1$ then
$p_i(i)\le \min(W)$. Again, this is possible due to the bound
\eqref{eq:drift_assignment_bound} and the fact that the event $S_k<
i \le T_k$ is determined solely by $(p_j(j))$ for $j<i$.

The coupling, together with the two enumerated facts above, yields
the following containment of events,
\begin{align*}
  \left\{\sum_{1\le j\le n/4} U_j \ge L\right\} &\subseteq \Big\{\sum_{i=1}^L S_i - T_{i-1}
  \le\frac{n}{4}\Big\},\\
  \left\{\sum_{1\le j\le n/4} V_j \ge (L-1)|W|\right\} &\subseteq \Big\{\sum_{i=1}^{L-1} T_{i} -S_{i}
  \le\frac{n}{4}\Big\}.
\end{align*}
Finally, defining
\begin{align*}
  B:= \sum_{1\le j\le n/4} U_j&\sim
  \bin\left(\left\lfloor\frac{n}{4}\right\rfloor, \frac{c_2(1-q)n}{L}\right),\\
  B':= \sum_{1\le j\le n/4} V_j &\sim
  \bin\left(\left\lfloor\frac{n}{4}\right\rfloor, c_1\right),
\end{align*}
we may continue \eqref{eq:first_second_regime_bound} and write
\begin{equation}\label{eq:LIS_B_B'_rel}
  \P(\LIS(\pi_n)\ge L)\ge \P(B\ge L)\, \P(B'\ge(L-1)|W|).
\end{equation}
We observe for later use that the restriction on $q$ in
\eqref{eq:cond_q_L} implies that $n\ge 8$ and hence $\lfloor
\frac{n}{4}\rfloor\ge \frac{n}{8}$. The analysis now splits
according to two regimes of the parameters.\\

{\bf First regime of the parameters:} Suppose in addition to
\eqref{eq:cond_q_L} that
\begin{equation}\label{eq:first_regime_assumption}
  L\le cn\sqrt{1-q}
\end{equation}
for some small absolute constant $c>0$. This implies that $\E(B) \ge
\frac{c_2(1-q)n^2}{8L}\ge 2L$, and it follows by
\eqref{eq:bin_ld_half_exp} that
\begin{equation}\label{eq:S_i_T_i_estimate}
  \P(B < L) \le e^{-\frac{c(1-q)n^2}{L}}.
\end{equation}
Moreover, recalling that $c_1=\frac{1}{16}$ and $(L-1)|W|\le
(L-1)(2+n/(1000L))\le 2L + n/1000 \le n/500$ if the constant in
\eqref{eq:first_regime_assumption} is sufficiently small, we have
$\E(B') \ge \frac{c_1 n}{8} \ge 2(L-1)|W|$. Using
\eqref{eq:bin_ld_half_exp} again, we have the bound
\begin{equation}\label{eq:T_i_S_i_estimate}
  \P(B' < (L-1)|W|) \le e^{-cn}.
\end{equation}
Putting together \eqref{eq:LIS_B_B'_rel},
\eqref{eq:S_i_T_i_estimate} and \eqref{eq:T_i_S_i_estimate} we
obtain
\begin{equation*}
  \P(\LIS(\pi_n)< L) \le e^{-\frac{c(1-q)n^2}{L}} + e^{-cn} \le e^{-\frac{c(1-q)n^2}{L}}
\end{equation*}
under the assumptions \eqref{eq:cond_q_L} and
\eqref{eq:first_regime_assumption}. This establishes
\eqref{eq:LIS_less_than_L_UB}. \\

{\bf Second regime of the parameters:} Now suppose, in addition to
\eqref{eq:cond_q_L} and instead of
\eqref{eq:first_regime_assumption}, that
\begin{equation}\label{eq:second_regime_assumption}
  L\ge Cn\sqrt{1-q}
\end{equation}
for some large absolute constant $C>0$. This implies, in particular,
that $L \ge \E(B)$.
It follows by \eqref{eq:bin_ld_lower_bound_for_upper_tail} that
\begin{equation}\label{eq:second_regime_B_estimate}
  \P(B\ge L) \ge \left(\frac{c(1-q)n^2}{L^2}\right)^L.
\end{equation}
Let us now make an additional assumption, which will imply that
$\E(B') \ge 2(L-1)|W|$. Since $(L-1)|W|\le n/1000 + 2L$, it suffices
to assume (recalling that $c_1=\frac{1}{16}$, $\lfloor
\frac{n}{4}\rfloor \ge \frac{n}{8}$ and hence $\E(B')\ge
\frac{n}{128}$) that
\begin{equation}\label{eq:second_regime_extra_assumption}
  L\le \frac{1}{2}\left(\frac{c_1}{16} - \frac{1}{1000}\right)n.
\end{equation}
Under this assumption, by \eqref{eq:bin_ld_half_exp},
\begin{equation}\label{eq:second_regime_B'_estimate}
  \P(B'\ge (L-1)|W|) \ge \P\left(B'\ge \frac{\E(B')}{2}\right) \ge 1
  - \exp\left(\frac{1}{8} \E(B')\right)\ge \frac{1}{2},
\end{equation}
where we have used the fact $\E(B')\ge 8$ which follows from our
assumptions \eqref{eq:cond_q_L}, \eqref{eq:second_regime_assumption}
and \eqref{eq:second_regime_extra_assumption}.
Putting together \eqref{eq:LIS_B_B'_rel},
\eqref{eq:second_regime_B_estimate} and
\eqref{eq:second_regime_B'_estimate} we have proven that
\begin{equation}\label{eq:second_regime_final_bound}
  \P(\LIS(\pi_n)\ge L)\ge
  \frac{1}{2}\left(\frac{c(1-q)n^2}{L^2}\right)^L \ge \left(\frac{c(1-q)n^2}{L^2}\right)^L
\end{equation}
under the assumptions \eqref{eq:cond_q_L},
\eqref{eq:second_regime_assumption} and
\eqref{eq:second_regime_extra_assumption}. To remove the extra
assumption \eqref{eq:second_regime_extra_assumption}, we note that
for any $k$ we have the trivial bound
\begin{equation*}
  \P(\LIS(\pi_k)=k) = Z_{k,q}^{-1} = (1-q)^k\prod_{i=1}^{k}
  (1-q^{i})^{-1} \ge (1-q)^k
\end{equation*}
by \eqref{eq:mu_n_q_def} and \eqref{eq:Z_formula}. Thus, using
Fact \ref{fact:indep-of-orderings-on-blocks}, for any $1\le L\le n$ we have
\begin{equation}\label{eq:second_regime_final_bound2}
\P(\LIS(\pi_n)\ge L) \ge \P(\LIS(\pi_L) = L) \ge (1-q)^L,
\end{equation}
establishing the bound \eqref{eq:second_regime_final_bound} (with a
different constant $c$) when the assumption assumption
\eqref{eq:second_regime_extra_assumption} is violated. Putting
together \eqref{eq:second_regime_final_bound} and
\eqref{eq:second_regime_final_bound2} establishes the lower bound in
\eqref{eq:LIS_greater_than_L}.

\subsection{Upper bound on the probability of a long increasing subsequence}\label{sec:upper_bound_LIS}
In this section we establish the remaining results of
Theorem~\ref{thm:lis-largedev}. In
Section~\ref{sec:upper_bound_prob_long_LIS} we estimate the
probability that the longest increasing subsequence of a random
Mallows permutation is exceptionally long and establish the upper
bound in \eqref{eq:LIS_greater_than_L}. The expected length of the
longest increasing subsequence is then estimated in
Section~\ref{sec:bounds_E_LIS}. Lastly, a result extending our tail
bounds for $\LIS(\pi)$ is proved at the end of
Section~\ref{sec:generalized_bound_on_LIS}. This result is used in
the arguments of Section~\ref{sec:wlln}.

\subsubsection{Very long increasing
subsequences are unlikely}\label{sec:upper_bound_prob_long_LIS}

In this section we establish the upper bound in
\eqref{eq:LIS_greater_than_L} of Theorem \ref{thm:lis-largedev}. In
fact, we prove the following slightly stronger result.
\begin{proposition}\label{prop:lis-ub}
Let $n\ge 1$, $0 < q \le 1- \frac{2}{n}$ and $\pi\sim\mu_{n,q}$, then,
\[
\p(\LIS(\pi)\ge L) \le \left(\frac{C(1-q)n^2}{L^2}\right)^L
\]
for all integer $L \ge Cn\sqrt{1-q}$.
\end{proposition}
The idea of the proof is to bound the probability that a fixed
subsequence is increasing and then apply a union bound over all
possible long increasing subsequences. For the remainder of this
section, assume $\pi \sim \mu_{n,q}$ for some fixed $n$ and $q$
satisfying the conditions of the proposition. Using
Corollary~\ref{cor:perms-that-are-mallows}, we couple $\pi$ with the
$q$-Mallows process $(p_m)$ so that
\begin{equation}\label{eq:pi_p_coupling_LIS_UB}
  \pi(i) = n+1 - p_n(i)\quad\text{ for all $1\le i\le n$}.
\end{equation}
For an increasing sequence of integers $I = (i_1,\ldots, i_m)$ and a
sequence of integers $J = (j_1, \ldots, j_m)$ satisfying that $1\le
j_k\le i_k$, define the event
\begin{equation}\label{eq:E_I_J_event}
E_{I,J}:=\{p_{i_k}(i_k)=j_k\text{ for all $1\le k\le m$}\}.
\end{equation}
Additionally, for an increasing sequence of integers $I =
(i_1,\ldots, i_m)\subseteq[n]$, define the event that $I$ is a set
of indices of an increasing subsequence,
\begin{equation}\label{eq:E_I_event}
E_{I} := \{\pi(i_{k+1})>\pi(i_k)\text{ for all $1\le k\le m-1$}\}.
\end{equation}
In the next lemma and proposition we estimate the probabilities of these
events.

\begin{lemma}\label{lem:prob-bound-i-j}
Let $m\ge 1$. Let $I=(i_1,\ldots,i_m)$ be an increasing sequence of
integers satisfying $i_1 \ge 1/(1-q)$, and let $J=(j_1,\ldots,j_m)$
be a sequence of integers satisfying $1\le j_k \le i_k$. Then
\[
\p(E_{I,J}) \le \left(C(1-q)\right)^m.
\]
\end{lemma}
\begin{proof} By \eqref{eq:Mallows_folder_dist},
\begin{equation*}
\p(p_{i_k}(i_k)=j_k\text{ for all $1\le k\le m$}) =
\displaystyle\prod_{1 \le k \le
  m}\frac{(1-q)q^{j_{k}-1}}{1-q^{i_k}} \le \left(C(1-q)\right)^m. \qedhere
\end{equation*} 
\end{proof}

%

\begin{proposition}\label{prop:prob-bound-i}
Let $1\le m\le n$ and let $I=(i_1, \ldots, i_m)\subseteq[n]$ be an
increasing sequence of integers. Then
\begin{equation*}
\p(E_I) \le \left(\frac{Cn(1-q)}{m} \right)^m.
\end{equation*}
\end{proposition}
\begin{proof}
Fix a sequence $I$ as in the proposition. Let $\mathcal{J}$ be the
set of all integer sequences $J=(j_1,\ldots,j_m)$ satisfying $1\le
j_k\le i_k$ for $1\le k\le m$ and satisfying that the event $E_I\cap
E_{I,J}$ is non-empty. Observe that by
\eqref{eq:Mallows_iteration_procedure}, the Mallows process
satisfies for every $1\le k\le m-1$ that
\begin{align*}
&p_{i_{k+1}}(i_k)\le p_{i_k}(i_k) + i_{k+1} - i_k\quad\text{ and}\\
&p_n(i_{k+1})< p_n(i_k)\text{ if and only if }
p_{i_{k+1}}(i_{k+1})<p_{i_{k+1}}(i_k).
\end{align*}
Thus the coupling
\eqref{eq:pi_p_coupling_LIS_UB} implies that in order that
$J\in\mathcal{J}$ it is necessary that
\begin{align}\label{eq:j_k_increasing_nec_cond}
j_{k+1}-j_k \le i_{k+1}-i_k\quad\text{ for all $1
\le k \le m-1$}.  
\end{align}
We conclude that if $J\in\mathcal{J}$, then the transformed sequence
$(\ell_1,\ldots, \ell_m)$ defined by $\ell_k:=j_k-i_k-k$ satisfies
\begin{align*}
&1-2n \le \ell_k \le -1\quad\text{ for all $1 \le k \le m$,\, and}  \\
&\ell_{k+1}<\ell_k\quad\text{ for all $1 \le k \le m-1$}.  
\end{align*}
Since the above transformation is one-to-one, it follows that
\begin{equation}\label{eq:J_family_estimate}
  |\mathcal{J}|\le {2n \choose m}.
\end{equation}

We proceed to establish the proposition by considering separately
several cases. Suppose first that $i_1 \ge 1/(1-q)$. Combining
Lemma~\ref{lem:prob-bound-i-j} and the bound
\eqref{eq:J_family_estimate}, we obtain that
\begin{align*}
\p(E_I) = \displaystyle\sum_{J\in\mathcal{J}}\p(E_I
  \cap E_{I,J}) \le \displaystyle\sum_{J\in\mathcal{J}}\p(E_{I,J}) \le |\mathcal{J}| \left(C(1-q)\right)^m \le
\left(\frac{Cn(1-q)}{m} \right)^m. 
\end{align*}
This establishes the proposition for the case that $i_1\ge 1/(1-q)$.

Now suppose that $i_m<1/(1-q)$. Observe that by the assumptions on
$q$ in Proposition~\ref{prop:lis-ub}, we have $1/(1-q)\le n/2$.
Thus, the translated sequence $I + \lceil n/2\rceil$ is contained in
$[1/1-q, n]$. Applying the translation invariance
Lemma~\ref{lem:shift}, the case that $i_m<1/(1-q)$ reduces to the
case that $i_1\ge 1/(1-q)$ and we conclude that the proposition
holds for such $I$ as well.

Finally, suppose that $i_1<1/(1-q)$ and $i_m\ge 1/(1-q)$. Let $1 \le
k \le m-1$ be such that $I_1:=(i_1,\ldots,i_k) \subseteq
[0,1/(1-q))$ and $I_2 := (i_{k+1},\ldots,i_m) \subseteq
[1/(1-q),n]$. By the independence of induced orderings Lemma
\ref{lem:indep-seqs}, we may apply the proposition to each of $I_1$
and $I_2$ to obtain
\begin{align}\label{eq:comparison_for_x_to_x}
\p(E_I) & \le \p(E_{I_1}\cap E_{I_2}) = \p(E_{I_1})\cdot
\p(E_{I_2})\le \frac{\left(Cn(1-q)\right)^m}{k^k(m-k)^{m-k}}\le
\left(\frac{Cn(1-q)}{m}\right)^m.
\end{align}
The last inequality follows once we recall that $(ca)^a\le a!\le
(Ca)^a$ for $a\ge 1$, and note that ${m \choose k} \le 2^m$. This
finishes the proof of the proposition.
\end{proof}

\begin{proof}[Proof of Proposition~\ref{prop:lis-ub}]
For $1\le m\le n$, denote by $\mathcal{I}_m$ the set of all
increasing integer sequences $I=(i_1, \ldots, i_m)\subseteq[n]$.
Observe that $|\mathcal{I}_m|\le {n \choose m}$. Applying a union
bound and Proposition \ref{prop:prob-bound-i} we obtain for all
integer $L\ge Cn\sqrt{1-q}$ that
\begin{align*}
\p(\LIS(\pi)\ge L) & \le \displaystyle\sum_{L\le m \le n,\,
I\in\mathcal{I}_m} \p(E_I) \le \displaystyle\sum_{m \ge  L} {n
\choose m}  \left(\frac{Cn(1-q)}{m}\right)^m \le\\
& \le \displaystyle\sum_{m \ge  L}
\left(\frac{Cn^2(1-q)}{m^2}\right)^m \le
\left(\frac{Cn^2(1-q)}{L^2}\right)^L.\qedhere
\end{align*}
\end{proof}
%

\subsubsection{Bounds for $\E(\LIS(\pi))$}\label{sec:bounds_E_LIS}

\begin{proof}[Proof of Proposition~\ref{prop:small_q_LIS_exp_bound}] Suppose
that $n\ge 1$, $0<q\le 1$ and $\pi\sim\mu_{n,q}$. Couple $\pi$ with
the $q$-Mallows process using
Corollary~\ref{cor:perms-that-are-mallows} so that $\pi(i) = n+1 -
p_n(i)$ for all $i$. Define
\begin{equation*}
  I_1:=\{1\le i\le n\,:\, p_i(i)=1\}.
\end{equation*}
Then, by the definition of the Mallows process,
\begin{equation}\label{eq:LIS_I_1_comparison}
  \LIS(\pi)\ge |I_1|.
\end{equation}
Observe that by \eqref{eq:Mallows_folder_dist}, for each $i\ge 1$,
\begin{equation*}
  \P(i\in I_1) = \P(p_i(i)=1) \ge
  1-q.
\end{equation*}
Together with \eqref{eq:LIS_I_1_comparison} this implies that
$\E(\LIS(\pi))\ge n(1-q)$. To see the other direction, define the
set of descents of $\pi$,
\begin{equation*}
  I_2:=\{1\le i\le n-1\,:\, \pi(i)>\pi(i+1)\}.
\end{equation*}
It is not hard to check that
\begin{equation}\label{eq:LIS_I_2_comparison}
  \LIS(\pi)\le n - |I_2|.
\end{equation}
By Corollary~\ref{cor:induced_permutation_Mallows}, for each $1\le
i\le n-1$,
\begin{equation*}
  \P(i\in I_2) = \frac{q}{1+q}.
\end{equation*}
Together with \eqref{eq:LIS_I_2_comparison} this implies that
$\E(\LIS(\pi))\le n-\frac{q}{1+q}(n-1)$.
\end{proof}

We continue to prove the bound \eqref{eq:E(LIS)} of Theorem
\ref{thm:lis-largedev}. Fix $n\ge 1$ and $\frac{1}{2}\le q\le 1 -
\frac{4}{n}$. We make use of the large deviation bounds in
\eqref{eq:LIS_greater_than_L} and \eqref{eq:LIS_less_than_L_UB}
shown previously. Set $L^* := 2C_0 n\sqrt{1-q}$ where $C_0$ is the
constant $C$ appearing in
Theorem \ref{thm:lis-largedev}. Applying
\eqref{eq:LIS_greater_than_L}, for any integer $L \ge L^*$,
\[
\P(\LIS(\pi) \ge L) \le \frac{1}{2^L}.
\]
Thus,
\begin{align}
\E(\LIS(\pi)) \le L^* + \displaystyle\sum_{L>L^*}\p(\LIS(\pi) \ge L)
\le L^* + \displaystyle\sum_{L>L^*} \frac{1}{2^L} \le L^*+1. \nonumber
\end{align}

Now let $c_0$ be the constant $c$ appearing in Theorem
\ref{thm:lis-largedev}. We will prove that
\begin{equation}\label{eq:lower_bound_E_LIS_to_prove}
  \E(\LIS(\pi))\ge \frac{c_0}{4}n\sqrt{1-q}.
\end{equation}
Since $\E(\LIS(\pi))\ge n(1-q)$ by
Proposition~\ref{prop:small_q_LIS_exp_bound}, the bound
\eqref{eq:lower_bound_E_LIS_to_prove} follows when $q\le
1-\frac{c_0^2}{16}$. Assume that $q>1-\frac{c_0^2}{16}$. Since we
have also assumed that $q\le 1 - \frac{4}{n}$ we obtain that
\begin{equation}\label{eq:n_q_relation_lower_bound_E_LIS}
  \frac{c_0}{2}n\sqrt{1-q} > 2n(1-q)\ge 8.
\end{equation}
Thus, defining $L^*:=c_0n\sqrt{1-q}$, it follows that
\begin{equation*}
  L^*\ge \lfloor L^*\rfloor \ge \frac{L^*}{2}\ge n(1-q).
\end{equation*}
Applying the bound \eqref{eq:LIS_less_than_L_UB} and using
\eqref{eq:n_q_relation_lower_bound_E_LIS} gives
\[
\p(\LIS(\pi) < \lfloor L^* \rfloor ) \le \exp\left(-
\frac{c_0n^2(1-q)}{\left\lfloor c_0 n \sqrt{1-q}
\right\rfloor}\right) \le \exp\left(- n \sqrt{1-q} \right) \le
\exp\left(-n(1-q)\right)\le \frac 12.
\]
Therefore,
\begin{align}
\E(\LIS(\pi)) & \ge \lfloor L^* \rfloor (1-\p(\LIS(\pi) < \lfloor
L^* \rfloor )) \ge \frac{L^*}{4}, \nonumber
\end{align}
proving \eqref{eq:lower_bound_E_LIS_to_prove} in the case
$q>1-\frac{c_0^2}{16}$, as required.

%

\subsubsection{The $\LIS$ of elements mapped far by the Mallows process}\label{sec:generalized_bound_on_LIS}

In this section we extend the bound of Proposition~\ref{prop:lis-ub}
to a refined estimate which will be used in Section~\ref{sec:wlln}.
Let $n\ge 1$, $0<q<1$ and let $\pi$ be a random permutation with the
$\mu_{n,q}$ distribution. Consider again the coupling
\eqref{eq:pi_p_coupling_LIS_UB} of $\pi$ with the $q$-Mallows
process $(p_k)$. Fix a real number $a>0$ and define a subset $T$ of
the integers by
\begin{equation*}
  T:=\left\{i\,:\, p_i(i)\ge \frac{a}{1-q}\right\}.
\end{equation*}
Thus, $T$ is the set of all elements which, at the time of their
assignment by the Mallows process, were assigned a value no smaller
than $a/(1-q)$. Let $B\subseteq[n]$ be a contiguous block of
integers, i.e., $B := \{i_0, \ldots, i_0+|B|-1\}$ for some $i_0\ge
1$ such that $i_0 + |B| - 1\le n$. Our main result concerns the
length of the longest increasing subsequence of $\pi$ restricted to
$B\cap T$.
\begin{theorem}\label{thm:generalized_upper_bound_LIS}
Suppose $n\ge 1, a>0$ and $\frac{1}{2} \le q \le 1- \frac{2}{n}$. If
$|B| \ge \frac{a}{1-q}$ then
\[
\p(\LIS(\pi_{B\cap T})\ge L) \le \frac{1}{|B|(1-q)}
\left(\frac{Ce^{-a}|B|^2(1-q)}{L^2}\right)^L
\]
for all integer $L \ge Ce^{-a/2}|B|\sqrt{1-q}$.
\end{theorem}
An important feature of this bound is that it is uniform in $n$. In
fact, the result is similar to the upper bound of
\eqref{eq:LIS_greater_than_L} in Theorem~\ref{thm:lis-largedev},
with $n$ replaced by $e^{-a/2}|B|$.

Observe the trivial inequality $\LIS(\pi_{B\cap T})\le \LIS(\pi_B)$.
It implies that if $a\le10$, say, the theorem follows from
Corollary~\ref{cor:induced_permutation_Mallows} and
Proposition~\ref{prop:lis-ub}. Thus we assume in the sequel that
$a>10$. Assume in addition that $\frac{1}{2} \le q \le 1-
\frac{2}{n}$, as in the theorem.

The proof strategy is a modification of the argument of
Proposition~\ref{prop:lis-ub}, using a union bound
over all possible increasing subsequences which are subsets of
$B\cap T$. Recall the definitions of the events $E_{I,J}$ and $E_I$
from \eqref{eq:E_I_J_event} and \eqref{eq:E_I_event}.
\begin{lemma}\label{lem:prob-bound-i-j-2}
Let $m\ge 1$. Let $I=(i_1,\ldots,i_m)$ be an increasing sequence of
integers, and let $J=(j_1,\ldots,j_m)$ be a sequence of integers
satisfying $\frac{a}{1-q}\le j_k \le i_k$. Then
\[
\p(E_{I,J}) \le (C(1-q))^m q^{\sum j_k}.
\]
\end{lemma}
\begin{proof} Observe that, since $a>10$, we must have $i_1\ge 1/(1-q)$. Thus, by
\eqref{eq:Mallows_folder_dist} and our assumption that $q\ge
\frac{1}{2}$,
\begin{equation*}
\p(p_{i_k}(i_k)=j_k\text{ for all $1\le k\le m$}) =
\displaystyle\prod_{1 \le k \le
  m}\frac{(1-q)q^{j_{k}-1}}{1-q^{i_k}} \le \left(C(1-q)\right)^m q^{\sum j_k}. \qedhere
\end{equation*} 
\end{proof}
We need the following combinatorial lemma, inspired by a related
fact on partitions (see, e.g., \cite[Theorem 15.1]{LinWil01}).
\begin{lemma}\label{lem:partition_related_counting}
Let $1\le m\le |B|$ and let $I=(i_1, \ldots, i_m)\subseteq B$ be an
increasing sequence of integers. For an integer $s\ge 1$ define a
family of integer sequences by
\begin{equation*}
  \mathcal{J}_{s,I}':=\left\{(j_1,\ldots, j_m)\,:\,\sum_{k=1}^m j_k = s,\;
  j_k\ge 0\;\text{ and }\; j_{k+1} - j_k \le
  i_{k+1} - i_k\right\}.
\end{equation*}
Then
\begin{equation*}
  |\mathcal{J}'_{s,I}| \le \left(\frac{C}{m^2}\right)^{m-1}\left(s^{m-1} + (m|B|)^{m-1}\right).
\end{equation*}
\end{lemma}
\begin{proof}
  Define a transformation from a sequence $J\in\mathcal{J}'_{s,I}$
  to a sequence $(\ell_1,\ldots, \ell_m)$ by
  \begin{equation*}
    \ell_k := j_k + i_m - i_k
    + (m-k).
  \end{equation*}
  It follows from the definition of $\mathcal{J}'_{s,I}$ that each
  $\ell_k$ is an integer, $\ell_1>\ell_2>\cdots\ell_m\ge 0$ and
  \begin{equation*}
    \sum_{k=1}^m \ell_k = s+mi_m - \sum_{k=1}^m i_k + \frac{m(m-1)}{2}=:s'.
  \end{equation*}
  Thus, all $m!$ permutations of $(\ell_1,\ldots, \ell_m)$ are
  distinct and each such permutation solves the equation
  \begin{equation}\label{eq:J_integer_problem}
    x_1 + \cdots + x_m = s'\quad\text{ where each $x_i$ is a
    non-negative integer}.
  \end{equation}
  Since the transformation from $J$ to $(\ell_k)$ is one-to-one, we
  conclude that $m!|\mathcal{J}'_{s,I}|$ is bounded above by the
  number of solutions to \eqref{eq:J_integer_problem}. Thus,
  \begin{equation*}
    |\mathcal{J}'_{s,I}|\le \frac{1}{m!}{s' + m - 1 \choose m-1}\le \left(\frac{C(s'
    + m)}{m^2}\right)^{m-1}\le
    \left(\frac{C(s+2m|B|)}{m^2}\right)^{m-1},
  \end{equation*}
  and the lemma follows from the fact that $(s + 2m|B|)^{m-1}\le
  (2\max(s, 2m|B|))^{m-1}\le (2s)^{m-1} + (4m|B|)^{m-1}$.
\end{proof}

\begin{proposition}\label{prop:prob-bound-i-2}
Let $1\le m\le |B|$ and let $I=(i_1, \ldots, i_m)\subseteq B$ be an
increasing sequence of integers. If $|B| \ge \frac{a}{1-q}$ then
\begin{equation*}
\p(E_I\cap \{I\subseteq T\}) \le (Ce^{-a})^m\left(\frac{|B|(1-q)}{m}
\right)^{m-1}.
\end{equation*}
\end{proposition}
\begin{proof}
Fix a sequence $I$ as in the proposition. For an integer $s\ge
ma/(1-q)$, define a family of integer sequences by
\begin{equation*}
  \mathcal{J}_{s,I}:=\left\{(j_1,\ldots, j_m)\,:\,\sum_{k=1}^m j_k = s,\;
  \frac{a}{1-q}\le j_k\le i_k\;\text{ and the event $E_I\cap E_{I,J}$ is non-empty}\right\}.
\end{equation*}
As in Proposition~\ref{prop:prob-bound-i},
\eqref{eq:j_k_increasing_nec_cond} holds for all
$J\in\mathcal{J}_{s,I}$. Thus $\mathcal{J}_{s,I}\subseteq
\mathcal{J}'_{s,I}$ and Lemma~\ref{lem:partition_related_counting}
implies that
\begin{equation*}
  |\mathcal{J}_{s,I}| \le \left(\frac{C}{m^2}\right)^{m-1}\left(s^{m-1} + (m|B|)^{m-1}\right).
\end{equation*}
Combining this with Lemma~\ref{lem:prob-bound-i-j-2} we obtain that
\begin{align}
\p(E_I\cap \{I\subseteq T\}) &= \displaystyle\sum_{s\ge
\frac{ma}{1-q}}\sum_{J\in\mathcal{J}_{s,I}}\p(E_I
  \cap E_{I,J}) \le  \sum_{s\ge
\frac{ma}{1-q}}\sum_{J\in\mathcal{J}_{s,I}}\p(E_{I,J}) \le
\left(C(1-q)\right)^m \sum_{s\ge \frac{ma}{1-q}} |\mathcal{J}_{s,I}|
q^{s} \le\nonumber\\
&\le \left(C(1-q)\right)^m m^{-2(m-1)} \left(\sum_{s\ge
\frac{ma}{1-q}} s^{m-1}q^{s} + \sum_{s\ge \frac{ma}{1-q}}
(m|B|)^{m-1}q^{s}\right).\label{eq:E_I_first_probability_estimate}
\end{align}
To estimate the first sum in
\eqref{eq:E_I_first_probability_estimate}, observe that the ratio of
consecutive elements in it is at most $(1+1/s)^{m-1}q\le 1 -
(1-q)/2$ since $s\ge ma/(1-q)\ge 10m/(1-q)$. Thus,
\begin{align*}
  &\sum_{s\ge
\frac{ma}{1-q}} s^{m-1}q^{s} \le \frac{2}{1-q}\left( \frac{ma}{1-q}
\right)^{m-1} q^{ma/(1-q)}\le
\frac{2e^{-ma}}{1-q}\left( \frac{ma}{1-q} \right)^{m-1}  \quad\text{ and }\\
  &\sum_{s\ge \frac{ma}{1-q}}
(m|B|)^{m-1}q^s \le \frac{1}{1-q} (m|B|)^{m-1} q^{ma/(1-q)}  \le
\frac{e^{-ma}(m|B|)^{m-1}}{1-q}.
\end{align*}
Plugging these bounds into \eqref{eq:E_I_first_probability_estimate}
and using the assumption $|B| \ge \frac{a}{1-q}$ yields the result
of the proposition.
\end{proof}
\begin{proof}[Proof of Theorem~\ref{thm:generalized_upper_bound_LIS}]
For $1\le m\le |B|$, denote by $\mathcal{I}_m$ the set of all
increasing integer sequences $I=(i_1, \ldots, i_m)\subseteq B$.
Observe that $|\mathcal{I}_m|= {|B| \choose m}$. Let $C_1$ be a
large absolute constant. Applying a union bound and Proposition
\ref{prop:prob-bound-i-2} we obtain for all integer $L\ge C_1
e^{-a/2}|B|\sqrt{1-q}$
that
\begin{align*}
\p(\LIS(\pi_{B\cap T})\ge L)  &\le \displaystyle\sum_{\substack{L\le
m \le |B|\\I\in\mathcal{I}_m}} \p(E_I\cap \{I\subseteq T\}) \le
\displaystyle\sum_{m \ge  L} {|B| \choose m}
(Ce^{-a})^m\left(\frac{|B|(1-q)}{m} \right)^{m-1} \le\\
& \le \frac{1}{|B|(1-q)}\displaystyle\sum_{m \ge  L}
m\left(\frac{Ce^{-a}|B|^2(1-q)}{m^2}\right)^m \le
\frac{1}{|B|(1-q)}\displaystyle\sum_{m \ge  L}
\left(\frac{Ce^{-a}|B|^2(1-q)}{m^2}\right)^m \le\\
&\le \frac{2}{|B|(1-q)} \left(\frac{Ce^{-a}|B|^2(1-q)}{L^2}\right)^L
\end{align*}
where for the last inequality we took the constant $C_1$ to be
sufficiently large.
\end{proof}

\section{Law of large numbers for $\LIS(\pi)$} \label{sec:wlln}
In this section we prove Theorem \ref{thm:lp-convergence-lis}. Let
$\pi\sim\mu_{n,q}$. We wish to show that
\begin{align}\label{eq:LIS-limit-I}
\frac{\LIS(\pi)}{n\sqrt{1-q}} \to 1\quad \mathrm{\ as \ }\; n \to
\infty,\; q \to 1,\; n(1-q) \to \infty
\end{align}
in $L_p$ for every $0< p < \infty$. The restrictions on $n$ and $q$
in the above limit should be interpreted as saying that $n \to
\infty$ and $q \to 1$ in any way so that $n(1-q) \to \infty$.

\subsection{Block decomposition}\label{sec:block_decomposition}
Let $n=n(q)$ be a function of $q$ such that
\begin{align}\label{eq:n_q_limit}
\lim_{q \to 1}\; n = \infty\quad\text{ and}\quad\lim_{q \to 1}\;
n(1-q)= \infty.
\end{align}
Let $\pi\sim\mu_{n,q}$. To prove \eqref{eq:LIS-limit-I} it suffices
to show that
\begin{align*}
\lim_{q \to 1}\; \frac{\LIS(\pi)}{n\sqrt{1-q}} =1
\end{align*}
in $L_p$ for every $0< p < \infty$. As mentioned in the
introduction, we will achieve this by partitioning $\{1,\ldots, n\}$
into blocks of size $\frac{\beta}{1-q}$, for some large $\beta$,
considering the longest increasing subsequence of the permutation
restricted to each block, and showing that the concatenation of
these subsequences is close to being an increasing subsequence for
the entire permutation. We proceed to make this idea formal.

Let $\beta>0$ and define a function $\beta(q)$ such that
$\beta(q)/(1-q)$ is an integer and $\beta(q) \to \beta$ as $q \to
1$. As a note to the reader we remark that we would have gladly set
$\beta(q)$ equal to $\beta$ in the rest of our argument, but we need
$\beta(q)/(1-q)$ to be an integer for technical reasons.
Define
\begin{align}\label{eq:m_beta_def}
 m:=\left \lfloor \frac{n(1-q)}{\beta(q)} \right \rfloor
\end{align}
and for $1 \le i \le {m}$, define
\begin{align*}
{B}_i:=\left\{ (i-1) \frac{\beta(q)}{1-q}+1,\ldots, i
\frac{\beta(q)}{1-q}\right\}.
\end{align*}
Thus the ${B}_i$ are blocks of size $\beta(q)/(1-q)$ of consecutive
integers which, possibly along with a block of smaller size
${B}_{{m}+1}:=\big\{{m}\frac{\beta(q)}{(1-q)}+1,\ldots,n\big\}$,
partition $\{1,\ldots, n\}$. For $1 \le i \le m+1$, let
\begin{align*}
X_i:= \LIS(\pi_{B_i})
\end{align*}
be the length of the longest increasing subsequence of the
restriction of $\pi$ to $B_i$. By Lemma \ref{lem:indep-seqs}, the
$X_i$ are independent. By
Corollary~\ref{cor:induced_permutation_Mallows}, each $X_i$ has the
distribution of the length of the longest increasing subsequence of
a Mallows permutation of length $|B_i|$ and parameter $q$.

We regard the above objects, $\beta(q)$, $m$, $(B_i)$ and $(X_i)$,
as implicit functions of $\beta$ and $q$. In particular, when we
take the limits $q \to 1$ and $\beta \to \infty$ below it will be
assumed that for every $\beta$ and $q$ these objects are defined by
the above recipe.

Using the triangle inequality,
\begin{align*}
\left|\frac{\LIS(\pi)}{n\sqrt{1-q}}-1\right| \le \left|
\frac{\LIS(\pi) - \sum_{i=1}^{m}X_i}{n\sqrt{1-q}}\right|
+ \left| \frac{\sum_{i=1}^{m}X_i}{n\sqrt{1-q}}
-1\right|.
\end{align*}
We will prove that
\begin{align}
\limsup_{\beta \to \infty}\; \limsup_{q \to 1}\; \E\left( \left|
\frac{\LIS(\pi) -
    \sum_{i=1}^{m}X_i}{n\sqrt{1-q}} \right| \right)
=0, \label{eq:LIS-Mallows}\\
\limsup_{\beta \to \infty}\; \limsup_{q \to 1}\; \E \left(\left|
\frac{\sum_{i=1}^{m}X_i}{n\sqrt{1-q}} -1 \right|\right)
=0. \label{eq:Mueller-Starr}
\end{align}
These equalities 
imply that
\begin{align}
\limsup_{\beta \to \infty}\; \limsup_{q \to 1}\; \E\left( \left|
  \frac{\LIS(\pi)}{n\sqrt{1-q}} -1 \right| \right) =0 \nonumber
\end{align}
and since $\pi$ does not depend on $\beta$, in fact
\begin{align}
\lim_{q \to 1}\; \E\left( \left|
  \frac{\LIS(\pi)}{n\sqrt{1-q}} -1 \right| \right) =0. \nonumber
\end{align}
In other words,
\[
\frac{\LIS(\pi)}{n\sqrt{1-q}} \stackrel{L^1}{\to} 1\quad\text{as
$q\to 1$}.
\]
Convergence in $L^1$ implies convergence in probability. By our
large deviation bounds, Theorem~\ref{thm:lis-largedev}, for any $0<
p<\infty$, we have
\begin{align}\label{eq:LIS_uniform_integrability}
\limsup_{q\to 1}\; \E\left(\left| \frac{\LIS(\pi)}{n\sqrt{1-q}}
\right|^p\right) < \infty.
\end{align}
By considering some $p'>p$ we conclude that for each fixed $p$,
$\left\{\left| \frac{\LIS(\pi)}{n\sqrt{1-q}} \right|^p\right\}$,
regarded as a set of random variables indexed by $q$, is uniformly
integrable (starting from $q$ sufficiently close to $1$) and hence
\[
\frac{\LIS(\pi)}{n\sqrt{1-q}} \stackrel{L^p}{\to} 1\quad\text{for
all $0< p <\infty$}.
\]

In the following sections, we prove \eqref{eq:LIS-Mallows} using
properties of the Mallows process and show how
\eqref{eq:Mueller-Starr} follows from the results of Mueller and
Starr \cite{MueSta11}.

\subsection{Comparing $\LIS$ with $\sum X_i$}
\label{sec:LIS-Xi-comparison}
In this section we establish \eqref{eq:LIS-Mallows}. Recall that
$X_i$ is the length of the longest increasing subsequence of
$\pi$ restricted to $B_i$. Since the $(B_i)$ partition
$[n]$
it follows trivially that
\begin{align}\label{eq:lis-upper}
\LIS(\pi) \le \displaystyle\sum_{i=1}^{m+1}X_i.
\end{align}
Next, we show a bound in the other direction. Recalling the
$q$-Mallows process of Section~\ref{sec:Mallows_process}, we now use
the coupling of $\pi$ and $(p_i)$,
\begin{equation*}
\pi(j) = n+1-p_n(j),
\end{equation*}
introduced in Corollary~\ref{cor:perms-that-are-mallows}. Let $a =
a(\beta)>0$ be any function of $\beta$ satisfying
\begin{equation}\label{eq:a_function_prop}
  a\to\infty\quad\text{and}\quad\frac{a}{\beta}\to 0\quad\text{as $\beta\to
  \infty$}.
\end{equation}
For each $i$, let $E_i$ be the subset of elements of the block $B_i$
whose final position, after the block $B_i$ is assigned by the
Mallows process, is at most $a/(1-q)$. That is,
\[
E_i := \left\{j \in B_i \ : \ p_{\max B_i}(j) \le
\frac{a}{1-q}\right\}.
\]
Let $F_i$ be the subset of $B_i$ which is initially assigned a
position larger than $a/(1-q)$ by the Mallows process. That is,
\[
F_i := \left\{j \in B_i \ : \ p_j(j) >
\frac{a}{1-q}\right\}.
\]

Let $I_i\subseteq B_i$ be the indices of an (arbitrary)
longest increasing subsequence in the restriction of $\pi$ to
$B_i$, so that $|I_i|=X_i$.
Define
\begin{equation}\label{eq:I_i_prime_def}
I_i' := I_i \setminus (E_i \cup F_i).
\end{equation}
The definition of $B_i, E_i$ and $F_i$ implies
that $\cup_i I_i'$ is a set of indices of an increasing
subsequence in $\pi$. To see this, let $j,k\in\cup_i I_i'$
satisfy $j<k$. If $j,k\in I_i$ for some $i$ then $\pi(j)<\pi(k)$ by
definition of $I_i$. Otherwise $j\in I_{i_1} \setminus
(E_{i_1} \cup F_{i_1})$ and $k\in I_{i_2} \setminus
(E_{i_2} \cup F_{i_2})$ for some $i_1<i_2$. Then, by the
definitions of $F_{i_2}$ and $E_{i_1}$,
\begin{align*}
p_k(k)  \le \frac{a}{1-q} <p_{\max B_{i_1}}(j) \le p_{k}(j),
\end{align*}
which implies that $p_n(k)<p_n(j)$, so that $\pi(k)>\pi(j)$. Thus,
\begin{equation}\label{eq:LIS_first_lower_bound}
  \LIS(\pi) \ge \sum_{i=1}^{m} |I_i'|.
\end{equation}
Moreover, the definition of $I_i$ and \eqref{eq:I_i_prime_def}
implies that
\begin{equation*}
  X_i = |I_i| \le |I_i'| + \LIS(\pi_{E_i}) +\LIS(\pi_{F_i}),
\end{equation*}
%
so that together with \eqref{eq:LIS_first_lower_bound} we have
\begin{align}\label{eq:lis-lower}
\LIS(\pi) \ge \sum_{i=1}^{m}X_i
-\sum_{i=1}^{m}\LIS(\pi_{E_i}) -
\sum_{i=1}^{m}\LIS(\pi_{F_i}).
\end{align}
Thus, from the upper and lower bounds \eqref{eq:lis-upper} and \eqref{eq:lis-lower}, we deduce that
\begin{align*}
\E\Bigg[\Bigg|\LIS(\pi)-  \sum_{i=1}^{m}X_i \Bigg|
\Bigg] \le \sum_{i=1}^{m}\E\left(\LIS(\pi_{E_i})\right)
+\sum_{i=1}^{m}\E\left(\LIS(\pi_{F_i})\right) +\E\left(
  X_{m+1}\right).
\end{align*}

Relation \eqref{eq:LIS-Mallows} is a direct consequence of the next
lemma, which provides asymptotic bounds for each of the terms on the
right-hand side.
\begin{lemma}
\begin{align}
&\limsup_{\beta \to \infty}\; \limsup_{q \to 1}\; \E\left(
\frac{X_{m+1}}{n\sqrt{1-q}} \right)
=0,\label{eq:leftover_LIS_bound}\\
&\limsup_{\beta \to \infty}\; \limsup_{q \to 1}\; \E\left(
\frac{\sum_{i=1}^{m} \E(\LIS(\pi_{E_i}))}{n\sqrt{1-q}}
\right)
=0,\label{eq:LIS_E_I_bound}\\
&\limsup_{\beta \to \infty}\; \limsup_{q \to 1}\; \E\left(
\frac{\sum_{i=1}^{m} \E(\LIS(\pi_{F_i}))}{n\sqrt{1-q}}
\right) =0.\label{eq:LIS_F_i_bound}
\end{align}
%
%
%
%
\end{lemma}

\begin{proof}
%

Throughout the proof we assume that $\beta$ is sufficiently large
and $q$ is sufficiently close to $1$ so that $n(1-q)$ is large,
$\beta(q)$ is close to $\beta$, $a$ is large and $\frac{a}{\beta}$
is small.

Recall that $X_{m+1}$ has the distribution of the length of the
longest increasing subsequence of a Mallows permutation of length
$|B_{m+1}|\le \frac{\beta(q)}{1-q}$ and parameter $q$. Hence
Theorem~\ref{thm:lis-largedev} implies that
\[
\E\left( X_{m+1} \right) \le \frac{C\beta(q)}{\sqrt{1-q}}
\]
for some constant $C>0$ independent of $q$ and $\beta$. Thus
\[
\lim_{q \to 1}\; \E\left(\frac{X_{m+1}}{n\sqrt{1-q}} \right) \le
\lim_{q \to 1}\; \frac{C\beta(q)}{n(1-q)} =0
\]
for any fixed $\beta>0$, by our assumption that $\beta(q)\to\beta$
and $n(1-q)\to\infty\,$ as $q$ tends to 1. This establishes
\eqref{eq:leftover_LIS_bound}.

We continue to bound $\E(\LIS(\pi_{E_i}))$. Our goal is to
show that $\LIS(\pi_{E_i})$ is stochastically dominated by the
longest increasing subsequence of a permutation with the $(\lfloor\frac{a}{1-q}\rfloor,
q)$-Mallows distribution. To see this, set
\begin{equation*}
  I:=\left(1,2,\ldots, \left\lfloor\frac{a}{1-q}\right\rfloor\right)\;\;\text{ and }\;\;\bar{E}_i:=(p_{\max B_i})^{-1}(I).
\end{equation*}
It follows that $E_i \subseteq \bar{E}_i$.
Now, denote $\sigma:=p_{\max B_i}$. Then
\begin{equation*}
  \LIS(\pi_{E_i}) = \LDS((p_n)_{E_i}) = \LDS(\sigma_{E_i}) \le
  \LDS(\sigma_{\bar{E}_i}) = \LDS((\sigma^{-1})_{I}) = \LIS(((\sigma^{-1})_I)^R).
\end{equation*}
Since $\sigma^{-1}\sim\mu_{\max B_i,1/q}$ by
Lemma~\ref{lem:reversal_inverse_Mallows}, it follows by
Corollary~\ref{cor:induced_permutation_Mallows} that
$(\sigma^{-1})_I\sim\mu_{\lfloor\frac{a}{1-q}\rfloor,1/q}$. Finally,
another application of Lemma~\ref{lem:reversal_inverse_Mallows}
shows that
$((\sigma^{-1})_I)^R\sim\mu_{\lfloor\frac{a}{1-q}\rfloor,q}$,
proving the required stochastic domination. Applying
Theorem~\ref{thm:lis-largedev} we conclude that
\begin{equation*}
  \E(\LIS(\pi_{E_i}))\le C\left\lfloor\frac{a}{1-q}\right\rfloor\sqrt{1-q} \le C\frac{a}{\sqrt{1-q}}.
\end{equation*}


Thus, recalling the definition of $m$ from
\eqref{eq:m_beta_def}, our assumption that $\beta(q)\to\beta$ as
$q\to 1$ and the properties of $a$ from \eqref{eq:a_function_prop},
we have
\begin{align*}
  &\limsup_{\beta \to \infty}\; \limsup_{q \to 1}\; \E\left[
\frac{\sum_{i=1}^{m} \E(\LIS(\pi_{E_i}))}{n\sqrt{1-q}}
\right] \le \limsup_{\beta \to \infty}\; \limsup_{q \to 1}\;
\frac{C
a m}{n(1-q)} \le\nonumber\\
\le &\limsup_{\beta \to \infty}\; \limsup_{q \to 1}\;
\frac{Ca}{\beta(q)} = \limsup_{\beta \to \infty}\; \frac{Ca}{\beta}
= 0,
\end{align*}
proving \eqref{eq:LIS_E_I_bound}.


We finish by bounding $\E(\LIS(\pi_{F_i}))$. Observe that
$\LIS(\pi_{F_i})$ is of the form studied in
Theorem~\ref{thm:generalized_upper_bound_LIS}. Hence we may apply
this theorem to deduce that for any integer $L\ge
Ce^{-a/2}|B_i|\sqrt{1-q}$,
\begin{equation}\label{eq:generalized_LIS_application}
\p(\LIS(\pi_{F_i})\ge L) \le \frac{1}{|B_i|(1-q)}
\left(\frac{Ce^{-a}|B_i|^2(1-q)}{L^2}\right)^L.
\end{equation}
For each $1\le i\le m$ we may set
\begin{equation*}
  L_0:=C_0 e^{-a/2}|B_i|\sqrt{1-q} = C_0 e^{-a/2}\frac{\beta(q)}{\sqrt{1-q}},
\end{equation*}
with a sufficiently large absolute constant $C_0$, and apply
\eqref{eq:generalized_LIS_application} to obtain
\begin{equation*}
  \E(\LDS(\pi_{F_i})) \le L_0 + \sum_{L>L_0} \p(\LIS(\pi_{F_i})\ge
  L) \le L_0 + \sum_{L>L_0} \frac{1}{|B_i|(1-q)}2^{-L} = L_0 + \sum_{L>L_0}
  \frac{1}{\beta(q)}2^{-L} \le L_0 + 1.
\end{equation*}
Finally, we conclude that
\begin{align*}
  &\limsup_{\beta \to \infty}\; \limsup_{q \to 1}\; \E\left[
\frac{\sum_{i=1}^{m} \E(\LIS(\pi_{F_i}))}{n\sqrt{1-q}} \right] \le
\limsup_{\beta \to \infty}\; \limsup_{q \to 1}\; \frac{m(L_0+1)}{n\sqrt{1-q}} \le\nonumber\\
\le &\limsup_{\beta \to \infty}\; \limsup_{q \to 1}\; C_0 e^{-a/2} +
\frac{\sqrt{1-q}}{\beta(q)} = \limsup_{\beta \to \infty}\; C_0
e^{-a/2} = 0,
\end{align*}
proving \eqref{eq:LIS_F_i_bound}.
\end{proof}

\subsection{Relating to the results of Mueller and Starr}
In this section we establish \eqref{eq:Mueller-Starr}. We rely on
the following result of Mueller and Starr, who proved a weak law of
large numbers for the longest increasing subsequence of a random Mallows permutation in the
regime that $n(1-q)$ tends to a finite limit.
\begin{theorem}[Mueller-Starr \cite{MueSta11}]\label{thm:MS-wlln}
Suppose that $(q_n)_{n=1}^\infty$ satisfies that the limit
\[\beta = \lim_{n \to \infty}\; n(1-q_n)\]
exists and is finite. Then for any $\varepsilon>0$, if $\pi\sim
\mu_{n,q_n}$ then
\begin{align}
\lim_{n \to \infty}\; \p\left(\left| \frac{\LIS(\pi)}{\sqrt{n}} -
    \ell(\beta) \right| > \varepsilon \right) =0, \nonumber
\end{align}
where\begin{align}\label{eq:ell_beta_def} \ell(\beta) =
\begin{cases}
2 \beta^{-1/2}\sinh^{-1}(\sqrt{e^\beta-1}) & \mathrm{for }\ \beta>0 \\
2 &  \mathrm{for }\ \beta=0 \\
2 |\beta|^{-1/2}\sin^{-1}(\sqrt{1-e^\beta}) & \mathrm{for }\ \beta<0
\end{cases}.
\end{align}
\end{theorem}
We continue with the notation of
section~\ref{sec:block_decomposition} and, in particular, suppose
that $n=n(q)$ is such that \eqref{eq:n_q_limit} holds. Recall that
$X_1$ is distributed as the length of a longest increasing
subsequence of a $(\beta(q)/(1-q),q)$-Mallows permutation. Since the
limit
\[
\lim_{q \to 1}\; \frac{\beta(q)}{1-q}\cdot(1-q) = \beta
\]
exists and is finite, we may apply Theorem \ref{thm:MS-wlln} to
$X_1$ and deduce that
\begin{equation}\label{eq:Mueller_Starr_conclusion}
  \sqrt{\frac{1-q}{\beta(q)}}\cdot X_1
  \to \ell(\beta)\quad \text{ in probability, as $q$ tends to $1$.}
\end{equation}
Now fix $\beta_0$ sufficiently large and $q_0$ sufficiently close to
$1$ so that if $\beta\ge \beta_0$ and $q_0\le q<1$ then
$\frac{1}{2}<q<1 - \frac{4(1-q)}{\beta(q)}$ so that our large
deviation estimate, inequality \eqref{eq:LIS_greater_than_L} in
Theorem~\ref{thm:lis-largedev}, may be applied to $X_1$. It
follows, as in \eqref{eq:LIS_uniform_integrability}, that for any
fixed $\beta\ge\beta_0$, the random variables
\begin{equation}\label{eq:X_1_beta_uniform_integrability}
  \left\{\left(\frac{\sqrt{1-q}}{\beta(q)}\cdot X_1\right)^2\right\}\text{
  indexed by $q_0\le q<1$ are uniformly integrable}.
\end{equation}
Since $\beta(q)\to\beta$ as $q\to 1$,
\eqref{eq:Mueller_Starr_conclusion} and
\eqref{eq:X_1_beta_uniform_integrability} imply that for any fixed
$\beta\ge \beta_0$,
\begin{equation*}
  \sqrt{\frac{1-q}{\beta}}\cdot X_1
  \to \ell(\beta)\quad \text{ in $L_2$, as $q$ tends to $1$.}
\end{equation*}
In particular, for any fixed $\beta\ge \beta_0$, we have
\begin{equation}\label{eq:X_1_beta_expectation_and_variance}
  \lim_{q\to 1} \sqrt{\frac{1-q}{\beta}}\cdot \E(X_1) =
  \ell(\beta)\quad\text{and}\quad \lim_{q\to 1}
  (1-q)\cdot\var(X_1) = 0.
\end{equation}

We now consider the random variable
\begin{equation*}
  Y:=\frac{\sum_{i=1}^{m}X_i}{n\sqrt{1-q}}.
\end{equation*}
In order to prove \eqref{eq:Mueller-Starr} we first show that
\begin{align}
  &\lim_{\beta\to \infty}\; \lim_{q\to 1}\; \E(Y) = 1\quad\text{
  and}\label{eq:Y_exp}\\
  &\lim_{\beta\to \infty}\; \lim_{q\to 1}\; \var(Y) = 0.\label{eq:Y_var}
\end{align}
To prove \eqref{eq:Y_exp} we note that since the $(X_i)$ are
identically distributed, we may write
\begin{equation}\label{eq:Y_exp_rewriting}
  \lim_{q\to 1}\; \E(Y) = \lim_{q\to 1}\; \frac{m}{n\sqrt{1-q}}
  \E(X_1) = \lim_{q\to 1}\; \frac{m \beta}{n(1-q)} \frac{\sqrt{1-q}}{\beta}\cdot
  \E(X_1).
\end{equation}
Now, by \eqref{eq:n_q_limit} and \eqref{eq:m_beta_def} we have
\begin{equation}\label{eq:m_beta_ratio_limit}
  \lim_{q\to 1} \frac{m \beta}{n(1-q)} = 1.
\end{equation}
Plugging this into \eqref{eq:Y_exp_rewriting} and using
\eqref{eq:X_1_beta_expectation_and_variance} implies that
\begin{equation}\label{eq:Y_exp_q_limit_final}
  \lim_{q\to 1}\; \E(Y) = \frac{1}{\sqrt{\beta}} \lim_{q\to 1}\; \sqrt{\frac{1-q}{\beta}}\cdot
  \E(X_1) = \frac{\ell(\beta)}{\sqrt{\beta}}
\end{equation}
for any fixed $\beta\ge \beta_0$. Finally, we observe that by
\eqref{eq:ell_beta_def} we have
\begin{equation*}
  \lim_{\beta\to\infty} \frac{\ell(\beta)}{\sqrt{\beta}} = 1,
\end{equation*}
which together with \eqref{eq:Y_exp_q_limit_final} implies
\eqref{eq:Y_exp}.

To prove \eqref{eq:Y_var} we rely also on the fact that the $(X_i)$
are independent. Thus, by \eqref{eq:m_beta_ratio_limit},
\begin{equation*}
  \lim_{q\to 1}\; \var(Y) = \lim_{q\to 1}\; \frac{m}{n^2(1-q)}
  \var(X_1) = \lim_{q\to 1}\; \frac{1}{\beta n}
  \var(X_1) = \lim_{q\to 1}\; \frac{1}{\beta n(1-q)} (1-q)\cdot
  \var(X_1).
\end{equation*}
Hence, if $\beta\ge \beta_0$ then \eqref{eq:n_q_limit} and
\eqref{eq:X_1_beta_expectation_and_variance} imply that
\begin{equation*}
  \lim_{q\to 1}\; \var(Y) = 0,
\end{equation*}
proving \eqref{eq:Y_var}.

Finally, by the triangle and Cauchy-Schwartz inequalities we have
\begin{equation*}
  \E|Y - 1|\le \E|Y - \E(Y)| + |\E(Y) - 1| \le \sqrt{\var(Y)} + |\E(Y) -
  1|,
\end{equation*}
which shows that \eqref{eq:Y_exp} and \eqref{eq:Y_var} imply
\eqref{eq:Mueller-Starr}.

\section{Decreasing subsequences}\label{sec:lds}
In this section we prove Theorems~\ref{thm:E(LDS)} and
\ref{thm:lds-largedev} concerning the length of the longest
decreasing subsequence in a Mallows permutation.
Part~\eqref{part:LDS_large_dev_upperbound} of
Theorem~\ref{thm:lds-largedev} is established in
Section~\ref{sec:lds-large-dev}. In Section \ref{sec:lds_lb}
we prove part \eqref{part:LDS_greater_than_L_lb} of Theorem \ref{thm:lds-largedev} and in Section \ref{sec:erdos-szekeres} we prove part \eqref{part:lds-less-than-ub}. In Section \ref{sec:E(LDS)} using the established large deviation inequalities for $\LDS(\pi)$ we derive the different regimes of the order of magnitude of $E(\LDS(\pi))$ proving Theorem~\ref{thm:E(LDS)}. This last section also includes the proof of Proposition \ref{prop:identity_for_small_q}.

\subsection{An upper bound on the probability of a long decreasing subsequence}\label{sec:lds-large-dev} In this section
we obtain an upper bound on the probability of having a long
decreasing subsequence in a Mallows permutation. Precisely, we show
that if $\pi\sim\mu_{n,q}$ for $0<q<1-\frac{2}{n}$ then
\begin{equation}\label{eq:LDS_upper_bound_in_sec}
  \P(\LDS(\pi) \ge L) \le
  n^8\begin{cases}\left(\frac{C}{(1-q)L^2}\right)^L&
  L\le\frac{3}{1-q}\\(C(1-q))^L
  q^{\frac{L(L-1)}{2}}&L>\frac{3}{1-q}\end{cases}
\end{equation}
for any $L\ge 2$. This establishes \eqref{eq:LDS_greater_than_L}. We
also establish \eqref{eq:LDS_greater_than_L_refined}, a more refined
result for small $q$, showing that for $0<q<\frac{1}{2}$ and $L\ge
2$,
\begin{equation}\label{eq:LDS_refined_upper_bound_in_sec}
  \P(\LDS(\pi)\ge L)\le nC^L q^{\frac{L(L-1)}{2}}.
\end{equation}
The method of proof, as in
Section~\ref{sec:upper_bound_prob_long_LIS}, is to first bound the
probability that a particular set of inputs to the
permutation forms a decreasing subsequence of length $L$ and then to
perform a union bound over all the possibilities for such inputs. However, the calculations turn out to be somewhat involved.

\subsubsection{Preliminary Calculations}
We begin with some preliminary calculations.
\begin{lemma}\label{lem:sum_bin_prob}
For any $0<p<1$ and integer $r\ge 1$, if we denote by $X(d)$ a
random variable with distribution $\bin(d,p)$ then
\begin{equation*}
  \sum_{d=0}^\infty \P(X(d)<r) = \frac{r}{p}.
\end{equation*}
\end{lemma}
In this lemma, as well as below, we say that $X$ has the
$\bin(0,p)$ distribution meaning that $X$ is the identically zero
random variable.
\begin{proof}
  Let $Y_1,Y_2,\ldots$ be an infinite sequence of independent
  Bernoulli$(p)$ random variables, i.e., $\P(Y_1=1)=1-\P(Y_1=0)=p$.
  Then
  \begin{equation*}
    \sum_{d=0}^\infty \P(X(d)<r) = \sum_{d=0}^\infty \E\left(
    {\mathbbm 1}_{\{\sum_{k=1}^d Y_i < r\}}\right) = \E \sum_{d=0}^\infty \left({\mathbbm 1}_{\{\sum_{k=1}^d Y_i <
    r\}}\right) = \E \bigg[\min\bigg(m\,:\, \sum_{k=1}^{m} Y_i = r\bigg)\bigg],
  \end{equation*}
  where ${\mathbbm 1}_E$ denotes the indicator random variable of the
  event $E$. Observing that $\min(m\,:\, \sum_{k=1}^{m} Y_i = r)$ has the
  distribution of the waiting time for $r$ successes in a sequence
  of independent trials with success probability $p$, that is,
  the distribution of a sum of $r$ independent geometric random variables with
  success probability $p$, we conclude that
\begin{equation*}
  \sum_{d=0}^\infty \P(X(d)<r) = \frac{r}{p}.\qedhere
\end{equation*}
\end{proof}

For integers $m\ge 1$ and $M\ge m$ define the set of integer vectors
\begin{equation*}
J_m(M):=\{(j_1,\ldots, j_{m})\,:\,0\le j_1<
    j_2<\cdots<j_m<M\}.
\end{equation*}
\begin{lemma}\label{lem:ind_calc}
  There exists an absolute constant $C_1>0$ such that for any integers $m\ge 2$ and $j_m\ge m-1$ we have
  \begin{equation*}
    \sum_{J_{m-1}(j_m)} \prod_{k=1}^{m-1} \frac{j_{k+1} - j_k}{j_k+1} \le
    \log(j_m+1)\left(\frac{C_1 j_m}{(m-1)^2}\right)^{m-1},
  \end{equation*}
\end{lemma}
In this lemma, as well as below, we write $\sum_{J_{m-1}(j_m)}$ as a
shorthand for $\sum_{(j_1,\ldots, j_{m-1})\in J_{m-1}(j_m)}$.
\begin{proof}
  We prove the claim by induction. For $m=2$ the claim is
  \begin{equation*}
    \sum_{j_1=0}^{j_2-1} \frac{j_2 - j_1}{j_1+1} \le C_1 j_2\log(j_2+1)
  \end{equation*}
  for any $j_2\ge 1$, which clearly holds if $C_1$ is sufficiently
  large. Now fix $m\ge 3$ and $j_m\ge m-1$, assume the claim holds for $m-1$ (and any $j_{m-1}\ge m-2$), and let
  us prove it for $m$. We have
  \begin{align*}
    \sum_{J_{m-1}(j_m)} &\prod_{k=1}^{m-1} \frac{j_{k+1} - j_k}{j_k+1}= \sum_{j_{m-1}=m-2}^{j_m-1} \left[\frac{j_m -
    j_{m-1}}{j_{m-1}+1} \sum_{J_{m-2}(j_{m-1})} \prod_{k=1}^{m-2} \frac{j_{k+1} -
    j_k}{j_k+1}\right] \le\\
    &\le \sum_{j_{m-1}=m-2}^{j_m-1} \frac{j_m -
    j_{m-1}}{j_{m-1}+1} \log(j_{m-1}+1)\left(\frac{C_1 j_{m-1}}{(m-2)^2}\right)^{m-2}
  \end{align*}
  by the induction hypothesis. It follows that
  \begin{equation}\label{eq:ind_step}
    \sum_{J_{m-1}(j_m)} \prod_{k=1}^{m-1} \frac{j_{k+1} - j_k}{j_k+1} \le
    \log(j_m+1)\left(\frac{C_1}{(m-2)^2}\right)^{m-2}
    \sum_{j_{m-1}=m-2}^{j_m-1} j_{m-1}^{m-3}(j_m-j_{m-1}).
  \end{equation}
  We have $\sum_{j_{m-1}=m-2}^{j_m-1}
  j_{m-1}^{m-3}(j_m-j_{m-1})\le \frac{4j_m^{m-1}}{(m-2)(m-1)}$. One
  way to see this is to let $f(x):=x^{m-3}(j_m-x)$ and
  $x_c:=\frac{(m-3)j_m}{m-2}$. Observing that $f$ attains its
  maximum on $[0,j_m]$ at $x_c$, we have $f(x)\le g(x)$ where
  $g(x):=f(x)$ for $x<x_c$ and $g(x):=f(x_c)$ for $x\ge x_c$. Now,
  since $g$ is increasing on $[0,j_m]$, we have $\sum_{j=1}^{j_m-1}
  f(j)\le \int_1^{j_m} g(x)dx$ which yields the required inequality.
  Thus, continuing \eqref{eq:ind_step}, we have
  \begin{align*}
    \sum_{J_{m-1}(j_m)} &\prod_{k=1}^{m-1} \frac{j_{k+1} - j_k}{j_k+1}\le
    4\log(j_m+1)\left(\frac{C_1}{(m-2)^2}\right)^{m-2}\frac{j_m^{m-1}}{(m-2)(m-1)}=\\
    &=
    \frac{4\log(j_m+1)}{C_1}\left(\frac{m-1}{m-2}\right)^{2m-3}
    \left(\frac{C_1j_m}{(m-1)^2}\right)^{m-1},
  \end{align*}
  from which the induction step follows if $C_1$ is sufficiently
  large.
\end{proof}

\begin{corollary}\label{cor:sum_ind}
  There exists an absolute constant $C>0$ such that for any integers $m\ge 2$ and $M\ge m$ we have
  \begin{equation*}
    \sum_{J_m(M)} \prod_{k=1}^{m-1} \frac{j_{k+1} - j_k}{j_k+1} \le
    \log(M+1)\left(\frac{CM}{m^2}\right)^m.
  \end{equation*}
\end{corollary}
\begin{proof}
  By Lemma~\ref{lem:ind_calc} we have
  \begin{align*}
    \sum_{J_m(M)} \prod_{k=1}^{m-1} \frac{j_{k+1} - j_k}{j_k+1}&\le
    \sum_{j_m=m-1}^{M-1} \log(j_m+1)\left(\frac{C_1
    j_m}{(m-1)^2}\right)^{m-1}\le\\
    &\le
    \log(M+1)\left(\frac{C_1}{(m-1)^2}\right)^{m-1}\frac{M^m}{m},
  \end{align*}
  from which the corollary follows for some $C>C_1$.
\end{proof}
For integers $m\ge 1$ and $M\ge 0$ define the (infinite) set of
integer vectors
\begin{equation*}
J_m'(M):=\{(j_1,\ldots, j_{m})\,:\,M\le j_1<
    j_2<\cdots<j_m\}.
\end{equation*}
\begin{lemma}\label{lem:second_j_bound}
  There exists an absolute constant $C>0$ such that for any $0<q<1$ and integers $m\ge 2$ and $M\ge 0$ we have
  \begin{equation*}
    \sum_{J_m'(M)} q^{\sum_{k=1}^m j_k}\prod_{k=1}^{m-1} (j_{k+1} - j_k)
    = \frac{q^{\frac{m(m-1)}{2} + mM}}{(1-q^m)\prod_{k=1}^{m-1}
    (1-q^{m-k})^2} \le \frac{C^m q^{\frac{m(m-1)}{2} +
    mM}}{(m'(1-q))^{2m'}},
  \end{equation*}
  where $m':=\min(m,\lfloor\frac{1}{1-q}\rfloor)$.
\end{lemma}
\begin{proof}
  We change variables, transforming the vector $I=(j_1,\ldots, j_m)$ to the vector $(j_1,d_1,\ldots,
d_{m-1})$ via the mapping $d_k := j_{k+1}-j_k$. Observing that this
transformation is one-to-one, we have
\begin{equation*}
  \sum_{J_m'(M)} q^{\sum_{k=1}^m j_k}\prod_{k=1}^{m-1} (j_{k+1} - j_k)= \sum_{j_1=M}^\infty \sum_D
  q^{mj_1 +\sum_{k=1}^{m-1} (m-k)d_k}\prod_{k=1}^{m-1} d_k,
\end{equation*}
where the sum is over all integer vectors $D:=\{(d_1,\ldots,
d_{m-1})\ :\ d_k\ge 1\}$. Observing that the sum of products equals
a product of sums since the factors involve different $d_k$'s, we
have
\begin{align*}
  \sum_{J_m'(M)} q^{\sum_{k=1}^m j_k}\prod_{k=1}^{m-1} (j_{k+1} -
  j_k) &= \left(\sum_{j_1=M}^\infty q^{mj_1}\right)
  \prod_{k=1}^{m-1}\left(\sum_{d=1}^\infty dq^{(m-k)d}\right) =
  \frac{q^{mM}}{1-q^m}\prod_{k=1}^{m-1}\frac{q^{m-k}}{(1-q^{m-k})^2},
\end{align*}
proving the equality in the lemma. To prove the inequality, we
observe that
\begin{equation*}
  (1-q^m)\prod_{k=1}^{m-1} (1-q^{m-k})^2 \ge \prod_{k=1}^m (1-q^k)^2 =
  \left[(1-q)^{m}\prod_{k=1}^m\left(\sum_{j=0}^{k-1}q^j\right)\right]^2.
\end{equation*}
Noting that $\sum_{j=0}^{k-1} q^j\ge ck$ when $k\le
\lfloor\frac{1}{1-q}\rfloor$ and $\sum_{j=0}^{k-1} q^j\ge
\frac{c}{1-q}$ when $k\ge \lfloor\frac{1}{1-q}\rfloor$, we deduce
\begin{equation*}
  (1-q)^m\prod_{k=1}^m\left(\sum_{j=0}^{k-1}q^j\right) \ge c^m (m')!(1-q)^{m'}
  \ge c^m(m'(1-q))^{m'},
\end{equation*}
as required.
\end{proof}

\subsubsection{Union bound} Fix $n\ge 3$, $0<q<1-\frac{2}{n}$ and let
$\pi\sim\mu_{n,q}$ for the remainder of this section and the next
(we assume that $n\ge 3$ since otherwise the range for $q$ is
empty). Using Corollary~\ref{cor:perms-that-are-mallows}, we couple
$\pi$ with the $q$-Mallows process so that
\begin{equation}\label{eq:pi_p_coupling_LDS_UB}
  \pi(i) = n+1 - p_n(i)\quad\text{ for all $1\le i\le n$}.
\end{equation}
In a similar (but not identical) way to
Section~\ref{sec:upper_bound_prob_long_LIS}, define, for an
increasing sequence of integers $I = (i_1,\ldots, i_m)$ and a
sequence of integers $J = (j_1, \ldots, j_m)$, the event
\begin{equation*}\label{eq:E_I_J_dec_event}
E_{I,J}:=\{p_{i_k}(i_k)=j_k+1\text{ for all $1\le k\le m$}\}.
\end{equation*}
Additionally, for an increasing sequence of integers $I =
(i_1,\ldots, i_m)\subseteq[n]$, define the event that $I$ is a set
of indices of a decreasing subsequence,
\begin{equation*}\label{eq:E_I_dec_event}
E_{I} := \{\pi(i_{k+1})<\pi(i_k)\text{ for all $1\le k\le m-1$}\}.
\end{equation*}
The starting point for our argument is a bound on the probability of
$E_{I,J}\cap E_I$. Recall the definition of $J_m(M)$ and $J_m'(M)$
from the previous section and define, for integers $m\ge 1$ and
$M\ge 1$, the set of integer vectors
\begin{equation*}
I_m'(M):=\{(i_1,\ldots, i_{m})\,:\,M\le i_1<
    i_2<\cdots<i_m\le n\}.
\end{equation*}

\begin{proposition}\label{prop:prob-I-decreases}
For any $m\ge 2$, $I\in I_m'(\lfloor\frac{1}{1-q}\rfloor)$ and $J\in
J_m'(0)$ we have
\begin{equation*}
 \p(E_{I,J}\cap E_I)\le \left(C(1-q)\right)^m
q^{\sum_{k=1}^m j_k} \prod_{k=1}^{m-1}\p(X_k<j_{k+1}-j_k),
\end{equation*}
where $X_k\sim \bin(i_{k+1}-i_{k} - 1,1-q^{j_k+1})$, $1 \le k \le
m-1$.
\end{proposition}
\begin{proof}
Fix $I$ and $J$ as in the proposition. By the
coupling~\eqref{eq:pi_p_coupling_LDS_UB} of $\pi$ with the Mallows
process, and the definition of the Mallows process, the event
$E_{I,J}\cap E_I$ occurs if and only if
\begin{align}
p_{i_k}(i_k) = j_k + 1,  \ \ \ \ \ & \forall 1 \le k \le m, \nonumber\\
p_{i_{k+1}}(i_{k+1})>p_{i_{k+1}}(i_k), \ \ \ \ \ & \forall 1 \le k
\le m-1.\label{eq:decreasing_cond_for_p}
\end{align}
If some $j_k\ge i_k$ the probability of this event is zero and the
proposition follows trivially. Assume from now on that $j_k<i_k$ for
all $k$. Then \eqref{eq:Mallows_folder_dist} implies that
\begin{equation}\label{eq:assignment_prob_bound}
  \P(E_{I,J}) = \P(\cap_{k=1}^m \{p_{i_k}(i_k) = j_k + 1\}) = \prod_{k=1}^m
  \frac{(1-q)q^{j_k}}{1-q^{i_k}} \le (C(1-q))^m q^{\sum_{k=1}^m
  j_k}
\end{equation}
since $i_k\ge \frac{1}{2(1-q)}$ for all $k$. Now, define the random
variables $D_k:=p_{i_{k+1}}(i_{k}) - p_{i_k}(i_k)$ for $1\le k\le
m-1$. Then we may reinterpret \eqref{eq:decreasing_cond_for_p} in
terms of the $D_k$. Indeed,
\begin{equation}\label{eq:E_I_D_k_reinterpretation}
  \text{on the event $E(I,J)$},\;\;
  p_{i_{k+1}}(i_{k+1})>p_{i_{k+1}}(i_k)\;\;\text{if and only if}\;\;
  D_k<j_{k+1} - j_k
\end{equation}
for each $1\le k\le m-1$. By \eqref{eq:Mallows_iteration_procedure},
\begin{equation}\label{eq:D_k_def}
  D_k\ge \sum_{i=i_k+1}^{i_{k+1}-1}{\bf 1}_{\{p_i(i)\le j_k+1\}}\quad\text{on the event
  $E_{I,J}$},
\end{equation}
where ${\bf 1}_E$ denotes the indicator random variable of the event
$E$, and, for all $i$,
\begin{equation}\label{eq:drift_prob}
  \P(p_i(i)\le j_k+1) = \frac{1+q+\cdots +
  q^{j_k}}{1+q+\cdots+q^{i-1}} = \frac{1-q^{j_k+1}}{1-q^i} \ge
  1-q^{j_k+1}.
\end{equation}
Hence, using the fact that the $(p_i(i))$ are independent, we may
combine \eqref{eq:D_k_def} and \eqref{eq:drift_prob} to deduce that
conditioned on $E_{I,J}$, the $(D_k)$ are independent and each $D_k$
stochastically dominates a binomial random variable with
$i_{k+1}-i_k-1$ trials and success probability $1-q^{j_k+1}$. In
particular,
\begin{equation*}
  \P(\cap_{k=1}^{m-1} \{D_k<j_{k+1}-j_k\}\,|\, E_{I,J}) \le \prod_{k=1}^{m-1}\P(X_k < j_{k+1}-j_k),
\end{equation*}
where $X_k\sim \bin(i_{k+1}-i_{k}-1,1-q^{j_k+1})$. Combined with
\eqref{eq:assignment_prob_bound} and
\eqref{eq:E_I_D_k_reinterpretation} this proves the proposition.
\end{proof}

As the next step in using a union bound over the sequences $I$ and
$J$, we continue by performing the summation over $I$.


\begin{proposition}\label{prop:sum-I-for_LDS}
For any $m\ge 2$ and $J\in J_m'(0)$ we have
\begin{equation*}
\sum_{I\in I_m'(\lfloor\frac{1}{1-q}\rfloor)} \p(E_{I,J}\cap E_I)\le
n(C(1-q))^m q^{\sum_{k=1}^m j_k} \prod_{k=1}^{m-1}
\frac{j_{k+1}-j_k}{1-q^{j_k+1}}.
\end{equation*}
\end{proposition}

\begin{proof}
Comparing the result of the proposition with
Proposition~\ref{prop:prob-I-decreases} we see it suffices to show
that
\begin{equation*}
    \sum_{I_m'(\lfloor\frac{1}{1-q}\rfloor)} \prod_{k=1}^{m-1}\p(X_k(I)<j_{k+1}-j_k)\le n \prod_{k=1}^{m-1}
    \frac{j_{k+1}-j_k}{1-q^{j_k+1}},
  \end{equation*}
where $X_k(I)\sim \bin(i_{k+1}-i_{k}-1,1-q^{j_k+1})$. We change
variables, transforming the vector $I=(i_1,\ldots, i_m)$ to the
vector $D=(i_1,d_1,\ldots, d_{m-1})$ via the mapping
\begin{equation*}
  d_k := i_{k+1}-i_k.
\end{equation*}
Observing that this transformation is one-to-one, we have
\begin{equation*}
  \sum_I \prod_{k=1}^{m-1}\p(X_k(I)<j_{k+1}-j_k)\le \sum_D
  \prod_{k=1}^{m-1}\p(X_k(D)<j_{k+1}-j_k),
\end{equation*}
where the sum is over all integer vectors $D$ satisfying $1\le
i_1\le n$ and $d_k\ge 1$ for $1\le k\le m-1$, and where $X_k(D)\sim
\bin(d_k-1,1-q^{j_k+1})$. We continue by observing that the product
does not depend on $i_1$, and further observing that the sum of
products becomes a product of sums since the factors involve
different $d_k$'s, whence
\begin{equation*}
  \sum_D
  \prod_{k=1}^{m-1}\p(X_k(D)<j_{k+1}-j_k) \le n \prod_{k=1}^{m-1}
  \left[\sum_{d=1}^\infty \P(X_k(d)<j_{k+1}-j_k)\right],
\end{equation*}
where $X_k(d)\sim \bin(d-1,1-q^{j_k+1})$. Applying
Lemma~\ref{lem:sum_bin_prob} we conclude that
\begin{equation*}
  \sum_{d=1}^\infty \P(X_k(d)<j_{k+1}-j_k) =
  \frac{j_{k+1}-j_k}{1-q^{j_k+1}},
\end{equation*}
and the proposition follows.
\end{proof}
We next perform the summation over $J$. This is best done separately
over two regimes. To deal with certain edge cases later in the
proof, we extend our previous definitions by setting
$J_0(M):=\{\emptyset\}, J_0'(M):=\{\emptyset\},
I_0'(M):=\{\emptyset\}$, for integer $M\ge 1$, and setting
$\P(E_{I,J})=\P(E_I) = 1$ whenever $I=J=\emptyset$. We also adopt
the convention that $0^0$ is $1$.
\begin{proposition}\label{prop:final_I_J_summation}
There exists an absolute constant $C_1>0$ such that for any integer
$m\ge 0$ we have
\begin{equation*}
\sum_{J\in J_m(\lfloor \frac{1}{2(1-q)}\rfloor)}\sum_{I\in
I_m'(\lfloor\frac{1}{1-q}\rfloor)} \p(E_{I,J}\cap E_I)\le
n^2\left(\frac{C_1}{(1-q)m^2}\right)^m,
\end{equation*}
and
\begin{equation*}
\sum_{J\in J_m'(\lfloor \frac{1}{2(1-q)}\rfloor)}\sum_{I\in
I_m'(\lfloor\frac{1}{1-q}\rfloor)} \p(E_{I,J}\cap E_I)\le
\frac{n(C_1(1-q))^m q^{m(m-1)/2}}{(m'(1-q))^{2m'}},
\end{equation*}
where $m':=\min(m,\lfloor\frac{1}{1-q}\rfloor)$.
\end{proposition}
\begin{proof}
The cases that $m\in\{0,1\}$ follow trivially since the right-hand
side of the above inequalities is larger than $1$ when $C_1$ is
sufficiently large. Thus we assume that $m\ge 2$.
  The relation
  \begin{equation*}
    1-q^a\ge \frac{(1-q)a}{1+(1-q)a}
  \end{equation*}
  holds for any $a\ge0$. Hence by Proposition~\ref{prop:sum-I-for_LDS}
  we have
  \begin{align*}
    \sum_{J\in J_m(\lfloor \frac{1}{2(1-q)}\rfloor)}&\sum_{I\in I_m'(\lfloor\frac{1}{1-q}\rfloor)}
\p(E_{I,J}\cap E_I) \le n(C(1-q))^m \sum_{J\in J_m(\lfloor
\frac{1}{2(1-q)}\rfloor)} \prod_{k=1}^{m-1}
\frac{j_{k+1}-j_k}{1-q^{j_k+1}} \le\\
&\le nC^m \sum_{J\in J_m(\lfloor \frac{1}{2(1-q)}\rfloor)}
\prod_{k=1}^{m-1} \frac{j_{k+1}-j_k}{j_k+1}.
  \end{align*}
  Thus, noting that $\log(\lfloor
  \frac{1}{2(1-q)}\rfloor+1)\le n$, the first part of the proposition follows from
  Corollary~\ref{cor:sum_ind}.

  Similarly,
  \begin{align*}
    \sum_{J\in J_m'(\lfloor \frac{1}{2(1-q)}\rfloor)}&\sum_{I\in I_m'(\lceil\frac{1}{1-q}\rceil)}
\p(E_{I,J}\cap E_I) \le n(C(1-q)^m \sum_{J\in J_m'(\lfloor
\frac{1}{2(1-q)}\rfloor)} q^{\sum_{k=1}^m j_k}\prod_{k=1}^{m-1}
\frac{j_{k+1}-j_k}{1-q^{j_k+1}} \le\\
&\le n(C(1-q))^m \sum_{J\in J_m'(\lfloor \frac{1}{2(1-q)}\rfloor)}
q^{\sum_{k=1}^m j_k}\prod_{k=1}^{m-1}(j_{k+1}-j_k),
  \end{align*}
from which the second part of the proposition follows by applying
Lemma~\ref{lem:second_j_bound} (and bounding $q^{m\lfloor
\frac{1}{2(1-q)}\rfloor}\le 1$).
\end{proof}

\subsubsection{Proof of bound} In this section we complete the estimate of
$\P(\LDS(\pi)\ge L)$. First, if $0<q<\frac{1}{2}$, we may apply the union bound and the
second part of Proposition~\ref{prop:final_I_J_summation} in a
straightforward way to obtain that for any $L\ge 2$,
\begin{equation*}
  \P(\LDS(\pi)\ge L)\le \sum_{J\in J_L'(0)}\sum_{I\in I_L'(1)} \p(E_{I,J}\cap E_I)
  \le nC^L q^{\frac{L(L-1)}{2}}\quad\left(0<q<\frac{1}{2}\right),
\end{equation*}
proving \eqref{eq:LDS_upper_bound_in_sec} for this range of $q$ and
establishing \eqref{eq:LDS_refined_upper_bound_in_sec}.

In the rest of the section we assume $q\ge\frac{1}{2}$ (and
$q<1-\frac{2}{n}$, as before). Fix $2\le L\le n$. The union bound
yields
\begin{equation}\label{eq:LDS_pi_union_bound}
  \P(\LDS(\pi)\ge L) \le \sum_{I\in I_L'(1)}
  \P(E_I).
\end{equation}
Now, given $I=(i_1,\ldots, i_L)\in I_L'(1)$ we let $a(I)$ be the
maximal $k$ such that $i_k<\lfloor\frac{1}{1-q}\rfloor$ (or 0 if no
such $k$ exists), and let $I_1:=(i_1,\ldots,i_{a(I)})$ and
$I_2:=(i_{a(I)+1},\ldots, i_L)$ (where one of these vectors may be
empty). By the independence of induced orderings Lemma
\ref{lem:indep-seqs},
\begin{equation}\label{eq:LDS_indep_induced_orderings}
  \P(E_I)\le \p(E_{I_1}\cap E_{I_2}) = \P(E_{I_1})\P(E_{I_2}).
\end{equation}
Define, for integers $m\ge 1$ and $M\ge 2$, the set of integer
vectors
\begin{equation*}
I_m(M):=\{(i_1,\ldots, i_{m})\,:\,1\le i_1<
    i_2<\cdots<i_m<M\}.
\end{equation*}
As before, we also set $I_0(M):=\{\emptyset\}$. Plugging
\eqref{eq:LDS_indep_induced_orderings} into
\eqref{eq:LDS_pi_union_bound} and using the translation invariance
Lemma~\ref{lem:shift} (with our assumption that
$\frac{1}{1-q}<\frac{n}{2}$) we find that
\begin{align}\label{eq:LDS_break_I}
  \P&(\LDS(\pi)\ge L)\le \sum_{I\in I_L'(1)}
  \P(E_I)\le \sum_{a=0}^{\min(L,\lfloor\frac{1}{1-q}\rfloor-1)} \sum_{I_1\in
  I_a(\lfloor\frac{1}{1-q}\rfloor)} \sum_{I_2\in
  I_{L-a}'(\lfloor\frac{1}{1-q}\rfloor)} \P(E_{I_1})\P(E_{I_2})\le\nonumber\\
  &\le \sum_{a=0}^{\min(L,\lfloor\frac{1}{1-q}\rfloor-1)}\left(\sum_{I\in
  I_a(\lfloor\frac{1}{1-q}\rfloor)} \P(E_I)\right) \left(\sum_{I\in
  I_{L-a}'(\lfloor\frac{1}{1-q}\rfloor)} \P(E_I)\right)\le\nonumber\\
  &\le \sum_{a=0}^{\min(L,\lfloor\frac{1}{1-q}\rfloor-1)}\left(\sum_{I\in
  I_a'(\lfloor\frac{1}{1-q}\rfloor)} \P(E_I)\right) \left(\sum_{I\in
  I_{L-a}'(\lfloor\frac{1}{1-q}\rfloor)} \P(E_I)\right).
\end{align}
Our next task is to estimate the first factor in the above product
for a fixed $0\le a\le \min(L,\lfloor\frac{1}{1-q}\rfloor-1)$. Using
the union bound,
\begin{equation*}
  \sum_{I\in
  I_a'(\lfloor\frac{1}{1-q}\rfloor)} \P(E_I)\le \sum_{J\in J_a'(0)}\sum_{I\in
  I_a'(\lfloor\frac{1}{1-q}\rfloor)} \p(E_{I,J}\cap E_I).
\end{equation*}
Now, given $J=(i_1,\ldots, i_a)\in J_a'(0)$ we let $b(J)$ be the
maximal $k$ such that $j_k<\lfloor\frac{1}{2(1-q)}\rfloor$ (or 0 if
no such $k$ exists), let $I^1:=(i_1,\ldots,i_{b(J)})$,
$I^2:=(i_{b(J)+1},\ldots, i_L)$, $J^1:=(j_1,\ldots,j_{b(J)})$ and
$J^2:=(j_{b(J)+1},\ldots, j_L)$ (where any of these vectors may be
empty). By Fact~\ref{fact:indep-of-orderings-on-blocks}, the event
$E_{I^1,J^1}\cap E_{I^1}$ is a function of $(p_i(i))$ for $i\le
i_{b_J}$, and the event $E_{I^2, J^2}\cap E_{I^2}$ is a function of
$(p_i(i))$ for $i>i_{b_J}$. Since the $(p_i(i))$ are independent we
obtain
\begin{equation*}
  \P(E_{I,J}\cap E_I)\le \P(E_{I^1,J^1}\cap E_{I^1}\cap E_{I^2, J^2}\cap E_{I^2})= \P(E_{I^1,J^1}\cap E_{I^1})\P(E_{I^2, J^2}\cap E_{I^2}).
\end{equation*}
Thus, in a similar way to \eqref{eq:LDS_break_I}, we obtain
\begin{align}\label{eq:p_I_decreasing_estimate}
  &\sum_{I\in
  I_a'(\lfloor\frac{1}{1-q}\rfloor)} \P(E_I)\le\nonumber\\
  &\le
  \sum_{b=0}^{a} \left(\sum_{J\in J_b(\lfloor\frac{1}{2(1-q)}\rfloor)}\sum_{I\in
  I_b'(\lfloor\frac{1}{1-q}\rfloor)} \P(E_{I,J}\cap E_I)\right) \left(\sum_{J\in J_{a-b}'(\lfloor\frac{1}{2(1-q)}\rfloor)}\sum_{I\in
  I_{a-b}'(\lfloor\frac{1}{1-q}\rfloor)} \P(E_{I,J}\cap E_I)\right).
\end{align}
To estimate this product, we let $C_1>0$ be the constant from
Proposition~\ref{prop:final_I_J_summation} and define, for $m\ge 0$,
\begin{align*}
  f(m)&:=\left(\frac{C_1}{(1-q)m^2}\right)^m,\\
  g(m)&:=\frac{(C_1(1-q))^m q^{m(m-1)/2}}{(m'(1-q))^{2m'}},
\end{align*}
where $m':=\min(m,\lfloor\frac{1}{1-q}\rfloor)$. It is immediate
that $g(m)\le f(m)$ if $m\le \lfloor\frac{1}{1-q}\rfloor$. In
addition, as in the last inequality of
\eqref{eq:comparison_for_x_to_x},
\begin{equation}\label{eq:f_multiplicative}
  f(k)f(m)\le C^{k+m}f(k+m)
\end{equation}
for $m,k\ge 0$. Now, applying
Proposition~\ref{prop:final_I_J_summation} to the sums in
\eqref{eq:p_I_decreasing_estimate} and recalling that
$a<\lfloor\frac{1}{1-q}\rfloor$, we deduce
\begin{align}\label{eq:first_I_in_I_a_estimate}
  \sum_{I\in
  I_a'(\lfloor\frac{1}{1-q}\rfloor)} \P(E_I)\le
  n^3\sum_{b=0}^{a} f(b)g(a-b)\le n^3\sum_{b=0}^{a}
  f(b)f(a-b)\le
  C^a n^3 f(a).
\end{align}
In a completely analogous fashion, we estimate the second factor in
\eqref{eq:LDS_break_I} by
\begin{equation}\label{eq:second_I_in_I_a_estimate}
  \sum_{I\in
  I_{L-a}'(\lfloor\frac{1}{1-q}\rfloor)} \P(E_I)\le
  n^3\sum_{b=0}^{\min(L-a,\lfloor\frac{1}{2(1-q)}\rfloor-1)}
  f(b)g(L-a-b).
\end{equation}
Plugging \eqref{eq:first_I_in_I_a_estimate} and
\eqref{eq:second_I_in_I_a_estimate} into \eqref{eq:LDS_break_I} and
again using \eqref{eq:f_multiplicative} we finally arrive at
\begin{align}\label{eq:final_LDS_inequality}
  \P(\LDS(\pi)\ge L)&\le n^6
  \sum_{a=0}^{\min(L,\lfloor\frac{1}{1-q}\rfloor-1)}\
  \sum_{b=0}^{\min(L-a,\lfloor\frac{1}{2(1-q)}\rfloor-1)}
  C^a f(a)f(b)g(L-a-b)\le\nonumber\\
  &\le C^L n^8\max_{0\le m\le
  \min\left(L,\frac{3}{2(1-q)}\right)} f(m)g(L-m).
\end{align}
It remains to estimate $f(m)g(L-m)$. It is simple to see that
$g(m)\le C^m f(m)$ when $m\le \frac{3}{1-q}$ since for such $m$,
$\frac{(m(1-q))^{2m}}{(m'(1-q))^{2m'}}\le C^m$. Hence, if we assume
that $L\le \frac{3}{1-q}$ we obtain by \eqref{eq:f_multiplicative}
that
\begin{equation*}
  \P(\LDS(\pi)\ge L) \le C^L n^8f(L) = n^8
  \left(\frac{C}{(1-q)L^2}\right)^L\qquad \left(L\le
\frac{3}{1-q}\right),
\end{equation*}
proving \eqref{eq:LDS_upper_bound_in_sec} in this case. We continue
to the case $L>\frac{3}{1-q}$. For all $0\le m\le \frac{3}{2(1-q)}$
we have $L-m\ge \frac{1}{1-q}$, $((1-q)^2m^2)^{-m}\le
C^{\frac{1}{1-q}}$ (by differentiating with respect to $m$) and
$q^{-m}\le C$ (by our assumption that $q\ge \frac{1}{2}$). Thus, for
these $m$,
\begin{align*}
  f(m)g(L-m) &= \left(\frac{C}{(1-q)m^2}\right)^m
  \frac{(C(1-q))^{L-m}q^{(L-m)(L-m-1)/2}}{((L-m)'(1-q))^{2(L-m)'}}=\\
  &=\frac{q^{-mL+m/2} }{((1-q)^2m^2)^m
  (\lfloor\frac{1}{1-q}\rfloor
  (1-q))^{2\lfloor\frac{1}{1-q}\rfloor}} C^{L-m}(1-q)^L q^{\frac{L(L-1)}{2}}\le\\
  &\le (C(1-q))^L q^{\frac{L(L-1)}{2}} \qquad\qquad\qquad\qquad\quad \left(L>
\frac{3}{1-q}\right).
\end{align*}
Using this estimate in \eqref{eq:final_LDS_inequality} finishes the
proof of \eqref{eq:LDS_upper_bound_in_sec}.

\subsection{A lower bound on $\p(\LDS(\pi) \ge L)$}\label{sec:lds_lb}

In this section we prove part \eqref{part:LDS_greater_than_L_lb} of
Theorem \ref{thm:lds-largedev} by establishing the bound
\eqref{eq:LDS_greater_than_L_lb}, giving a lower bound on the
probability of a long decreasing subsequence. We give two bounds,
one which applies only when the length $L$ of the subsequence
satisfies $C(1-q)^{-1/2} < L < (1-q)^{-1}$, and one which applies
for all $L$. The first bound is superior to the second in the cases
to which it applies.


\begin{proposition}
Let $n\ge 1$, $\frac{1}{2}\le q\le 1-\frac{4}{n}$ and
$\pi\sim\mu_{n,q}$. There exist absolute constants $C,c>0$ such that
for all integer $L$ satisfying
\begin{equation}\label{eq:conditions_on_L_LDS_lower_bound}
\frac{C}{\sqrt{1-q}} \le L \le \frac{1}{1-q}
\end{equation}
we have
\[
\p(\LDS(\pi) \ge L) \ge  1-
\left(1-\left(\frac{c}{(1-q)L^2}\right)^L\right)^{\left\lfloor\frac{n(1-q)}{4}\right\rfloor}.
\]
\end{proposition}
\begin{proof}
Fix an integer $L$ satisfying
\eqref{eq:conditions_on_L_LDS_lower_bound} with the constant $C$
large enough and the constant $c$ small enough for the following
calculations. Using Corollary~\ref{cor:perms-that-are-mallows}, we
couple $\pi$ with the $q$-Mallows process so that
\begin{equation}\label{eq:pi_p_coupling_LDS_LB}
  \pi(i) = n+1 - p_n(i)\quad\text{ for all $1\le i\le n$}.
\end{equation}
For $1\le k\le L$, define the set of integers
\begin{equation*}
O_k:=\left[1+\frac{3(k-1)}{(1-q)L},\,
1+\frac{3(k-1)+1}{(1-q)L}\right] \cap \mathbb{Z}.
\end{equation*}
Observe that
\begin{equation}\label{eq:lower_bound_O_k}
  |O_k|\ge \left\lfloor\frac{1}{(1-q)L}\right\rfloor \ge 1
\end{equation}
by \eqref{eq:conditions_on_L_LDS_lower_bound}. Let
\begin{equation*}
  N:=\left\lfloor\frac{n(1-q)}{4}\right\rfloor
\end{equation*}
and observe that $N\ge 1$ by our assumption on $q$. For $1\le j\le
N$ and $1\le k\le L$ define the set of integers and the event
\begin{align*}
  I_{j,k}&:=\left[\frac{j+2}{1-q} + \frac{k-1}{(1-q)L},\, \frac{j + 2}{1-q} + \frac{k}{(1-q)L}\right)\cap\mathbb{Z},\\
  E_{j,k}&:=\{\exists\, i\in I_{j,k}\text{ such that } p_i(i)\in
  O_k\}.
\end{align*}
Observe that $\max_{j,k} (\max(I_{j,k}))\le n$ by our assumption on
$q$. Our strategy for proving a lower bound for $\p(\LDS(\pi)\ge L)$
is based on the following containment of events,
\begin{equation}\label{eq:LDS_LB_containment_of_events}
  \{\LDS(\pi)\ge L\}\supseteq \cup_{j=1}^N \cap_{k=1}^L E_{j,k}.
\end{equation}
Let us prove this relation. Suppose that $\cap_{k=1}^L E_{j,k}$
occurs for some $1\le j\le N$. For each $k$, let $i_{j,k}\in
I_{j,k}$ be such that $p_{i_{j,k}}(i_{j,k})\in O_k$. For each $1\le
k\le L-1$ we have by \eqref{eq:Mallows_iteration_procedure} that
\begin{align*}
  p_{i_{j,k+1}}(i_{j,k})&\le p_{i_{j,k}}(i_{j,k}) + i_{j,k+1} - i_{j,k}
  \le \max(O_k) + \max(I_{j,k+1}) - \min(I_{j,k}) <\\
  &< 1+\frac{3(k-1)+1}{(1-q)L} + \frac{2}{(1-q)L} \le \min(O_{k+1}) \le
  p_{i_{j,k+1}}(i_{j,k+1}).
\end{align*}
This implies, again by \eqref{eq:Mallows_iteration_procedure}, that
$p_n(i_{j,k})<p_n(i_{j,k+1})$ and hence, by
\eqref{eq:pi_p_coupling_LDS_LB}, that $\pi(i_{j,k})>\pi(i_{j,k+1})$.
Thus the event $\{\LDS(\pi)\ge L\}$ occurs.

We continue to establish a lower bound for the probability of the
event on the right-hand side of
\eqref{eq:LDS_LB_containment_of_events}. Observe that the sets
$(I_{j,k})$ are pairwise-disjoint. Hence, since the random variables
$(p_i(i))$ are independent, we have
\begin{equation}\label{eq:LDS_pi_lower_bound_with_E_j_k}
  \P(\LDS(\pi)\ge L) \ge \P\left(\cup_{j=1}^N \cap_{k=1}^L E_{j,k}\right) = 1 -
  \prod_{j=1}^N \P\left(\cup_{k=1}^L E_{j,k}^c\right) = 1 - \prod_{j=1}^N \left(1 -
  \prod_{k=1}^L \P(E_{j,k})\right).
\end{equation}
Now, to estimate $\P(E_{j,k})$, observe first that $\max_k
(\max(O_k))\le \frac{3}{1-q}\le \min_{j,k} (\min(I_{j,k}))$ by
\eqref{eq:conditions_on_L_LDS_lower_bound}. In addition, it follows
from our assumption that $q\ge \frac{1}{2}$ that $\min_{m\in O_k} (
q^{m-1})\ge c>0$. Thus, by \eqref{eq:Mallows_folder_dist} and
\eqref{eq:lower_bound_O_k}, for each $j$ and $k$,
\begin{align*}
  \P(E_{j,k}) &= \P(\cup_{i\in I_{j,k}} \{p_i(i)\in O_k\}) = 1 -
  \prod_{i\in I_{j,k}} \left(1 - \p(p_i(i)\in O_k)\right) = 1 -
  \prod_{i\in I_{j,k}} \left(1 - \frac{(1-q)\sum_{m\in O_k} q^{m-1}}{1-q^i}\right)
  \ge\\
  &\ge
  1 - \prod_{i\in I_{j,k}} \left(1 - c(1-q)|O_k|\right) \ge 1 - \prod_{i\in I_{j,k}} \left(1 - \frac{c}{L}\right) = 1 -
  \left(1 - \frac{c}{L}\right)^{|I_{j,k}|},
\end{align*}
and, since $\max_{j,k}|I_{j,k}|\le
\left\lceil\frac{1}{(1-q)L}\right\rceil\le CL$ and
$\min_{j,k}|I_{j,k}|\ge \left\lfloor\frac{1}{(1-q)L}\right\rfloor\ge
1$ by \eqref{eq:conditions_on_L_LDS_lower_bound}, we may continue
the last inequality to obtain
\begin{equation*}
  \P(E_{j,k}) \ge \frac{c|I_{j,k}|}{L} \ge \frac{c}{(1-q)L^2}.
\end{equation*}
Plugging this estimate into \eqref{eq:LDS_pi_lower_bound_with_E_j_k}
finishes the proof of the proposition.
\end{proof}

We now prove our second bound, which applies to all $L$. The
strategy in this bound is to simply look for a decreasing
subsequence composed of \emph{consecutive} elements.
\begin{proposition}
Let $n\ge 1$, $0< q< 1$ and $\pi\sim\mu_{n,q}$. Then for all integer
$L\ge 2$,
\[
\p(\LDS(\pi) \ge L) \ge 1 - \left(1 - q^{\frac{L(L-1)}{2}}
(1-q)^L\right)^{\lfloor \frac{n}{L}\rfloor}.
\]
\end{proposition}
\begin{proof}
Let $N:=\left\lfloor\frac{n}{L}\right\rfloor$ and define the sets
$I_i:=\{1+(i-1)L,\, 2+(i-1)L,\ldots, iL\}$ for $1\le i\le N$. Define
the events
\begin{equation*}
  E_i:=\{\pi_{I_i}\text{ is the reversed identity}\},\quad(1\le i\le
  N).
\end{equation*}
Then we have the following containment of events,
\begin{equation*}
  \{\LDS(\pi)\ge L\}\supseteq\cup_{i=1}^N E_i.
\end{equation*}
The events $(E_i)$ are independent by Lemma~\ref{lem:indep-seqs},
and have the same probability by
Corollary~\ref{cor:induced_permutation_Mallows}. Hence,
\begin{equation}\label{eq:LDS_pi_lower_bound_with_E_i}
  \P(\LDS(\pi)\ge L)\ge \P[\cup_{i=1}^N E_i] = 1 - \P(\cap_{i=1}^N
  E_i^c) = 1 - \prod_{i=1}^N (1 - \P(E_i)) = 1 - (1 - \p(E_1))^N.
\end{equation}
Since the reversed identity permutation on $L$ elements has
$L(L-1)/2$ inversions, we conclude by
Corollary~\ref{cor:induced_permutation_Mallows},
\eqref{eq:mu_n_q_def} and \eqref{eq:Z_formula} that
\begin{equation*}
  \P(E_1) = \frac{q^{\frac{L(L-1)}{2}}}{Z_{L,q}} =
  q^{\frac{L(L-1)}{2}}(1-q)^L\prod_{i=1}^L\frac{1}{1-q^i}\ge
  q^{\frac{L(L-1)}{2}}(1-q)^L.
\end{equation*}
Plugging this estimate into \eqref{eq:LDS_pi_lower_bound_with_E_i}
finishes the proof of the proposition.
\end{proof}

\subsection{Upper bound on $\P(\LDS(\pi) <
L)$}\label{sec:erdos-szekeres}

In this section use a classical combinatorial result of Erd\"{o}s
and Szekeres to show that $\LDS(\pi)$ is not likely to be very
small, proving the bound~\eqref{eq:LDS_less_than_L} of
Theorem~\ref{thm:lds-largedev}. The following well-known theorem is
a consequence of the pigeonhole principle.
\begin{theorem}[Erd\"{o}s-Szekeres] \label{thm:Erd-Sze}Let $r,s\ge 1$ be any
  integers such that
  $n>(r-1)(s-1)$. Then a permutation of length $n$ contains either an
  increasing subsequence of length $r$ or a decreasing subsequence of
  length $s$.
\end{theorem}
The theorem allows us to translate the large deviation bound on
$\LIS(\pi)$ given by the upper bound of
\eqref{eq:LIS_greater_than_L} into an upper bound on the probability
that $\LDS(\pi)$ is very small.
\begin{proposition} \label{thm:lower-bnd-lds}
There are absolute constants $C,c>0$ for which, if $n\ge 1$,
$\frac{1}{2} \le q \le 1-\frac{4}{n}$ and $\pi \sim \mu_{n,q}$, then
for all integer $2\le L<\frac{c}{\sqrt{1-q}}$,
\[
\p(\LDS(\pi)<L) \le \left(C(1-q)L^2\right)^{\frac{n}{L}}.
\]
\end{proposition}
\begin{proof}
  By Theorem~\ref{thm:Erd-Sze}, for any integer $L\ge 2$,
  $\{\LDS(\pi)<L\}\subseteq\{\LIS(\pi)\ge\lceil\frac{n}{L-1}\rceil\}$.
  If, in addition, $L<\frac{c}{\sqrt{1-q}}$, we may apply the upper bound of \eqref{eq:LIS_greater_than_L}
  to obtain
  \begin{equation*}
    \P(\LDS(\pi)<L)\le \P\left(\LIS(\pi)\ge \left\lceil\frac{n}{L-1}\right\rceil\right)\le
    \min\Bigg(1, \left(\frac{C(1-q)n^2}{\lceil\frac{n}{L-1}\rceil^2}
    \right)^{\lceil\frac{n}{L-1}\rceil}\Bigg)\le
    \left(C(1-q)L^2\right)^{\frac{n}{L}}.\qedhere
  \end{equation*}
\end{proof}

It is possible to use Theorem~\ref{thm:Erd-Sze} in the other
direction as well, to prove upper bounds for $\P(\LIS(\pi)<L)$ via
upper bounds on $\P(\LDS(\pi)\ge L)$. For certain ranges of $n,q$
and $L$ this provides an improvement over
\eqref{eq:LIS_less_than_L_UB}. For instance, when $q =
1-\frac{4}{n}$ and $L=4$, the bound \eqref{eq:LIS_less_than_L_UB}
shows that $\P(\LIS(\pi)<4)\le e^{-cn}$, whereas
Theorem~\ref{thm:Erd-Sze} and the
bound~\eqref{eq:LDS_greater_than_L} show that $\P(\LIS(\pi)<4)\le
(C/n)^{cn}$.
We do not pursue a systematic study
of the ranges in which each of the bounds is optimal, nor do we
prove a matching lower bound for $\P(\LIS(\pi)<L)$ here. We direct
the reader to Section~\ref{sec:open-ques} for a discussion of these
open problems.

%
%

\subsection{Bounds for $\E(\LDS(\pi))$}\label{sec:E(LDS)}
In this section we prove Theorem~\ref{thm:E(LDS)}. The proof
requires also Proposition~\ref{prop:identity_for_small_q} which we
now establish.
\begin{proof}[Proof of Proposition~\ref{prop:identity_for_small_q}]
Fix $n\ge 2$ and $0<q\le\frac{1}{n}$. Since $1-x\le \exp(-x)$ for
all $x$ and $1-x\ge\exp(-Cx)$ for $0<x\le\frac{1}{2}$,
\begin{align*}
  \max(c, 1-Cnq)\le &(1-q)^n\le 1-cnq,\\
  1-Cq\le &\prod_{i=1}^n(1-q^i) \le 1.
\end{align*}
Now, letting $\pi\sim\mu_{n,q}$, \eqref{eq:mu_n_q_def} and
\eqref{eq:Z_formula} show that
\begin{equation*}
  \P(\pi\text{ is not the identity}) = 1 -
  \frac{(1-q)^n}{\prod_{i=1}^n(1-q^i)} \le 1 - (1-q)^n\le Cnq,
\end{equation*}
and, if $n$ is sufficiently large,
\begin{equation*}
  \P(\pi\text{ is not the identity}) = 1 -
  \frac{(1-q)^n}{\prod_{i=1}^n(1-q^i)} \ge 1 - \frac{1-cnq}{1-Cq}\ge
  cnq.
\end{equation*}
To obtain the lower bound for small $n$, let $\sigma\in S_n$ be any
permutation with $\inv(\sigma)=1$ (here we assume $n\ge 2$). Then,
by \eqref{eq:mu_n_q_def} and \eqref{eq:Z_formula},
\begin{equation*}
  \P(\pi = \sigma) = \frac{q(1-q)^n}{\prod_{i=1}^n (1-q^i)} \ge
  c q.\qedhere
\end{equation*}
\end{proof}
We now establish Theorem \ref{thm:E(LDS)} using the large deviation
inequalities proved above. We consider separately several different
regimes depending on the relative sizes of $q$ and $n$.
\begin{proof}[Proof of Theorem \ref{thm:E(LDS)}]
The constants $C_0, c_0, c_1$ appearing in the proof below are fixed
positive constants with $C_0$ taken large enough for our
calculations and $c_0,c_1$ taken small enough for our calculations.
Also, we will assume throughout the proof of
\eqref{eq:E_LDS_large_q} that $n\ge C$ for some constant $C$,
sufficiently large for our calculations. This is without loss of
generality since the theorem bounds $\E(\LDS(\pi))$ up to constants,
and we may always adjust these constants so that
\eqref{eq:E_LDS_large_q} applies also to the case $n\le C$.

\begin{enumerate}[{\em (i)}]
\item Suppose $1-\frac{C_0}{(\log n)^2}\le q\le1-\frac{4}{n}$.

Let $L^* := \frac{c}{\sqrt{1-q}}$, for a sufficiently small $c$.
Then, by \eqref{eq:LDS_less_than_L},
\begin{align}
\E(\LDS(\pi)) \ge L^*\p(\LDS(\pi) \ge L^*)  = L^*
\left(1-\p\left(\LDS(\pi) < \lceil L^* \rceil\right)\right)  \ge L^*
\left(1-\left(C_{13}(1-q)\lceil L^*\rceil^2 \right)^{n/\lceil
L^*\rceil}\right) \ge \frac{L^*}{2}, \nonumber
\end{align}
where $C_{13}$ is the constant $C$ appearing in
\eqref{eq:LDS_less_than_L}. Now let $L^* := \frac{C}{\sqrt{1-q}}$
where $C$ is chosen large enough so that, using the lower bound on
$q$, $L^* \geq 9\log_2 n$. Therefore, by the first bound of
\eqref{eq:LDS_greater_than_L},
\begin{align}
\E(\LDS(\pi)) \le L^* + n\p(\LDS(\pi) > L^*) \le L^* + n\p(\LDS(\pi)
\ge \lceil L^* \rceil) \le L^*+ n^9 \left(\frac{1}{2}\right)^{L^*}
\le L^*+1. \nonumber
\end{align}

\item Suppose $1-\frac{c_0\log\log n}{\log n }\le q\le 1-\frac{C_0}{(\log
n)^2}$. Note that this is only part of the range of $q$'s in the
second part of the theorem. The other part will be treated later.

Let $L^* := \frac{c \log n}{\log((1-q) (\log n)^2)}$ for a
sufficiently small $c$. We claim that
\begin{equation}\label{eq:L_star_restrictions}
  \frac{C_{12}}{\sqrt{1-q}} \le L^* \le \frac{1}{2(1-q)}
\end{equation}
where $C_{12}$ is the constant $C$ appearing in the first part of
inequality \eqref{eq:LDS_greater_than_L_lb}. To see this, observe
that $L^*\ge\frac{C_{12}}{\sqrt{1-q}}$ is equivalent to
\begin{equation*}
  c\sqrt{(1-q)(\log n)^2}\ge C_{12}\log((1-q)(\log n)^2),
\end{equation*}
which holds when $(1-q)(\log n)^2$ is at least a sufficiently large
constant. This follows from the upper bound on $q$ by taking $C_0$
large enough. Similarly, observe that $L^*\le\frac{1}{2(1-q)}$ is
equivalent to
\begin{equation*}
  2c(1-q)(\log n)^2\le\log((1-q)(\log n)^2) \log n,
\end{equation*}
which holds when $e\le (1-q)(\log n)^2\le \frac{1}{2c}\log
n\cdot\log \log n$, which follows from our restrictions on $q$ by
taking $C_0$ large enough and $c_0$ small enough. This establishes
\eqref{eq:L_star_restrictions}.

Next, we claim that
\begin{equation}\label{eq:LDS_lower_bound_half}
\P(\LDS(\pi)\ge L^*)= \P(\LDS(\pi)\ge \lceil L^* \rceil ) \ge
\frac{1}{2}.
\end{equation}
Observing that \eqref{eq:L_star_restrictions} implies that
$\frac{C_{12}}{\sqrt{1-q}}\le \lceil L^*\rceil\le \frac{1}{1-q}$,
\eqref{eq:LDS_lower_bound_half} will follow from the first part of
\eqref{eq:LDS_greater_than_L_lb} if we show that
\begin{equation*}
 \left\lfloor \frac{n(1-q)}{4}\right\rfloor \left(\frac{c_{12}}{(1-q)\lceil L^* \rceil^2}\right)^{\lceil L^* \rceil} \ge \log 2,
\end{equation*}
where $c_{12}$ is the constant $c$ appearing in the first part of
\eqref{eq:LDS_greater_than_L_lb}. 
Recalling our bounds on $q$, it suffices to show that
\begin{equation}\label{eq:LDS_lower_bound_half_reduction}
  \frac{n}{(\log n)^2} \exp\left(-\lceil L^*\rceil\log\left(\frac{(1-q)\lceil L^*\rceil^2}{c_{12}}\right)\right)\ge \frac{8\log 2}{C_0}.
\end{equation}
Now, taking the constant in the definition of $L^*$ small enough, we
have $(1-q)\lceil L^*\rceil^2/c_{12}\le (1-q)(\log n)^2$. Therefore,
again taking the constant in the definition of $L^*$ small enough,
$\lceil L^*\rceil \log((1-q)\lceil L^*\rceil^2/c_{12})\le
\frac{1}{2}\log n$. This establishes
\eqref{eq:LDS_lower_bound_half_reduction} and hence
\eqref{eq:LDS_lower_bound_half}. Finally,
\eqref{eq:LDS_lower_bound_half} implies that
\begin{equation*}
  \E(\LDS(\pi)) \ge L^* \p(\LDS(\pi) \ge L^*) \ge \frac{L^*}{2}.
\end{equation*}

Now let $L^* := \frac{C \log n}{\log((1-q) (\log n)^2)}$ for a
sufficiently large $C$. As in the proof of
\eqref{eq:L_star_restrictions}, also in this case we have $\lceil
L^*\rceil\le \frac{3}{1-q}$ if the constant $C_0$ is large enough
and the constant $c_0$ is small enough. We also have $L^*\ge 2$ by
our restrictions on $q$ and by taking the constant $C$ large enough.
Hence we may apply the first bound of \eqref{eq:LDS_greater_than_L}
and obtain the bound below, taking $C$ to be large enough
\begin{equation}\label{eq:bound}
  \P(\LDS(\pi)\ge L^*)= \P(\LDS(\pi)\ge \lceil L^* \rceil)\le n^8\left(\frac{C_{10}}{(1-q)\lceil
  L^*\rceil^2}\right)^{\lceil
  L^*\rceil},
\end{equation}
where $C_{10}$ is the constant $C$ from
\eqref{eq:LDS_greater_than_L}.
We claim that the right-hand side of \eqref{eq:bound} is at most
$\frac{1}{n}$ if the constant in the definition of $L^*$ is taken
large enough. Equivalently,
\begin{equation*}
  \lceil L^*\rceil\log((1-q)\lceil L^*\rceil^2 / C_{10})\ge 9\log n.
\end{equation*}
For this, substituting the definition of $L^*$ with a large enough
constant, it suffices to show that
\begin{equation*}
  (1-q)\lceil L^*\rceil^2 / C_{10}\ge \left((1-q)(\log n)^2\right)^{\frac{1}{2}}.
\end{equation*}
We now substitute the definition of $L^*$ in the left-hand side.
Again taking the constant $C$ large enough, the inequality reduces
to showing
\begin{equation*}
  \frac{(1-q)(\log n)^2}{(\log((1-q)(\log n)^2))^2} \ge \left((1-q)(\log
  n)^2\right)^{\frac{1}{2}}.
\end{equation*}
Denoting $y:=(1-q)(\log n)^2$, we may rewrite this as
\begin{equation*}
  y^{\frac{1}{2}}\ge (\log y)^2.
\end{equation*}
This inequality is satisfied whenever $y$ is sufficiently large, and
this condition is assured in our setting by choosing the constant
$C_0$ in the upper bound on $q$ large enough.

Finally, we conclude that
\begin{equation*}
  \E(\LDS(\pi))\le L^* + n\P(\LDS(\pi)\ge L^*) \le L^* + 1.
\end{equation*}

\item Suppose $1 - \frac{c_1(\log\log n)^2}{\log n}\le q\le 1-\frac{c_0\log\log n}{\log n}$. Continuing the previous item, the second part of the theorem
will follow by showing that for this range of $q$'s,
$\E(\LDS(\pi))\approx \frac{\log n}{\log\log n}$. Note that the
assumptions on $q$ imply that for some constants $C(c_0), c(c_1),
C(c_1)>0$ we have
\begin{align}
  c(c_1)\log \log n\le \log&\left(\frac{1}{1-q}\right)\le C(c_0)\log\log
  n,\label{eq:log_1_over_1_q_inequality}\\
  e^{-C(c_1)(1-q)}\le &q\le
  e^{-(1-q)}.\label{eq:q_inequality_LDS_bounds}
\end{align}

Let $L^* := \frac{c \log n}{\log\log n}$ for a sufficiently small
$c$. We take $n$ sufficiently large compared to $c$ so that $L^*\ge
2$. By the second part of \eqref{eq:LDS_greater_than_L_lb},
\begin{align}
  \E(\LDS(\pi))&\ge L^*\P(\LDS(\pi)\ge \lceil L^* \rceil)\ge L^*\left(1-\left(1-q^{\frac{\lceil L^* \rceil(\lceil L^* \rceil-1)}{2}}
  (1-q)^{\lceil L^* \rceil}\right)^{\left\lfloor\frac{n}{\lceil L^* \rceil}\right\rfloor}\right).\label{eq:LDS_third_part_LB}
\end{align}
Applying \eqref{eq:log_1_over_1_q_inequality} and
\eqref{eq:q_inequality_LDS_bounds}, recalling our assumptions on $q$
and taking $c$ small enough, we have
\begin{align*}
  q^{\frac{\lceil L^* \rceil (\lceil L^* \rceil-1)}{2}}(1-q)^{\lceil L^* \rceil}&\ge
  \exp\left(-C(c_1)(1-q)\frac{\lceil L^* \rceil(\lceil L^* \rceil-1)}{2}-C(c_0)\lceil L^* \rceil\log\log n
  \right) \ge  \\ & \ge
  \exp\left(-C(c_1)(1-q)(L^*)^2-2C(c_0)L^*\log\log n
  \right) \ge\\
  &\ge \exp\left(-\frac{1}{2}\log n\right) = \frac{1}{\sqrt{n}}.
\end{align*}
Substituting into \eqref{eq:LDS_third_part_LB} shows that
\begin{equation*}
  \E(\LDS(\pi))\ge L^{*}\left(1
  - \left(1 - \frac{1}{\sqrt n}\right)^{\left\lfloor
  \frac{n}{L^*}\right\rfloor}\right) \ge
  \frac{L^*}{2}.
\end{equation*}
%
%

Now, let $L^* := \frac{C \log n}{\log\log n}$ for a sufficiently
large $C$. Applying \eqref{eq:log_1_over_1_q_inequality},
\begin{equation*}
q^{\frac{\lceil L^* \rceil(\lceil L^* \rceil-1)}{2}}(1-q)^{\lceil
L^* \rceil}\le
  \exp\left(-c(c_1)\lceil L^* \rceil \log\log n
  \right) \le \exp\left(-c(c_1)L^*\log\log n
  \right) \le \frac{1}{n^{10}}.
\end{equation*}
For our choice of $L^*$ we have $L^*>\frac{3}{1-q}$ by the upper
bound on $q$. Thus, using the second part of
\eqref{eq:LDS_greater_than_L},
\begin{equation}\label{eq:LDS_third_part_UB}
  \E(\LDS(\pi))\le L^* + n\P(\LDS(\pi)\ge \lceil L^{*}\rceil)\le L^* +
  n^9(C_{10}(1-q))^{\lceil L^* \rceil}
  q^{\frac{\lceil L^* \rceil(\lceil L^* \rceil-1)}{2}}\le L^* + \frac{(C_{10})^{\lceil L^*\rceil}}{n}\le L^* + 1.
\end{equation}

%
%
%
%
%
%

\item Let $\frac{1}{n} \le q \le 1-c_1\frac{(\log\log n)^2}{\log n}$. In
this regime we have for an appropriate $C(c_1)>0$,
\begin{align*}
  &\log\left(\frac{1}{1-q}\right)\le C(c_1)\log\log
  n,\\
  &\log(1/q)\ge 1-q\ge c_1\frac{(\log \log n)^2}{\log n}.
\end{align*}

Let $L^* := c\sqrt{\frac{\log n}{\log(1/q)}}$ for a sufficiently
small $c$. If $L^*<2$ then, trivially,
\begin{equation*}
  \E(\LDS(\pi))\ge 1> \frac{1}{2}L^*.
\end{equation*}
Otherwise, assume that $L^*\ge 2$. Then, as in
\eqref{eq:LDS_third_part_LB},
\begin{align}
  \E(\LDS(\pi))&\ge L^*\P(\LDS(\pi)\ge \lceil L^* \rceil)\ge L^*\left(1-\left(1-q^{\frac{\lceil L^* \rceil(\lceil L^* \rceil-1)}{2}}
  (1-q)^{\lceil L^* \rceil}\right)^{\left\lfloor\frac{n}{\lceil L^* \rceil}\right\rfloor}\right).\label{eq:LDS_fourth_part_LB}
\end{align}
We may estimate the term on the right-hand side as
\begin{align*}
  q^{\frac{\lceil L^* \rceil (\lceil L^* \rceil-1)}{2}}(1-q)^{\lceil L^* \rceil} & \ge
  \exp\left(-\log\left(\frac{1}{q}\right)\frac{\lceil L^* \rceil(\lceil L^* \rceil-1)}{2}-C(c_1)\lceil L^* \rceil\log\log n
  \right) \ge \\
  & \ge \exp\left(-\log\left(\frac{1}{q}\right)(L^*)^2-2C(c_1)L^* \log\log n
  \right) \ge \exp(-\frac{1}{2}\log n) = \frac{1}{\sqrt{n}}.
\end{align*}
Plugging into \eqref{eq:LDS_fourth_part_LB} implies that
$\E(\LDS(\pi))\ge \frac{L^*}{2}$.

Now, let $L^* := C\sqrt{\frac{\log n}{\log(1/q)}}$ for a
sufficiently large $C$. First observe that $L^*\ge 2$ by our
assumptions on $q$ and $C$. In addition, note that when $q\ge
\frac{1}{2}$ we have $\log(1/q)\le C'(1-q)$ for some $C'>0$. It
follows that $L^*> \frac{3}{1-q}$ for our range of $q$'s. Thus,
using the second part of \eqref{eq:LDS_greater_than_L},
\begin{equation}\label{eq:LDS_fourth_part_UB}
  \E(\LDS(\pi))\le L^* + n\P(\LDS(\pi)\ge \lceil L^{*}\rceil)\le L^* +
  n^9(C_{10}(1-q))^{\lceil L^* \rceil}
  q^{\frac{\lceil L^* \rceil(\lceil L^* \rceil-1)}{2}}.
\end{equation}
Then, using that $L^*-1\ge \frac{L^*}{2}$ since $L^*\ge 2$,
\begin{equation*}
  q^{\frac{\lceil L^* \rceil (\lceil L^* \rceil-1)}{2}}(1-q)^{\lceil L^* \rceil}\le
  \exp\left(-\log\left(\frac{1}{q}\right)\frac{\lceil L^* \rceil(\lceil L^* \rceil-1)}{2}
  \right) \le \exp\left(-\log\left(\frac{1}{q}\right)\frac{(L^*)^2}{4}
  \right)\le \exp(-10\log n) = \frac{1}{n^{10}}.
\end{equation*}
Plugging into \eqref{eq:LDS_fourth_part_UB} and using our assumption
on $q$,
\begin{equation*}
  \E(\LDS(\pi))\le L^* + \frac{(C_{10})^{\lceil L^*\rceil}}{n}\le L^* + 1.
\end{equation*}
\item Let $n\ge 2$ and $0<q\le \frac{1}{n}$.

By the second part of \eqref{eq:LDS_greater_than_L_lb},
\begin{align*}
  \E(\LDS(\pi))-1&=\sum_{L=2}^n \P(\LDS(\pi)\ge L)\ge \P(\LDS(\pi)\ge
  2)\ge \left(1-\left(1-q^{\frac{2(2-1)}{2}}
  (1-q)^{2}\right)^{\left\lfloor\frac{n}{2}\right\rfloor}\right)
  \ge\\
  &\ge \left(1-\exp\left(-\left\lfloor\frac{n}{2}\right\rfloor
  q(1-q)^2\right)\right)\ge cnq,
\end{align*}
where in the second to last inequality we used the fact that when
$0\le x\le \frac{1}{2}$, $\exp(-x)\le 1-cx$ for some $c>0$. If $2\le
n\le 3$, Proposition~\ref{prop:identity_for_small_q} implies that
\begin{equation*}
  \E(\LDS(\pi))-1\le 2\P(\pi\text{ is not the
  identity})\le cnq.
\end{equation*}
Otherwise, if $n\ge 4$, we may use the second part of
\eqref{eq:LDS_greater_than_L} along with
Proposition~\ref{prop:identity_for_small_q} to obtain
\begin{align*}
  \E(\LDS(\pi))-1&=\sum_{L=2}^n \P(\LDS(\pi)\ge L)\le 3\P(\LDS(\pi)\ge 2) + n\P(\LDS(\pi)\ge 5)\\
  &\le 3\P(\pi\text{ is not the
  identity}) + n\P(\LDS(\pi)\ge 5)\le Cnq + Cn^9q^{10}\le Cnq +
  Cq.\qedhere
\end{align*}
\end{enumerate}
\end{proof}

\section{Variance of the length of monotone subsequences}\label{sec:lis-var}
In this section we prove Proposition~\ref{prop:var_prop}, giving a
bound on the variance of $\LIS(\pi)$ and a Gaussian tail bound for
it.

Fix $n\ge 1, q>0$, and let $(p_i)$ be the $q$-Mallows
  process. Since $p_n\sim\mu_{n, 1/q}$, and $q$ is arbitrary, it
  suffices to show that
  \begin{equation}\label{eq:var_LIS_p_n}
    \var(\LIS(p_n))\le n-1
  \end{equation}
  and, for all $t>0$,
  \begin{equation}\label{eq:LIS_Gaussian_tail}
    \P(|\LIS(p_n)-\E(\LIS(p_n))|> t\sqrt{n-1})< 2e^{-t^2/2}.
  \end{equation}
  Recall from the definition of the Mallows process that $p_n$ is
  determined by the random variables $(p_i(i))$, $2\le i\le n$, and
  that these random variables are independent. Let us define a
  function $f$ by the relation
  \begin{equation*}
    \LIS(p_n) = f(p_2(2), p_3(3),\ldots, p_n(n)).
  \end{equation*}
  We will show that $f$ has the bounded
  differences property. Precisely, that if $x := (x_2,\ldots, x_n)$ and
  $x' := (x_2',\ldots, x_n')$ satisfy $1\le x_i,
  x_i'\le i$ for all $i$ and $x_i = x_i'$ for all but one
  value of $i$, then
  \begin{equation}\label{eq:f_bdd_diff}
    |f(x) - f(x')| \le 1.
  \end{equation}
  This implies \eqref{eq:var_LIS_p_n} and
  \eqref{eq:LIS_Gaussian_tail} by standard facts. To see
  this, define the martingale $L_i:=\E(\LIS(p_n)\,|\,(p_j(j)), j\le
  i)$ for $1\le i\le n$, where we note that $L_1=\E(\LIS(p_n))$ since $p_1(1)$ is constant.
  Then \eqref{eq:f_bdd_diff} and \cite[Theorem 7.4.1]{AloSpe08} imply that
  $|L_{i+1}-L_i|\le 1$ for all $i$, almost surely. Thus, by the
  martingale property,
  \begin{equation*}
    \var(\LIS(p_n)) = \E(L_n - L_1)^2 = \sum_{i=2}^n \E(L_i -
    L_{i-1})^2 \le n-1.
  \end{equation*}
  The tail bound \eqref{eq:LIS_Gaussian_tail} follows from the
  Bernstein-Hoeffding-Azuma inequality \cite[Theorem 7.2.1]{AloSpe08}.

  Let us now prove \eqref{eq:f_bdd_diff}. Let $x, x'$ be as above and suppose that $x_i = x_i'$
  for all $i\neq i_c$, and $x_{i_c}\neq x_{i_c}'$. By symmetry of $x$ and $x'$, it suffices to show that
  \begin{equation}\label{eq:f_bdd_diff_sym}
    f(x')\ge f(x)-1.
  \end{equation}
  Write $(p_i^x)$, $1\le i\le n$, for the first $n$ permutations in the Mallows process which result when $p_1(1) = 1$ and $p_i(i) = x_i$ for $2\le i\le n$. Similarly let $(p_i^{x'})$
  be the first $n$ permutations which result when $p_i(i)=x_i'$ for $2\le i\le n$.
  Recall that, by definition, $\LIS(p_n^x) = f(x)$ and let $1\le i_1<\cdots<i_{f(x)}\le n$
  be the indices of
  an (arbitrary) longest increasing subsequence in $p_n^x$. That is,
  $(i_1,\ldots, i_{f(x)})$ satisfy
  \begin{equation}\label{eq:p_n_x_lis}
    p_n^x(i_{j+1})>p_n^x(i_j)\quad\text{for $1\le j\le f(x)-1$}.
  \end{equation}
  We will make repeated use of the following two facts which follows directly
  from the definition of the Mallows process: For any $i<j$, if
  $p_j(i)<p_j(j)$ then $p_k(i)<p_k(j)$ for all $k\ge j$. In addition, the values of $(p_m(m))$, $i\le m\le j$, determine whether $p_j(i)<p_j(j)$ (this is a special case of
  Fact~\ref{fact:indep-of-orderings-on-blocks}).
  Let us now consider several cases.
  \begin{enumerate}
    \item If $f(x)=1$ then \eqref{eq:f_bdd_diff_sym} is trivial.
    \item If
  $i_c>i_{f(x)-1}$ it follows from \eqref{eq:p_n_x_lis} that $(i_1,\ldots, i_{f(x)-1})$ form the indices of an increasing
  subsequence in $p_n^{x'}$ and hence $f(x')\ge f(x)-1$, proving
  \eqref{eq:f_bdd_diff_sym}. Similarly, if $i_c<i_2$ it follows from \eqref{eq:p_n_x_lis}
  that $(i_2, \ldots, i_{f(x)})$ form the indices of an increasing
  subsequence in $p_n^{x'}$ and hence $f(x')\ge f(x)-1$, again proving
  \eqref{eq:f_bdd_diff_sym}.
  \item Finally, suppose that $i_2\le i_c\le
  i_{f(x)-1}$ and let $2\le j_c\le f(x)-1$ be equal to the maximal integer
  for which $i_{j_c}\le i_c$. In this case, by the aforementioned facts about the Mallows process, we have that each of $(i_1,\ldots,
  i_{j_c-1})$ and $(i_{j_c+1}, \ldots, i_{f(x)})$ form the indices
  of an increasing subsequence in $p_n^{x'}$. Hence, to prove
  \eqref{eq:f_bdd_diff_sym}, it suffices to prove that
  $p_n^{x'}(i_{j_c+1})>p_n^{x'}(i_{j_c-1})$, which is
  equivalent to
  \begin{equation}\label{eq:LIS_skipping_one_element}
    p_{i_{j_c+1}}^{x'}(i_{j_c+1})>p_{i_{j_c+1}}^{x'}(i_{j_c-1}).
  \end{equation}
  Condition~\eqref{eq:p_n_x_lis} implies that
  \begin{equation}\label{eq:p_i_0_x_j_0}
    p_{i_c}^x(i_{j_c-1})<p_{i_c}^x(i_{j_c}).
  \end{equation}
  Now, \eqref{eq:Mallows_iteration_procedure} implies that in a
  Mallows process $(p_i)$,
  $p_i(j) - p_{i-1}(j)$ can change by at most one when $p_i(i)$
  changes. Thus, we deduce from \eqref{eq:p_i_0_x_j_0} that
  \begin{equation*}
    p_{i_c}^{x'}(i_{j_c-1})\le p_{i_c}^x(i_{j_c}).
  \end{equation*}
  By \eqref{eq:Mallows_iteration_procedure} again, since $x_i=x_i'$ for $i_c<i\le i_{j_c}+1$,
  we conclude that
  \begin{equation*}
    p_{i_{j_c+1}}^{x'}(i_{j_c-1})\le p_{i_{j_c+1}}^x(i_{j_c}) <
    p_{i_{j_c+1}}^x (i_{j_c+1}) = x_{i_{j_c+1}} = x'_{i_{j_c+1}} =
    p_{i_{j_c+1}}^{x'} (i_{j_c+1}),
  \end{equation*}
  proving \eqref{eq:LIS_skipping_one_element} and finishing the
  proof of the proposition.\qedhere
\end{enumerate}

\section{Discussion and open questions}\label{sec:open-ques}
A number of interesting directions remain for further research.
\begin{enumerate}

\item (Variance of the $\LIS$ and limiting distribution). A natural next step is to determine the variance of the longest increasing subsequence and its limiting distribution.
By the work of Baik, Deift and Johansson \cite{BaiDeiJoh99} the
variance is of order $n^{1/3}$ and the limiting distribution is
Tracy-Widom when $q=1$. In the case that $0<q<1$ is fixed we expect
the variance to be of order $n$ and the limiting distribution to be
Gaussian. Establishing these last facts should not be difficult
(Proposition~\ref{prop:var_prop} shows one direction for the
variance). It is less clear what the variance and limiting
distribution should be in the intermediate regime of $q$ though it
may at least seem reasonable that the variance decreases with $q$.


The bounds on the displacement obtained in
Theorem~\ref{thm:displacement} show that in the graphical
representation of a Mallows permutation most points lie in a strip
whose width is proportional to $1/(1-q)$ (see
Figure~\ref{fig:mallows-points}). This suggests a possible
connection between the length of the longest increasing subsequence
of a Mallows permutation and the model of last passage percolation
for random points in a strip. The analogy is not perfect, however,
since the points in the graphical representation of the Mallows
measure are correlated. It is not clear whether, asymptotically,
these correlations have a significant effect on the variance and
limiting distribution (see also Question
\ref{question:limits_graphical_repr} below).


Chatterjee and Dey \cite{ChaDey13} investigated \emph{undirected}
first passage percolation in the rectangle $[0,k]\times [0, h_k]$
and conjectured that the first passage time has variance
$kh_k^{-1/2+o(1)}$ and Gaussian limit distribution when $h_k\ll
k^{2/3}$. They proved that the limiting distribution is indeed
Gaussian when $h_k\ll k^{1/3}$ and gave certain evidence for the
full conjecture (as well as similar results in higher dimensions).

Several authors \cite{BaiSui05,BodMar05,Sui06} have investigated
\emph{directed} first and last passage percolation in the rectangle
$[0,k]\times [0, h_k]$. They have shown that when $1\ll h_k \ll
k^{3/7}$ the passage time converges to the Tracy-Widom distribution,
in contrast to the aforementioned results of \cite{ChaDey13} for
undirected first passage percolation. While directed last passage
percolation is more similar to the longest increasing subsequence
model than undirected first passage percolation, the convergence to
the Tracy-Widom law in this result seems related to the fact that
the rectangle considered is horizontal, unlike our diagonal strip.

Thus an intriguing question is which limit distribution appears for
the length of the longest increasing subsequence in the intermediate
regime of $q$, when $q\to 1$ with $n$ at some rate. Is it a
Tracy-Widom distribution as is the case for $q=1$, or is it the
Gaussian distribution as we expect for fixed $0<q<1$, or some other
possibility? Is it the same throughout the entire intermediate
regime?

%

What is the dependence of the variance on $n$ and $q$? Does it have
the asymptotic form $n^a (1-q)^b$, for some $a,b$, as the
expectation does? Possibly, if there are several regimes for the
limiting distribution then there would also be several regimes for
the values of $a$ and $b$ depending on the precise rate at which $q$
tends to $1$ with $n$.

In Section \ref{sec:LIS-Xi-comparison} we have shown that the longest increasing subsequence is close to a sum of i.i.d. random variables corresponding to
the longest increasing subsequences of disjoint blocks of elements. However, our bounds on the error terms in this approximation do not seem to be strong enough to draw useful conclusions on the distribution or variance of the longest increasing subsequence.

\item (RSK correspondence). In prior work on the distribution of the longest increasing subsequence for the uniform distribution, e.g., \cite{LogShe77,VerKer77,BaiDeiJoh99},
the combinatorial bijection known as the Robinson-Schensted-Knuth
(RSK) correspondence between permutations and Young tableaux has
played an important role. A natural question is to study the measure
induced on Young tableaux by the RSK correspondence applied to
Mallows-distributed permutations.

\item \label{question:limits_graphical_repr}(Limits of graphical representation). Consider the graphical
representation of Mallows-distributed permutations as in
Figure~\ref{fig:mallows-points}. Theorem~\ref{thm:displacement} and
the figure suggest that the empirical distribution of the points in
a square of width $\frac{1}{1-q}$ around the diagonal converges to a
limiting density. What is the form of this density? Starr
\cite{Sta09} has answered this question in the regime where $n(1-q)$
tends to a finite constant.

Additionally, what is the local limit of the points in the graphical
representation (the limit when zooming to a scale in which there is
one point per unit area on average)? Is it a Poisson process or does
it have non-trivial correlations?

%

A related question is to understand the joint distribution of displacements beyond the estimates given in Theorem \ref{thm:displacement}.


\item (Law of large numbers for $\LDS$). It remains to establish a law of large numbers for the longest decreasing
subsequence. Extrapolating from the results of Mueller and Starr
\cite{MueSta11}, we expect that the length of the longest decreasing
subsequence multiplied by $\sqrt{1-q}$ converges in probability to
the constant $\pi$, at least when $n(1-q)\to\infty$ and $(\log
n)^2(1-q)\to 0$. See also Remark~\ref{rem:LDS_limit}.

\item (Expected $\LIS$ for fixed $q$). Fix $0<q<1$ and let $\pi_n$ have the $(n,q)$-Mallows distribution.
Corollary~\ref{cor:induced_permutation_Mallows} implies that
\begin{equation*}
  \E(\LIS(\pi_{n+m})) \le \E(\LIS((\pi_{n+m})_{\{1,\ldots, n\}}) +
  \LIS((\pi_{n+m})_{\{n+1,\ldots, n+m\}})) = \E(\LIS(\pi_n)) +
  \E(\LIS(\pi_m)).
\end{equation*}
Thus, by Fekete's subadditive lemma,
\begin{equation*}
  \lim_{n\to\infty}\frac{\E(\LIS(\pi_n))}{n}= \inf_n \frac{\E(\LIS(\pi_n))}{n} =:
  c(q).
\end{equation*}
It would be interesting to find an explicit expression for $c(q)$.
Proposition~\ref{prop:small_q_LIS_exp_bound} shows that $1-q\le
c(q)\le \frac{1}{1+q}$ for all $0<q<1$, which is rather tight for
small $q$. In addition, Theorem~\ref{thm:lp-convergence-lis} and the
above representation of $c(q)$ as an infimum imply that
\begin{equation*}
  \limsup_{q\uparrow1}\frac{c(q)}{\sqrt{1-q}}\le 1.
\end{equation*}
%
%

\item (Improved large deviation bounds). Our large deviation results are not always sharp. For instance, our bound \eqref{eq:LIS_less_than_L_UB} on the lower
tail of $\LIS(\pi)$ can probably be improved. Deuschel and Zeitouni
\cite{DeuZei99} proved that $\P(\LIS(\pi)<c\sqrt{n})$ is
exponentially small in $n$ for a uniform permutation $\pi\in S_n$.
However, substituting $q=1-4/n$ (which one may expect behaves
similarly to the uniform case) in \eqref{eq:LIS_less_than_L_UB}
yields only that $\P(\LIS(\pi)<c\sqrt{n})$ is at most exponentially
small in $\sqrt{n}$. See also the remark at the end of
Section~\ref{sec:erdos-szekeres}.

\end{enumerate}

\bibliographystyle{plain}
\bibliography{all}

\end{document}